\documentclass{amsart}
\pdfoutput=1
\allowdisplaybreaks
\usepackage{mathpazo}
\usepackage{amsmath,amsfonts,amssymb}
\usepackage{amsrefs}
\usepackage{amsthm}
\usepackage{latexsym,amsmath,amssymb,amsfonts}
\usepackage{rotating}
\usepackage{mathrsfs}
\usepackage{amscd}
\usepackage{hyperref} 
\usepackage{euscript}
\usepackage{hhline}
\usepackage{graphicx,epstopdf}
\usepackage{epsfig}
\usepackage{xcolor}
\usepackage[all,color]{xy}
\usepackage{tikz}
\usetikzlibrary{cd}
\usetikzlibrary{matrix}
\usepackage{spectralsequences}
\usepackage{float}
\usepackage{mathtools} 

\newlength{\fighskip} \fighskip=2pt
\newlength{\figvskip} \figvskip=3pt

\usepackage{hyperref}

\advance\textheight by \topskip
\numberwithin{equation}{section}

\renewcommand{\epsilon}{\varepsilon}

\newcommand{\R}{\mathbb{R} }
\newcommand{\C}{\mathbb{C} }
\renewcommand{\P}{\mathbb{P} }
\newcommand{\Z}{\mathbb{Z} } 
\newcommand{\Q}{\mathbb{Q} }
\newcommand{\U}{{\mathcal{U}}}

\newcommand{\LL}{\mathcal{L} }

\newcommand{\zz}{{\mathbf{z}}}
\newcommand{\ww}{{\mathbf{w}}}
\newcommand{\vv}{{\mathbf{v}}}
\newcommand{\xx}{{\mathbf{x}}}
\newcommand{\yy}{{\mathbf{y}}}
\newcommand{\kk}{{\mathbf{k}}}
\newcommand{\jj}{{\mathbf{j}}}

\newcommand{\xalpha}{{\boldsymbol \alpha}}
\newcommand{\xbeta}{{\boldsymbol \beta}}

\newcommand{\OO}{{\mathcal O}}

\newcommand{\DD}{{\mathbb D}}

\newcommand{\TT}{\boldsymbol T}
\newcommand{\ZZ}{{\boldsymbol Z}}

\newcommand{\hotimes}{\hat\otimes}

\renewcommand{\emptyset}{\varnothing}

\renewcommand{\j}{{\bf{j}}}

\newcommand{\HH}{\mathscr{H} }

\newcommand{\inv}{^{-1}}

\newcommand{\bs}[1]{{\boldsymbol{#1}}}

\newcommand{\LR}[1]{\left ( #1 \right )}
\newcommand{\Lr}[1]{\left [ #1 \right ]}
\newcommand{\lr}[1]{\left\{ #1\right \}}
\newcommand{\lR}[1]{\left < #1\right >}
\newcommand{\set}[2]{{\left\{ \left. #1 \,\right|\, #2 \right\}}}    

\providecommand{\abs}[1]{\left\lvert#1\right\rvert}

\newcommand{\ali}[1]{$$\begin{aligned} #1 \end{aligned}$$}

\newcommand{\lra}{ \longrightarrow}
\newcommand{\lmt}{ \longmapsto}
\newcommand{\hlra}{{\,\xhookrightarrow{\quad\,}\,}} 
\newcommand{\slra}{{\overset{\simeq}{\lra}}} 
\newcommand{\islra}{{\,\xhookrightarrow{\,\,\simeq\,\,}\,}} 

\newcommand{\Omegaz}{\Omega_{\zz}}

\newcommand{\VAvacuum}{{\left |0\right>}}  

\newcommand{\Kdisk}{\mathcal K}  

\newcommand{\sT}{\sigma(T)}

\newcommand{\bp}{{\Bar{\partial}}}
\renewcommand{\d}{{\mathrm{d}+\bp}}
\newcommand{\bd}{{(\d)}}
\newcommand{\bfd}{{\mathbf d}}

\newcommand{\rSquare}{|\xx|^2 + 2|\zz|^2} 

\newcommand{\ZZi}{z^1,\cdots,z^n}
\newcommand{\PPi}{\partial_1,\cdots, \partial_n}

\newcommand{\setn}{\lr{1,\cdots, n}}
\newcommand{\punctured}{{\setminus\lr{0}}}

\newcommand{\Ocplx}{\mathcal{O}^{cplx}} 
\newcommand{\A}{A^{\bullet}}  
\newcommand{\Ak}{A^{\bullet}_{(k)}}  
\newcommand{\Adef}{\Aconk}
\newcommand{\Hdef}{\Hconk}
\newcommand{\Aconk}{A^{\bullet}_{con,(k)} } 
\newcommand{\Hconk}{H^{\bullet}_{con,(k)} } 

\newcommand{\Xdef}{X^{def}} 
\newcommand{\Odef}{\OO^{def}} 
\newcommand{\AOdef}{A\OO^{def,\bullet} } 
\newcommand{\Ocon}{\OO^{con}} 
\newcommand{\OconU}{\OO^{con}} 

\newcommand{\RGamma}{R\Gamma(\maka,\Ocplx_{\maka})} 

\newcommand{\aka}{{\R^{n'}\times \C^n}} 
\newcommand{\paka}{{\R^{n'}\times \C^n \punctured}} 
\newcommand{\maka}{{\Conf_m(\R^{n'}\times \C^n)}} 
\newcommand{\cmaka}{{\cConf_m(\R^{n'}\times \C^n)}} 

\newcommand{\GD}{\widehat{\mathcal{D}}}
\newcommand{\cD}{{\mathcal{D}}}
\newcommand{\cGD}{{\mathcal{D}}_{conn}}
\newcommand{\NAI}{{\mathcal{N}}}
\newcommand{\ExtVert}{A}

\newcommand{\VertexDiagram}{{\Gamma_{W}(a)}}
\newcommand{\chorddiagram}{{\Gamma_{W}(a,b)}} 
\newcommand{\ConfSpaceCohomology}{{\Hconk(\Conf_A(\aka))}} 
\newcommand{\DefinableChainsA}{{\Aconk(\Conf_A(\aka))}} 

\newcommand{\GlobalRes}{\oint} 

\newcommand{\GIK}{\widehat{\operatorname{I}}}

\newcommand{\calU}{{\mathcal U}}

\def\Xint#1{\mathchoice
	{\XXint\displaystyle\textstyle{#1}}%
	{\XXint\textstyle\scriptstyle{#1}}%
	{\XXint\scriptstyle\scriptscriptstyle{#1}}%
	{\XXint\scriptscriptstyle\scriptscriptstyle{#1}}%
	\!\int}
\def\XXint#1#2#3{{\setbox0=\hbox{$#1{#2#3}{\int}$}
		\vcenter{\hbox{$#2#3$}}\kern-.5\wd0}}

\def\dashint{\Xint-}

\DeclareMathOperator{\BL}{BL}

\DeclareMathOperator{\Conf}{Conf}
\newcommand{\cConf}{{\overline{\Conf}}}

\DeclareMathOperator{\dVol}{dVol}
\DeclareMathOperator{\End}{End}

\newcommand{\Ho}{\mathscr{H}}
\newcommand{\Hol}{\OO^{hol}}
 
\DeclareMathOperator{\Hom}{Hom}
\DeclareMathOperator{\I}{I}
\DeclareMathOperator{\Id}{Id}
\renewcommand{\Im}{\mathrm{Im} \;\!}
\DeclareMathOperator{\Ker}{Ker}

\DeclareMathOperator{\Maps}{Maps}

\DeclareMathOperator{\Obs}{Obs}
\DeclareMathOperator{\Opens}{Opens}

\DeclareMathOperator{\pr}{pr}
\DeclareMathOperator{\PV}{PV}

\renewcommand{\Re}{\text{Re}}
\DeclareMathOperator{\Res}{Res}

\DeclareMathOperator{\sign}{sign}
\DeclareMathOperator{\sgn}{sgn}

\DeclareMathOperator{\Spec}{Spec}

\DeclareMathOperator{\Vol}{Vol}

\theoremstyle{plain}

\newtheorem{thm}{Theorem}[section]
\newtheorem{lemma}[thm]{Lemma}
\newtheorem{lem}[thm]{Lemma}

\newtheorem{prop}[thm]{Proposition}

\newtheorem{cor}[thm]{Corollary}

\theoremstyle{definition}

\newtheorem{defn}[thm]{Definition}

\newtheorem{eg}[thm]{Example}
\newtheorem{notation}[thm]{Notation}

\theoremstyle{remark}

\newtheorem{rem}[thm]{Remark}

\allowdisplaybreaks[4]  

\title{\textbf On the Formality of Configuration Spaces of $\aka$}

\author{Si Li, Peng Yang, and Jiawei Zhou}

\address{
S.~ Li:
Yau Mathematical Sciences Center, Tsinghua University, Beijing, China
}
\email{sili@mail.tsinghua.edu.cn}

\address{
P.~ Yang: 
Beijing Institute of Mathematical Sciences and Applications, Beijing, China
}
\email{yangpeng@bimsa.cn}

\address{
J.~ Zhou:
Department of Mathematics, Nanchang University, Nanchang, Jiangxi, China
}
\email{jiaweizhou90@ncu.edu.cn}

\begin{document}

\begin{abstract}
This paper presents a complete classification of the formality of configuration spaces of $\aka$. 
We define a constructible de Rham-Dolbeault cohomology theory which provides a constructible CDGA (commutative differential graded algebra) model of $\maka$. 
For $(n'=0,n\ge2)$ or $(n'=1,n\ge1)$,  the CDGAs are non-formal. 
For $n'\ge2,n\ge1$, we establish an explicit quasi-isomorphism between the constructible CDGA and its cohomology by using a diagrammatic CDGA of admissible diagrams and a regularized configuration space integral, which leads to the formality. 
As an application, we show that the local operator algebra of a topological-holomorphic field theory on $\aka$ ($n'\ge2,n\ge1$) is homotopically equivalent to a higher dimensional analog of vertex algebras.
\end{abstract}

\maketitle
  
\tableofcontents


\section{Introduction}
\label{sec:introduction}

A CDGA model encodes the rational homotopy type of a space. The space is called formal if this CDGA is connected by quasi-isomorphisms to its cohomology algebra equipped with the zero differential (see Definition \ref{defn:formality}).  In this paper, we consider the configuration spaces
$$
\maka  = \set{(\ZZ_1,\cdots,\ZZ_m)\in (\aka)^m}{\ZZ_i \ne \ZZ_j \text{ for } i\ne j}, \qquad n\ge1
$$
equipped with two CDGAs: the CDGA of smooth $\bd$-forms $\A(\maka)$, and its constructible counterpart $\Aconk(\maka)$ in which every differential form has constructible coefficients as defined in \cite{int-closed} (see Section \ref{sec:constructible-chains}). Such CDGAs encode topological data in the $\R^{n'}$-directions and holomorphic data in the $\C^n$-directions. For $m\ge 2$, we present a complete classification of their formality summarized in the following table: 
\begin{center}
\begin{tabular}{|c|c|c|c|}
    \hline
      & $n=0$ & $n=1$ & $n>1$ \\
    \hline
    $n'=0$ & $\emptyset$ & formal & non-formal \\
    \hline
    $n'=1$ & formal & non-formal & non-formal \\
    \hline
    $n'>1$ & formal & formal & formal \\
    \hline
\end{tabular}
\end{center}  
While the configuration space of $\R^{n'}$ is always formal, formality is not always guaranteed when complex dimensions are involved.

\subsection{Formality and the Little Disks Operad} 

The little $N$-disks operad, introduced by Boardman, Vogt and May, plays a fundamental role in encoding the algebraic structures of $N$-fold loop spaces. 
The celebrated formality theorem for this operad, which asserts a weak equivalence between the singular chains on the little $N$-disks operad and its homology, was originally proved by Kontsevich \cite{Kontsevich-Operads} and subsequently developed in the setting of CDGA cooperads \cite{Formality-little-disks}. An alternative proof based on Drinfeld associators was given by Tamarkin \cite{Tamarkin-formality}. 
For $N=2$, the formality directly implies the deformation quantization of any Poisson manifold \cite{Kontsevich-DQ}. 

The proof of formality of the little $N$-disks operad depends crucially on the formality of the configuration spaces $\Conf_m(\R^N)$, which was established using the Kontsevich configuration space integral \cite{Kontsevich-Operads}
$$
\I: \cD_N(m) \lra \Omega_{PA}(C_N[m])
$$ 
between a CDGA of admissible diagrams $\cD_N(m)$ and the piecewise semialgebraic de Rham complex on the compactification $C_N[m]$ of $\Conf_m(\R^N)$.

The systematic study of cohomology of configuration spaces began with Arnold's work \cite{Arnold-conf-cohomology} where he computed the singular cohomology ring of $\Conf_m(\R^2)$ and found the celebrated Arnold relation. 
The cohomology ring of $C_N[m]$ is generated by the cohomology classes $\theta_{ab}$, with relations $\theta_{ab}=(-1)^N \theta_{ba}$, $\theta_{ab}^2=0$,  and the Arnold relation $\theta_{ab}\theta_{bc}+\theta_{bc}\theta_{ca}+\theta_{ca}\theta_{ab}=0$. Here $\theta_{ab}$ is the pullback of the cohomology class of the volume form of $S^{N-1}$ via the projection
\ali{
\pr_{ab}: C_N[m] \lra C_N[2] \cong S^{N-1}
} 
that forgets $m-2$ points. 
We assign the differential form $\theta_{ab}$ to every edge $e:a\to b$ in a diagram.  The configuration space integral $\I$ integrates over all but $m$ of the points. The differential of $\cD_N(m)$ is contraction of edges, which reflects Stokes theorem on $C_N[m]$.

In this paper, we generalize the formality of $\Conf_m(\R^N)$ to $\maka$ for $n'\ge2$. The diagonal maps 
\ali{
\phi_{ij}: \quad \maka &\lra \paka \\
  (\ZZ_1,\cdots,\ZZ_m)       &\lmt \ZZ_i-\ZZ_j
}
take values in $\paka$, so we begin by studying the cohomology of $\paka$  
\ali{
H^\bullet(\paka)&=H^\bullet(\A(\paka),\d), \\
\Hdef(\paka)&=H^\bullet(\Aconk(\paka),\d).
} 
See Section \ref{sec:CDGAs} for definitions. 
Let $\omega$ be a smooth $n'+n-1$ form on $\paka$ which solves the distribution equation (see Section \ref{sec:local-coho})
$$\bd \omega=\delta_0  \prod_{i=1}^{n'}dx^i\prod_{j=1}^{n} d\bar z^j,
$$
where $\delta_0$ is the Dirac delta distribution at $0$. 
Then we have  

\begin{prop}[see Proposition \ref{prop:coh-local}] 
For $n'+n\ge2$, the cohomology of the CDGAs $\A(\paka)$ and $\Aconk(\paka)$ are as follows:
\begin{equation*} 
\Hdef(\paka)  = 
\begin{cases}
\C[\ZZi ]  & \bullet=0 \\
\C[\PPi ]\omega & \bullet=n'+n-1 \\
0 & \text{else}
\end{cases} 
\end{equation*}  
and 
$$
H^\bullet(\paka)=
\begin{cases}
\OO^{hol}(\C^n)  & \bullet=0 \\
\overline{\C[\PPi ]\omega} & \bullet=n'+n-1 \\
0 & \text{else}
\end{cases}.
$$
\end{prop}
The above two equations identify both $\Hdef(\paka)$ and $H^\bullet(\paka)$ as subspaces of $\A(\paka)$. Under this identification, the latter is the completion of the former with respect to the nuclear topology on smooth differential forms.

The non-formal cases are classified as follows. 
\begin{thm}[see Section \ref{02}-\ref{1>1}]
For $(n'=0,n\ge2)$ or $(n'=1,n\ge1)$, the CDGAs $\A(\maka)$ and $\Aconk(\maka)$ are non-formal for all $m\ge2$.
\end{thm}


We compute the cohomology of $\A(\maka)$ and $\Aconk(\maka)$. These cohomology rings, viewed as CDGAs with zero differential, are generated by the degree-zero cohomology and by diagonal classes which are holomorphic derivatives of $\omega_{ij}=\phi_{ij}^*\omega$. These generators satisfy a generalized version of the Arnold relation.

\begin{thm}[see Theorem \ref{thm:coho-ring-conf} and Corollary \ref{cor:Hdef-group-structure}]
For $n'+n\ge2$, the cohomology ring $\Hdef(\maka)$ is generated  by 
$$
\C[z_j^1,\cdots,z_j^n]\oplus \oplus_{i<j} {\C[\partial_{z^1_j},\cdots, \partial_{z^n_j}]\omega_{ij}},
$$
subject to the following relations:  
\begin{itemize}
\item $f(\partial_{\zz_j})\omega_{ij}\cdot g(\partial_{\zz_j})\omega_{ij}=0$, for $f(\partial_{\zz_j})\omega_{ij}, g(\partial_{\zz_j})\omega_{ij}\in {\C[\partial_{z^1_j},\cdots, \partial_{z^n_j}]\omega_{ij}}$,
\item $
\zz_{ij}^{\xalpha}\partial_{\zz_j}^{\xbeta}\omega_{ij}=
\begin{cases}
(-1)^{\alpha_1+\cdots +\alpha_n} \prod_{i=1}^n \frac{\beta_i!}{(\beta_i-\alpha_i)!} \, \partial_{\zz_j}^{\xbeta-\xalpha}\omega_{ij} & \alpha_i\le \beta_i \text{ for all } i \\
0 & \alpha_i>\beta_i \text{ for some } i \\
\end{cases},
$
\item the Arnold relation:
$$g(\partial_{\zz_j})(f(\partial_{\zz_i})\omega_{ij}h(\partial_{\zz_k})\omega_{jk})+h(\partial_{\zz_k})(g(\partial_{\zz_j})\omega_{jk}f(\partial_{\zz_i})\omega_{ki})+f(\partial_{\zz_i})(h(\partial_{\zz_k})\omega_{ki}g(\partial_{\zz_j})\omega_{ij})=0.$$
\end{itemize}
Moreover, when identified with subspaces of $\A(\maka)$, $H^\bullet(\maka)$ is the closure of \\ $\Hconk(\maka)$ with respect to the nuclear topology. 
\end{thm}

We define the diagram complex $\cD(m)$ to encode the homotopy information of $\maka$. The cohomology ring of $\maka$ has two types of generators, degree $0$ terms  and degree $n'+n-1$ terms, so we assign the following weights for every diagram $\Gamma$ (see Definition \ref{def:diagram}):
\begin{itemize}
\item a weight $W_\Gamma^V(v)$ for each vertex $v$, which is a degree-zero cohomology class on $\aka$,
\item a weight $W_\Gamma^E(e)$ for each edge $e$, which is a degree $n'+n-1$ cohomology class on $\Conf_2(\aka)$.
\end{itemize}
The differential on $\cD(m)$ is given by contraction of edges.
Using the notation in Section \ref{sec:formality-conf}, 
the configuration space integral associated with $\Gamma$ is defined by  
$$
\dashint \prod_{v\in I_\Gamma} d\zz_v \,  \prod_{v\in V_\Gamma} \theta_v^* (W^V_\Gamma(v)) \prod_{e\in E_\Gamma} \theta_e^* (W^E_\Gamma(e)\omega),
$$
where $\theta_v^*$ and $\theta_e^*$ pullback the weights to $\maka$, and $\dashint$ is the regularized integration along compactified fibers. A residue formula identifies the boundary contribution associated with a two-point collision with the edge-contraction differential, which yields a quasi-isomorphism of CDGAs
$$
\bar \I: \cD(m) \slra \A(\maka).
$$
After eliminating diagrams with internal vertices, we get a quasi-isomorphism of CDGAs 
$$
\cD(m) \slra H^\bullet(\maka).
$$
The formality theorem follows from the above quasi-isomorphisms.
\begin{thm}[see Theorem \ref{thm:formality} and Remark \ref{rem:formality}]
When $n'\ge2, n\ge1$, the CDGAs $\A(\maka)$ and $\Adef(\maka)$ are formal.
\end{thm}
  
We defer the study of the (co)operad structures for $\maka$ and the CDGAs of differential forms on it to future work.

\subsection{Tame Geometry and Constructible Differential Forms}
The concept of tame geometry originates in Grothendieck's Esquisse d'un Programme \cite{Esquisse}, which advocated a fundamental reformulation of point-set topology to preserve the intrinsic tameness of geometric intuition. 
O-minimal geometry is a framework that fulfills Grothendieck’s vision and has by now evolved from its roots in model theory to an indispensable tool across differential geometry, Hodge theory, and arithmetic geometry. 

The simplest example of an o-minimal structure is that of semi-algebraic sets, which appears prominently in Kontsevich’s proof of the formality of the little disks operad \cite{Kontsevich-Operads}. Because the natural projection between compactified configuration spaces is not a smooth fiber bundle (See \cite{Formality-little-disks}, Example 5.22), standard differential forms cannot be integrated along the fibers in the usual sense. Kontsevich addressed this by treating the projection as a semi-algebraic bundle and replacing smooth forms by piecewise semi-algebraic (PA) forms, thereby establishing a robust framework for fiber integration on semi-algebraic stratifications. The story was developed in detail in \cites{semi-algebraic,Formality-little-disks}. 

There are several versions of de Rham cohomology theories in the o-minimal setting, including a constructible one \cites{o-minimal-de-rham,O-minimal-deRham}. In this paper, we define a novel version of constructible $\bd$-cohomology, which applies simultaneously to the de Rham and Dolbeault settings. 
Our motivation for defining the constructible $\bd$-cohomology is to encode the algebraic structure of operator product expansions (OPEs) in topological-holomorphic field theories. As asserted in \cite{chiral-algebra}, the space of chiral operations for a $\mathcal D$-module $\mathcal{A}$ on an algebraic curve $X$ is encoded in the $\mathcal D_{X^n}$-module map
$$
\Hom_{\mathcal D_{X^n}} \LR{j_*j^*\mathcal{A}^{\boxtimes n},\Delta_*\mathcal{A}}
$$
where $j:\Conf_n(X) \to X^n$ is the open immersion and $\Delta:X\to X^n$ is the diagonal embedding. A chiral algebra is a Lie algebra object in the operad of chiral operations. The construction is generalized to $X=\C^n$ using a Jouanolou model as a CDGA model of the space of derived global sections of the structure sheaf of $X$ \cite{Higher-CA-Jouanolou}. 

To describe the chiral operations of topological-holomorphic field theories, the appropriate structure sheaf is the sheaf $\Ocon$ of functions that are locally constant in the real directions and regular in the complex directions (see Definition \ref{def:con-structure-sheaf}), whose CDGA of derived global sections can be computed by the constructible Dolbeault complex defined in Section \ref{sec:constructible-chains}.

\begin{thm}(see Theorem \ref{thm:cons-resolution})
Let $U\subset \R^p\times \C^q$ be an open subvariety. Then $\Aconk(U)$ computes the CDGA of derived global sections of the sheaf $\OconU$  
$$
\Aconk(U) \cong R\Gamma(U,\Ocon).
$$
\end{thm}
In particular,  
$$
H^\bullet(\Aconk(\aka),\d) \cong \C[z^1,\cdots,z^n]
$$
is the polynomial ring on $\C^n$. 
Thus, $\Aconk$ is a suitable CDGA to describe the collision geometry of topological-holomorphic field theories. 
The space of constructible differential forms is sufficiently rigid to support algebraic and diagrammatic arguments, while remaining large enough to contain the generalized Bochner-Martinelli kernels and the configuration space integrals used throughout the paper.
A full development of the constructible $\bd$-cohomology theory will be given in a subsequent paper. 

\subsection{From Formality to Vertex Algebras}

The motivation for studying the formality of these mixed configuration spaces is rooted in quantum field theory. In the formalism of factorization algebras \cite{kevin-owen}, a field theory on a spacetime manifold $M$ assigns a chain complex $Obs(U)$ of observables to each open subset $U\subset M$. For pairwise disjoint open subsets $U_1\sqcup \cdots \sqcup U_m \subset U$, there is a factorization product
$$
\Obs(U_1) \otimes \cdots \otimes \Obs(U_m) \lra \Obs(U)
$$
which encodes how observables on separated regions combine, or in physical terms, the operator product expansion.  
Let $U_i$ be small balls with centers $\ZZ_i$ and radii $r_i$. In the limit where every $r_i\to 0$, every $U_i$ shrinks to the point $\ZZ_i$, so the asymptotic behavior of the factorization product recovers the OPE of point operators on these $\ZZ_i$. In this limit, we get the configuration space of these $Z_i$, and the factorization product is governed by the derived global sections of the structure sheaf on the configuration space. Different CDGA models of the derived global sections yield homotopically equivalent descriptions of the OPE.

In particular, when the configuration space is formal, the OPE is completely encoded homotopically by the cohomology ring. For example, chiral conformal field theories on $\C$ lead to vertex algebras, and the local observables of a topological field theory on $\R^N$ are encoded in a $P_N$-algebra, that is, a commutative algebra together with a Poisson bracket of degree $1-N$ \cite{Secondary-Products}. 

For $(n'=0,n\ge2)$ or $(n'=1,n\ge1)$, the CDGAs $\A(\maka)$ and $\Aconk(\maka)$ are non-formal. 
On $\C^n$ and $\R\times \C$, higher-dimensional analogs of vertex algebras were proposed, using cohomology of certain CDGAs associated with configuration spaces \cites{Cohomological-VA,raviolo-vertex-algebras}. 
Because the associated CDGAs of configuration spaces of $\C^n$ and $\R\times \C$ are not formal, we expect to find additional homotopy information when constructing them. 
On the other hand, higher-dimensional chiral algebras associated to $\C^n$ were studied in \cite{Higher-CA-Jouanolou} using a Jouanolou model $J^m_n$ of derived global sections of the structure sheaf on $\Conf_m(\C^n)$. There is a natural inclusion of CDGAs
$$
J^m_n \slra \Aconk(\Conf_m(\C^n))
$$
which is a quasi-isomorphism. The non-formality is reflected by infinitely many higher chiral operations in the unit chiral algebra.

Following \cite{Higher-CA-Jouanolou}, we can define a higher-dimensional chiral algebra on $\aka$ where the space of derived global sections is given by $\Aconk(\maka)$. 
For $n'\ge2,n\ge1$, the formality of \\ $\Aconk(\maka)$ as a $\mathcal D_{\C^{mn}}$-module implies that the associated higher-dimensional chiral algebra is homotopically entirely determined by chiral operations built using $\Hconk(\maka)$. The collection of these chiral operations forms a higher-dimensional analog of vertex algebras, which is defined in Section \ref{sec:VA}. We plan to investigate its detailed properties in a subsequent work.

\subsection{Outline of the Paper}
\begin{itemize}
\item Section \ref{sec:CDGA-homotopy}: We review the homotopy theory of CDGAs, including the notion of formality.
\item Section \ref{sec:CDGAs}: 
We introduce the CDGA of smooth $\bd$-forms
$$
\A(\maka)
$$
and the CDGA of constructible  $\bd$-forms
$$
\Aconk(\maka) 
$$
\item Section \ref{sec:2-point-CDGA}: We treat the two-point configuration spaces 
$$
\Conf_2(\aka)\cong \LR{\aka} \times \LR{\paka}.
$$
We compute the cohomology ring of $\A(\paka)$ and $\Aconk(\paka)$, construct a mixed Jouanolou model for $\paka$, and show that the CDGAs are formal except when $(n'=0,n\ge2)$ or $(n'=1,n\ge1)$.
\item Section \ref{sec:cohomology-conf}: We compute the cohomology ring of $\maka$.
\item Section \ref{sec:Regularized-Integrals}: We use the projectivization of $\paka$ to form a Fulton-MacPherson-type compactification of $\maka$. 
We study the asymptotic behavior of the generalized Bochner-Martinelli kernel and develop the regularized integration and residue calculus.
\item Section \ref{sec:formality-conf}: We define the CDGA $\cD(A)$ of (weighted) admissible  diagrams. For $n'\ge2,n\ge1$, we construct quasi-isomorphisms 
$$
\DefinableChainsA \longleftarrow \cD(A) \lra \ConfSpaceCohomology
$$
thereby proving that the CDGA $\Aconk(\maka)$ is formal. The formality result also holds for $\A(\maka)$.
\item Section \ref{sec:VA}: We introduce the notion of $(n',n)$-vertex algebras as an application of the formality theorem to topological-holomorphic field theories on $\aka$, for $n'\ge2,n\ge1$.
\end{itemize}

\subsection{Acknowledgments}
Part of this work is based on the second author's PhD thesis. 
The first and second authors would like to thank Minghao Wang and Zhengping Gui for helpful discussions and insightful comments. S.~L. was supported by NSFC No.12571068. 
J.~Z. was supported by Natural Science Foundation of Jiangxi, China (No. 20262BAC240204).

\subsection{Notations} 
\begin{itemize}
\item A point in $\aka$ is denoted by $\ZZ=(\xx,\zz)=(x^1,\cdots,x^{n'},z^1,\cdots,z^n)$. We use subscripts to indicate different points. 
\item The configuration space of $m$ points in $\aka$ is
$$
\maka  = \set{(\ZZ_1,\cdots,\ZZ_m)\in (\aka)^m}{\ZZ_i \ne \ZZ_j \text{ for } i\ne j}.
$$
\item $\hotimes$ denotes the completed tensor product of nuclear spaces. 
\item A hat symbol $\,\,\widehat{\, }\,\,$ over a component of a vector indicates that it is omitted.
\end{itemize}

\section{Homotopy Theory of CDGAs}\label{sec:CDGA-homotopy}    
In this section, we review the homotopy theory of CDGAs over $\C$ and refer the reader to  \cites{Rational-homotopy-theory,algebraic-operads}.

\begin{defn}
A {CDGA} is a graded $\mathbb{C}$-vector space $A = \bigoplus_{p\in\Z} A^p$ equipped with a linear map $d: A^p \to A^{p+1}$ such that $d^2 = 0$, and a multiplication satisfying:
\begin{enumerate}
    \item {Graded commutativity:} For any homogeneous elements $x, y \in A$, $xy = (-1)^{|x||y|}yx$.
    \item {Leibniz rule:} $d(xy) = (dx)y + (-1)^{|x|}x(dy)$.
\end{enumerate}
\end{defn}

Given a CDGA $(A, d)$, the graded vector space $H^\bullet(A) = \ker d / \text{im } d$ naturally inherits the structure of a CDGA equipped with the induced multiplication and zero differential.  

\begin{defn}
A CDGA morphism $f: A \to B$ is a quasi-isomorphism (or a weak equivalence) if it induces an isomorphism  $H^\bullet(f): H^\bullet(A) \xrightarrow{\simeq} H^\bullet(B)$ on cohomology.  
\end{defn}

\begin{defn}\label{defn:formality}
A CDGA $A$ is formal if it is connected to its cohomology $H^\bullet(A)$ (viewed as a CDGA with zero differential) by a zigzag of  quasi-isomorphisms

\begin{center}
\begin{tikzcd}
  & A_1 \arrow[ld, "\simeq"'] \arrow[rd, "\simeq"] &     & {} \arrow[ld, "\simeq"'] & \cdots & {} \arrow[rd, "\simeq"] &        & A_{2k+1} \arrow[ld, "\simeq"'] \arrow[rd, "\simeq"] &                    \\
A &                                                               & A_2 &                          &        &                         & A_{2k} &                                                     & {(H^\bullet(A),0)}
\end{tikzcd}.
\end{center}

\end{defn}

To study the algebraic data of a CDGA $(A, d)$ encoded in its cohomology, we utilize the Homotopy Transfer Theorem. This is achieved via a strong homotopy retraction (SDR) datum
\[
\begin{array}{c}
\begin{tikzcd}
h \circlearrowleft (A, d) \arrow[r, "p", shift left] & (H^\bullet(A), 0) \arrow[l, "i", shift left]
\end{tikzcd}
\end{array}
\]
consisting of chain maps $i$ and $p$, and a degree $-1$ homotopy $h: A^\bullet \to A^{\bullet-1}$ such that:
\begin{enumerate}
    \item $p \circ i = \text{id}_{H^\bullet(A)}$;
    \item $i \circ p - \text{id}_A = d h + h d$;
    \item $h^2 = 0$, $h \circ i = 0$, and $p \circ h = 0$.
\end{enumerate}

\begin{thm}[Homotopy Transfer]\label{thm:Homotopy-Transfer}
Given such an SDR datum, there exists an $A_\infty$-algebra structure $\{m_n\}_{n \ge 1}$ on $H^\bullet(A)$ uniquely determined by the CDGA structure of $A$ and the maps $(i, p, h)$. In particular:
\begin{itemize}
    \item $m_1 = 0$;
    \item $m_2(a, b) = p(i(a) \cdot i(b))$, which recovers the standard multiplication on $H^\bullet(A)$;
    \item For $n \ge 3$, the operations $m_n$ are defined via summation over planar rooted trees, where edges are decorated by the homotopy $h$ and vertices by the multiplication in $A$.
\end{itemize} 
Moreover, there exists an $A_\infty$ quasi-isomorphism between $H^*(A)$ and $A$.
\end{thm}
 
Massey products provide a concrete way to detect nontrivial $A_\infty$-structures.

\begin{defn}[Triple Massey product]
Let $\alpha_1, \alpha_2, \alpha_3 \in H^\bullet(A)$ be cohomology classes, represented by cocycles $a_1, a_2, a_3 \in A$, respectively. Assume that $\alpha_1 \alpha_2 = 0$ and $\alpha_2 \alpha_3 = 0$ in $H^\bullet(A)$. 

The triple Massey product $\langle \alpha_1, \alpha_2, \alpha_3 \rangle$ is defined as the coset
\[
\langle \alpha_1, \alpha_2, \alpha_3 \rangle = 
[a_1 a_{23} - (-1)^{|a_1|}a_{12} a_3] \quad \in \quad 
H^\bullet(A) \big/ (\alpha_1 \cdot H^\bullet(A) + H^\bullet(A) \cdot \alpha_3),
\]
where $a_{12}, a_{23} \in A$ are any cochains satisfying $d(a_{12}) = a_1 a_2$ and $d(a_{23}) = a_2 a_3$. The coset is independent of the choices of $a_{12}$ and $a_{23}$. 
\end{defn}

\begin{prop} \label{Massey-nonformal}
If $A$ is formal, then the $A_\infty$-structure on $H^\bullet(A)$ can be chosen such that $m_n = 0$ for all $n \ge 3$. In particular,  all Massey products on $H^\bullet(A)$ vanish. 
\end{prop}

\section{CDGAs Associated to Configuration Spaces}\label{sec:CDGAs}
In this section, we consider CDGAs of chains associated to the configuration space of $m$ points in $\aka$
$$
\maka  = \set{(\ZZ_1,\cdots,\ZZ_m)\in (\aka)^m}{\ZZ_i \ne \ZZ_j \text{ for } i\ne j}
$$
where we denote a point in $\aka$ by $\ZZ=(\xx,\zz)=(x^1,\cdots,x^{n'},z^1,\cdots,z^n)$.


The plan of this section is as follows.
\begin{itemize}
\item \ref{sec:analytic-chains}: We define $\A(\maka)$, the CDGA of $\bd$-forms on $\maka$.
\item \ref{sec:constructible-chains}: We define $\Adef(\maka)$, the CDGA of constructible $\bd$-forms on $\maka$. 

\end{itemize}

\subsection{The CDGA of \texorpdfstring{$\bd$}{(d+∂)}-forms}\label{sec:analytic-chains} 

Let $U \subset \R^p\times \C^q$ be an open submanifold. 

\begin{defn}
Define  
$$
A^s(U) = \set{ \sum_{u+v=s}\sum_{|I|=u,|J|=v} f_{IJ} dx^{I} d\bar z^{J} }{f_{IJ} \text{ are smooth functions on $U$}}
$$
and  
$$
\A(U):=\bigoplus_{s=0}^{p+q} A^s(U).
$$
\end{defn}
$\A(U)$ is naturally a CDGA when equipped the wedge product and the differential $\d$, where $\mathrm{d}$ is the de Rham differential on $\R^p$ and $\bp$ is the Dolbeault differential on $\C^q$. Denote its cohomology by
$$
H^\bullet(U) = H^\bullet(\A(U),\d).
$$

\begin{lem}[The $\bd$-Poincar\'e Lemma]\label{lem:Ck-Poincare}
Let $U$ be an open polydisk in $\R^p\times \C^q$, then the positive-degree cohomology of $(\A (U),\d)$ vanishes. 
\end{lem}

\begin{proof}[Sketch of proof]
Let $\alpha\in \A (U)$. 
Following Grothendieck's  proof of the $\bp$-Poincar\'e Lemma 
(see for example \cite{Complex-geometry}, Proposition 1.3.8), we can find a $\bd$-exact form $\beta\in \A (U)$ such that $\alpha-\beta$  contains no $d\bar z^j$ terms, i.e., $\alpha-\beta$ is holomorphic in any $z^j$. Then, applying a homotopy operator associated with the contraction $\gamma_t:(\xx,\zz)\mapsto (t\xx,\zz)$, $0\le t\le 1$, yields its $\mathrm{d}$-inverse.  
\end{proof} 

As $U$ varies, these spaces form a sheaf $\A$ of CDGAs on $\R^p\times \C^q$. Let $\Hol$ be the sheaf of smooth $\bd$-closed functions.  
By Lemma \ref{lem:Ck-Poincare}, $\A$ is a fine resolution of $\Hol$.

For $U=\maka$, we have a CDGA 
$$
\A(\maka).
$$

\subsection{The CDGA of Constructible \texorpdfstring{$\bd$}{(d+∂)}-forms}\label{sec:constructible-chains}

\subsubsection{O-minimal Structures}\label{sec:o-minimal}
O-minimal structures are generalizations of semialgebraic sets while retaining strong finiteness properties. The standard reference is \cite{tame-o-minimal}.

\begin{defn} 
A structure on the real field $(\mathbb{R}, +, \cdot, <)$ is a sequence $\mathcal{S} = (S_n)_{n \in \mathbb{N}}$, where each $S_n$ is a collection of subsets of $\R^n$,  satisfying the following: 
\begin{enumerate}
    \item $S_n$ contains all algebraic subsets, i.e., sets defined by polynomial equalities and inequalities.
    \item $S_n$ is closed under finite union, finite intersection, and complement.
    \item If $A \in S_m$ and $B \in S_n$, then $A \times B \in S_{m+n}$.
    \item If $p:\R^{n+1}\to \R^n$ is a linear projection and $A \in S_{n+1}$, then $p(A)\in S_n$.
\end{enumerate}
The elements of $S_n$ are called  $\mathcal S$-definable subsets of $\R^n$. A map (or function) $f:A\to B$ between $\mathcal S$-definable sets is called definable if its graph is $\mathcal S$-definable.
\end{defn}

\begin{rem}
It is worth distinguishing the definition of a definable map $f$ from the intuitive notion borrowed from topology or measure theory, where maps are defined via pullbacks (e.g., $f^{-1}(T)$ being open or measurable for relevant $T$). While definability of the graph implies that $f^{-1}(T)$ is definable for every definable set $T$, the converse does not hold in general. 
\end{rem}

\begin{defn} 
A structure $\mathcal{S}$ is said to be \textit{o-minimal} 
if $S_1$ consists precisely of finite unions of points and open intervals.
\end{defn}

\begin{eg}
The smallest o-minimal structure on $(\mathbb{R}, +, \cdot, <)$ is the structure $\R_{alg}$ of semialgebraic sets, which was used in \cite{Kontsevich-Operads} to describe the projection geometry of the compactified configuration spaces of $\R^N$.
\end{eg}

\begin{defn}
A function $f:\R^n\to\R$ is called a restricted real analytic function if
$$
f(\xx)=\begin{cases}
\widetilde f(\xx) & \text{if } \xx\in [-1, 1]^n \\
0 & \text{else}
\end{cases},
$$
where $\widetilde f$ is a real analytic function defined on an open neighborhood of $[-1, 1]^n$ in $\R^n$.

\end{defn}

\begin{eg}
The structure $\mathbb{R}_{an}$, defined as the smallest structure on $(\mathbb{R}, +, \cdot, <)$ in which all restricted analytic functions are definable,  is o-minimal. 
\end{eg}

Unless specified otherwise, in this paper the word "definable" always means definable in the o-minimal structure $\R_{an}$.

\begin{eg} 
The structure $\mathbb{R}_{an,exp}$, defined as the smallest structure on $(\mathbb{R}, +, \cdot, <)$ in which all restricted analytic functions and the global exponential function $\exp:\R\to \R$ are definable, is o-minimal. 
\end{eg} 

\begin{eg} \label{eg:exp-nondefinable}
The o-minimal condition imposes strong constraints on definable functions on affine spaces. For example, the function $e^z=e^{x+iy}=e^x(\cos y+i\sin y)$ on $\C\simeq \R^2$ is not definable in any o-minimal structure since $\cos y$ has infinitely many zeros.
\end{eg}

\subsubsection{Constructible Differential Forms 
}
 
\begin{defn} 
A set $A \subseteq \mathbb{R}^n$ is \textit{globally subanalytic} if its image under the compactifying map 
\[ \psi: \mathbb{R}^n \to [-1, 1]^n, \quad \psi(x) = \left( \frac{x_1}{\sqrt{1+x_1^2}}, \dots, \frac{x_n}{\sqrt{1+x_n^2}} \right) \]
is a subanalytic subset of $\mathbb{R}^n$, meaning it is locally a projection of a semi-analytic set. 

A function $f: A \to \mathbb{R}$ is globally subanalytic if its graph $\Gamma(f)$ is a globally subanalytic set.  
\end{defn}
 
\begin{thm}[\cite{LC2}]
A set/function is definable in $\R_{an}$ iff it is globally subanalytic.  
\end{thm}

\begin{defn}[Constructible functions]
A constructible function $f:A\to \R$ is a function of the form 
$$
f=\sum_{i=1}^k f_i \prod_{j=1}^{l_i} \log(g_{ij})
$$
where $f_i:A\to \R$, $g_{ij}:A\to \R_{>0}$ are globally subanalytic functions. 
\end{defn}

Constructible functions are preserved under both partial derivatives \cites{derivative-closed,derivative-closed-2} and parametric integration \cite{int-closed}. Moreover, coefficients of $\omega_{BM}$, the generalized Bochner-Martinelli kernel defined in Section \ref{sec:local-coho}, are constructible, hence the diagram integrals defined in Section \ref{sec:Conf-Space-Integrals} yield constructible forms.

 

Let $U$ be a definable open subspace of $\R^p\times \C^q$, and $p+q\le k<\infty$. 
\begin{defn}[Constructible differential forms]\label{def:constructible-forms}
For $s\le k$, define  
\ali{
\qquad A^s_{con,(k)}(U) = 
\set {\sum_{u+v=s}\sum_{|I|=u,|J|=v} f_{IJ} dx^{I} d\bar z^{J} }{f_{IJ} \text{ are constructible $C^{k-s}$   functions on $U$} }
}
and the space of constructible $C^{k-\bullet}$ forms
$$
\Aconk(U):=\bigoplus_{s=0}^{p+q} A^s_{con,(k)}(U).
$$
\end{defn}
$\Aconk(U)$ is naturally a CDGA when equipped with the wedge product and the differential $\d$.

\begin{rem}
We do not have a smooth constructible partition of unity since any smooth constructible function on $\R^p\times \C^q$ is real analytic \cite{derivative-closed-2}. This is why we assume $k<\infty$ in Definition \ref{def:constructible-forms}.
\end{rem}

We have the following chain of inclusions of CDGAs 
$$
\cdots \hlra A^\bullet_{con,(p+q+2)}(U)\hlra A^\bullet_{con,(p+q+1)}(U) \hlra A^\bullet_{con,(p+q)}(U).
$$

\begin{defn}
The constructible $C^{k-\bullet}$ $\bd$-cohomology of $U$ is defined as 
$$
\Hconk(U) = H^\bullet(\Aconk(U),\d).
$$
\end{defn}

\subsubsection{Mixed Real-complex Varieties}\label{sec:Real-complex-Varieties}

In this paper, all mixed real-complex varieties under consideration are open subvarieties of $\R^p\times\C^q$. 
Specifically, such a variety $U$ can be written as $D(I)=\R^p\times \C^q \setminus V(I)$, where $I \subset\C[x^1,\cdots,x^p,z^1,\cdots,z^q]$ is an ideal, and $V(I)\subset \R^p\times \C^q$ is the vanishing locus of $I$.
 
Since an open subvariety of $\R^p\times \C^q$ need not be connected when regarded as a smooth manifold, it is necessary to work with its connected components in order to construct a cover.
\begin{defn}
For every open subvariety $U$ of $\R^p\times\C^q$, we define a site $U_{alg}$ as follows. The category $U_{alg}$ has as objects the collection
$$
\Opens(U)=\set{V\subset U}{V \text{ is a finite union of connected components of the manifold $D(I_V)$ for some $I_V$}}.
$$
with morphisms given by inclusions. We equip $U_{alg}$ with the coverage generated by finite covers
\end{defn}

Chow’s theorem states that every closed analytic subvariety of 
$\C\P^n$ is algebraic. 
However, this statement does not extend to $\C^n$; the graph of $e^z$ is not algebraic. 
Nevertheless, o-minimal structures possess strong finiteness properties, which lead to the following o-minimal analog of Chow’s Theorem for $\C^n$.

\begin{thm}[O-minimal Chow’s Theorem, \cites{YS-chow1,YS-chow2}]\label{thm:ominimal-chow}
Let $Y\subset\C^n$ be a closed analytic subvariety
whose underlying set is definable in some o-minimal structure. Then $Y$ is algebraic.
\end{thm} 

\begin{prop}
\label{prop:Hcons-affine}
Let $X$ be a complex algebraic variety, and $f:X\to \C$ be a constructible $C^1$ function satisfying $\bp f=0$. Then $f$ is algebraic.
\end{prop}

\begin{proof}
Every constructible function is definable in the o-minimal structure $\R_{an,exp}$. 
For each affine open subset $U\subset X$, choose a closed embedding $U\hookrightarrow \C^n$ and apply Theorem \ref{thm:ominimal-chow} to the graph $\Gamma_{f|_U}\subset U\times\C \subset \C^k\times\C$, we obtain that  $f|_U$ is algebraic.
\end{proof}

\begin{eg}
\label{cor:0-cohomology-affine}
Let $U=\R^p\times\C^q \setminus V(I)$ and let $f$ be a constructible $C^1$ function on a connected component of $U$ with $\bd f=0$. Because  $\partial_{x^i}f=\partial_{z^j}f=0$ for all $i,j$, $f$ is locally constant in the $\R^p$-directions and holomorphic in the $\C^q$-directions. For every $\xx\in \R^p$, the restriction  $f|_{\lr{\xx}\times \C^q}$ is algebraic by Proposition \ref{prop:Hcons-affine}. Thus f can be written as a rational function $g/h$, $g,h\in\C[z^1,\cdots,z^q]$, with $h$ not vanishing on the component. 

Let $\Ocon(U)$ denote the space of constructible $C^1$ functions on $U$ that are $\bd$-closed.

\end{eg}

\begin{defn}\label{def:con-structure-sheaf}
Let $U\subset\R^p\times\C^q$ be an open subvariety. The functor  
\ali{
U_{alg} \ni V \lmt \OconU(V) 
} 
defines a sheaf on the site $U_{alg}$. 
This sheaf is called the constructible structure sheaf and is denoted by $\OconU$.  
\end{defn}
Note that when $p=0$ and $U\subset \C^q$, the sheaf $\OconU$ coincides with the usual structure sheaf of $U$ when $U$ is regarded as a complex variety.

\begin{thm}\label{thm:cons-resolution}
$\Aconk(U)$ computes the CDGA of derived global sections of the sheaf $\OconU$ on the site $U_{alg}$
$$
\Aconk(U) \cong R\Gamma(U,\Ocon).
$$
\end{thm}

We will not use Theorem \ref{thm:cons-resolution} in this paper and defer the full proof to a future work.

\begin{proof}[Sketch of proof]
Let $U_i$ be a finite cover of $U$, then the Thom–Whitney totalization $\operatorname{TW}\LR{\Gamma(U_\bullet,\OconU)}$ represents $R\Gamma(U,\Ocon)$. Using a calculation on affine intersections we can show there are quasi-isomorphisms of CDGAs
\ali{
\Aconk(U)
\overset{\simeq}{\lra}
\operatorname{TW}\LR{\Aconk(U_\bullet)}
\overset{\simeq}{\longleftarrow} 
\operatorname{TW}\LR{\Gamma(U_\bullet,\OconU)}.
}
\end{proof} 
 
\begin{rem}
Let $X$ be a reduced separated smooth complex algebraic variety of finite type. Then we can view $X$ as a $\R_{an}$-definable manifold and define the CDGA of constructible $C^{k-\bullet}$ differential forms on $X$. In this case, $\Aconk(X)$ is still a CDGA model for the derived global sections of the structure sheaf $\OO_X$ of regular functions on $X$
$$
\Aconk(X) \cong R\Gamma(X,\OO_X).
$$
In particular, the constructible Dolbeault  cohomology computes the sheaf cohomology of $\OO_X$
$$
\Hconk(X) \cong H^\bullet(X,\OO_X).
$$
A sheaf version also holds. 
\end{rem}

\subsubsection{The Constructible \texorpdfstring{$\bd$}{(d+∂)}-equation}
In this section, we prove that the constructible $\bd$-equation is solvable on a product of punctured disks in $\R$ or $\C$, which will be used to calculate the constructible $\bd$-cohomology of $\maka$. 

 
\begin{notation}
Let $U\subset \R^p\times \C^q$ be a definable open subset. We say the constructible $\bd$-equation is solvable on $U$ if for every $\alpha\in A_{con,(k)}^{\bullet>0}(U)$ with $\bd\alpha=0$, there exists $\beta\in A_{con,(k+1)}^{\bullet-1}(U)$ such that $\bd \beta=\alpha$. 
\end{notation}

\begin{rem}
Let $\alpha$ be a constructible $C^k$ $\bd$-closed form. 
For $|I|\le k$,  $\partial^I \alpha$ is constructible, so its coefficients are locally H\"older continuous by Proposition \ref{prop:Holder-continuous}. 
Consequently, every solution $\beta$ of the $\bd$-equation $\bd \beta=\alpha$ is automatically $C^{k+1}$.
\end{rem}

\begin{lem}
Let $f$ be a continuous globally subanalytic function on $\R^q$. Then there exist $r,C\in\R$ such that $|f(\xx)|\le |\xx|^r+C$.
\end{lem}

\begin{proof}
Let $\tau$ denote the definable map
\ali{
\tau: S^{q-1} \times [0,+\infty) &\lra \R^q \\
     (\xx,t) & \lmt t\xx
}
and $F$ be the definable function on $\R^{q}\times \R$, which is defined by
\ali{
F(\xx,t)=
\begin{cases}
f\circ \tau & \xx\in S^{q-1} \\
0 & \text{else}
\end{cases}.
}
Denote by $F_\xx$ the restriction of $F$ to $\set{(\xx,t)}{t\in \R}$. 
By \cite{poly-bound2}, Proposition 5.2, there exist $r_1,\cdots,r_l \in\Q$ such that for every $\xx\in \R^{q}$, either $F_\xx$ is identically zero for sufficiently large t, or there is an index i such that $\lim_{t\to +\infty} F_\xx(t)/t^{r_i}=c(\xx) \ne 0$.

Let $r=\max\lr{r_1,\cdots,r_l,0}+1$ and $\tilde F(\xx,t)=F(\xx,t) t^{-r}$, then $\lim_{t\to+\infty}\tilde F(\xx,t)=0$ for each $\xx$. By \cite{int-closed}, Proposition 1.5, ${F(\xx,t)\le 1}$ holds for $t>g(\xx)$, for some globally subanalytic function $g(\xx)$. Since $S^{q-1}$ is compact, there exists $C_1>0$ such that $g(\xx)<C_1$ for $\xx\in S^{q-1}$. Let $C=\max_{|\xx|<C_1}|f(\xx)|$, then we have $|f(\xx)|\le |\xx|^r+C$. 
\end{proof}

\begin{cor}\label{cor:polynomial-bound}
Let $f$ be a continuous constructible function on $\R^q$. Then there exist $r,C\in\R$ such that $|f(\xx)|\le |\xx|^r+C$.
\end{cor}

\begin{prop}\label{prop:Constructible-poincare-disk}
The  constructible  $\bd$-equation is solvable on $\C$. 
\end{prop} 

\begin{proof}
Let $f(z)$ be a constructible $C^k$ function on $\C$. We have $|f(z)|\le |z|^r+C$ for some $r,C$ by Corollary \ref{cor:polynomial-bound}. 

For $|w|>|z|$, we have  
\ali{
\frac{1}{z-w}
=-\frac{1}{w} \frac{1}{1-\frac{z}{w}}
=-\sum_{l\ge 0} \frac{z^l}{w^{l+1}}.
}
Choose an integer $s>r+3$ and define 
$$
g_{\infty}(z,w)=-\sum_{l= 0}^{s} \frac{z^l}{w^{l+1}}.
$$
$g_{\infty}(z,w)$ is holomorphic in $z$. Let $\chi_{\infty}(w)$ be a constructible bump function satisfying
$$
\chi_{\infty}(w)=
\begin{cases}
1 & |w|>N+1 \\
0 & |w|<N
\end{cases}.
$$
Let $\dVol_w$ be a translation-invariant  volume form on $\C$, normalized such that 
$$
\int_\C \dVol_w \, \partial_{\bar z} \LR{\frac{1}{z-w}} = 1.
$$ 
Then the following integral
$$
h(z) = \int_\C \dVol_w \, f(w) \LR{\frac{1}{z-w} - g_{\infty}(z,w)\chi_{\infty}(w)}
$$
converges absolutely for every $z$, and defines a $\bp$-inverse of $f(z)d\bar z$. $h$ is constructible since  partial derivatives and parametric integrals preserve constructibility. 
\end{proof}
 
\begin{lem}
Let $f$ be a constructible function on $\C\setminus\lr{p_1,\cdots,p_u}$. Then, for each i, there exist some $r_i,C_i$ such that $|f|< C_i|z-p_i|^{r_i}$ near each $p_i$. 
\end{lem}

\begin{proof}
Using a constructible $C^k$ partition of unity, we may assume $f$ is a constructible function on $\C\setminus\lr{p_i}$. Then apply Corollary \ref{cor:polynomial-bound} to the function $f((z-p_i)^{-1})$.
\end{proof}

\begin{cor}\label{cor:poincare-punctured-disk}
The  constructible $\bd$-equation is solvable on $\C\setminus\lr{p_1,\cdots,p_u}$. 
\end{cor} 
\begin{proof} [Sketch of proof]
Expand $\frac{1}{z-w}= -\sum_{l\ge 0} \frac{(z-p_i)^l}{(w-p_i)^{l+1}}$ near each $p_i$ and construct $g_i(z,w)$ and $\chi_i(w)$ as in Proposition \ref{prop:Constructible-poincare-disk}. Then 
$$
 \int_\C \dVol_w \, f(w) \LR{\frac{1}{z-w} - g_{\infty}(z,w)\chi_{\infty}(w)-\sum_{i=1}^u g_{i}(z,w)\chi_{i}(w)} 
$$
converges and defines a constructible $\bp$-inverse of $fd\bar z$. 
\end{proof}

\begin{cor}  \label{cor:poincare-punctured-diskR}
Let $\DD$ be an open disk in $\C$ with radius $0<R<\infty$. Then the  constructible $\bd$-equation is solvable on $\DD$ and on  $\DD\setminus\lr{p_1,\cdots,p_u}$. 
\end{cor}

\begin{proof}[Sketch of proof]
Expand 
\ali{
\frac{1}{z-w}
=\frac{1}{\LR{z-\frac{Rw}{|w|}} -\LR{w-\frac{Rw}{|w|}}}
=-\sum_{l\ge0} \frac{\LR{z-\frac{Rw}{|w|}}^l}{\LR{w-\frac{Rw}{|w|}}^{l+1}}
}
near the boundary and construct $g_{\partial \DD}$ and $\chi_{\partial \DD}$ as in Proposition \ref{prop:Constructible-poincare-disk}. Then 
$$
 \int_{\DD} \dVol_w \, f(w) \LR{\frac{1}{z-w} - g_{\partial \DD}(z,w)\chi_{\partial \DD}(w)-\sum_{i=1}^u g_{i}(z,w)\chi_{i}(w)} 
$$
converges and defines a constructible $\bp$-inverse of $fd\bar z$. 
\end{proof}

\begin{lem} \label{lem:definable-Poincare}
Let $U=\prod_{i}U_i$, where each $U_i$ is one of the following:  $\C$, $\C\setminus\lr{p_1,\cdots,p_u}$, $\DD$, $\DD\setminus\lr{p_1,\cdots,p_u}$, a definable open subset of $\R$. 
Then the positive-degree cohomology of $(\Aconk (U),\d)$ vanishes. 
\end{lem}

The proof is like that of Lemma \ref{lem:Ck-Poincare}, except  each integral $\int \frac{dw}{z-w} (-)$ in the one-dimensional $\bp$-Poincar\'e Lemma is replaced  by a convergent one.

\begin{proof}[Sketch of proof]
Let $\alpha(\xx,\zz)$ be a $\bd$-closed form on $U$ with degree $\ge1$. 
Use Proposition \ref{prop:Constructible-poincare-disk}, Corollary \ref{cor:poincare-punctured-disk} and Corollary \ref{cor:poincare-punctured-diskR} to eliminate all $d\bar z^j$ in $\alpha$ inductively, 
and then use the homotopy operator in Lemma \ref{lem:Ck-Poincare}.
\end{proof}

\section{Two-point Configurations}\label{sec:2-point-CDGA}

In this section, we study configuration spaces of two points in $\R^{n'} \times \C^n$. We may fix one point at the origin, therefore the other point lies in the punctured space $\paka$, in which a point is denoted by 
$$\ZZ=(\xx,\zz)=(x^1,\cdots,x^{n'},z^1,\cdots,z^n).$$

The plan of this section is as follows.
\begin{itemize}
\item \ref{sec:local-coho}: We compute the cohomology groups of $\A(\paka)$ and  $\Adef(\paka)$. 
\item \ref{sec:residues}: We define the residue map $\oint$ which picks up the cohomology class of a $\bd$-closed form on $\paka$.
\item \ref{sec:loc-ring-structure}: We determine the ring structures on $H^\bullet(\paka)$ and $\Hdef(\paka)$.
\item \ref{sec:chain-level}: We study a mixed Jouanolou model $J_{n',n}$ associated to $\paka$ and relations of its generators. 
\item \ref{01}-\ref{>1>0}: We study formality of $\Adef(\paka)$ and $\A(\paka)$ for $n>0$. The results are  summarized in the following table: 
\begin{center}
\begin{tabular}{|c|c|c|}
    \hline
       & $n=1$ & $n>1$ \\
    \hline
    $n'=0$   & formal & non-formal \\
    \hline
    $n'=1$   & non-formal & non-formal \\
    \hline
    $n'>1$   & formal & formal \\
    \hline
\end{tabular}
\end{center}
\end{itemize}

\subsection{Cohomology Groups}\label{sec:local-coho}

Let $\omega_{BM}$ be the generalized Bochner-Martinelli kernel 
$$
\omega_{BM}
= 
\frac{C}{(\rSquare )^{n + \frac{n'}{2}}} \cdot 
\left( \sum_{i = 1}^n (- 1)^{i - 1} 2\bar z^i \left(
    \prod_{j \neq i} d \bar z^j \right) d^{n'} x +
 \sum_{i = 1}^{n'} (- 1)^{n + i - 1} x^i d^n
    \bar z \left( \prod_{j \neq i} d x^j \right) \right) \prod_{i=1}^n dz^i
$$
where $C$ is a constant such that $\bd \omega_{BM}=\delta_0 \dVol$,  where $\delta_0$ is the Dirac delta distribution at $0$.  Let 
$$
\omega= 
\frac{C}{(\rSquare )^{n + \frac{n'}{2}}} \cdot 
\left( \sum_{i = 1}^n (- 1)^{i - 1} 2\bar z^i \left(
    \prod_{j \neq i} d \bar z^j \right) d^{n'} x +
 \sum_{i = 1}^{n'} (- 1)^{n + i - 1} x^i d^n
    \bar z \left( \prod_{j \neq i} d x^j \right) \right).
$$
In this paper, the specific numerical value of $C$ will not be used, and we will sometimes omit $C$ altogether.

Denote $\partial_i=\LL_{\partial_{z^i}}$. Let $\overline{\C[\PPi ]\omega}$ be the closure of the subspace 
$$
\C[\PPi ]\omega \subset \A(\paka) 
$$
where $\A(\paka)$ is equipped with the standard nuclear topology.

\begin{prop}\label{prop:coh-local}
For $n'+n\ge2$, 
The cohomology group of $\A(\paka)$ is 
$$
H^\bullet(\paka)=
\begin{cases}
\OO^{hol}(\C^n)  & \bullet=0 \\
\overline{\C[\PPi ]\omega} & \bullet=n'+n-1 \\
0 & \text{else}
\end{cases}
$$
where $\OO^{hol}(\C^n)$ is the space of holomorphic functions on $\C^n$. The cohomology group of $\Adef(\paka)$ is 
$$
\Hdef(\paka)  = 
\begin{cases}
\C[\ZZi ]  & \bullet=0 \\
\C[\PPi ]\omega & \bullet=n'+n-1 \\
0 & \text{else}
\end{cases}
$$

For $n'=0,n=1$, we have 
$$
H^\bullet(\C\punctured) =
\begin{cases}
\OO^{hol}(\C\punctured)  & \bullet=0 \\
0 & \text{else}
\end{cases}
$$
$$
\Hdef(\C\punctured) =
\begin{cases}
\C [z,z^{-1}]  & \bullet=0 \\
0 & \text{else}
\end{cases}
$$

In both cases, there is a natural inclusion
$$
\Hdef(\paka) \hlra H^\bullet(\paka)
$$
which naturally identifies $\Hdef(\paka)$ with a dense subspace of $H^\bullet(\paka)$.
\end{prop}

\begin{proof} 
We  compute both cohomology groups simultaneously using \v Cech cohomology. 
Let $\mathcal U$ be the open cover of $\paka$ defined by
$$
\mathcal U=\lr{U_i}_{i=1,\cdots,n+2n'}, \qquad 
U_i=\begin{cases}
\vspace{1mm} \set{(\xx,\zz)\in \paka}{x^{i}> 0}, & 1\le i\le n' \\
\vspace{1mm} \set{(\xx,\zz)\in \paka}{x^{i-n'}< 0}, & n'+1\le i\le 2n' \\
\set{(\xx,\zz)\in \paka}{z^{i-2n'}\ne 0}, & 2n'+1\le i\le 2n'+n
\end{cases}
$$ 
Denote $U_I=\cap_{i\in I} U_i$. Then each nonempty $U_I$ is of the form
$$
U_I = \R^{I_1} \times (\R_+)^{J_1} \times  \C^{I_2} \times (\C\punctured)^{J_2}.
$$
By Lemma \ref{lem:Ck-Poincare} and Lemma \ref{lem:definable-Poincare}, we have
$$
H^{\bullet>0}(U_I) = H_{con,(k)}^{\bullet>0}(U_I) = 0.
$$

Denote by $\Hol$ the sheaf of $\bd$-closed functions, and recall $\Ocon$ is the sheaf of constructible $\bd$-closed functions defined in Section \ref{sec:Real-complex-Varieties}.
By Example \ref{cor:0-cohomology-affine}, we have 
\ali{
\Hol (U_I) &=   \Hol(\C)^{\hotimes I_2} \,\hotimes\, \Hol(\C\punctured)^{\hotimes J_2}, \\
\Ocon (U_I) &= \Ocon(\C)^{\hotimes I_2} \,\hotimes\, \Ocon(\C\punctured)^{\hotimes J_2}
}
and 
\ali{
\Hol(\C\punctured) &= \Hol(\C) \oplus \Hol_{-}(\C\punctured), \\
\Ocon(\C\punctured) &= \Ocon(\C) \oplus \Ocon_{-}(\C\punctured),
}
where $\Hol_{-}(\C\punctured)$ is the space of power series $\sum_{k>0} a_{k} z^{-k}$ that converges on $\C\punctured$, and $\Ocon_{-}(\C\punctured)$ is the subspace of $\Hol_{-}(\C\punctured)$ which contains power series with finite terms. 

A direct computation shows the \v Cech cohomology groups vanish unless $\bullet=0 \text{ or } n'+n-1$, where
\begin{itemize}
\item The $0$-th cohomology groups are 
\ali{
H^{0}(\paka) &= \Hol (\paka)   = \Hol (\C^n), \\
H_{con,(k)}^{0}(\paka) &= \Ocon(\paka)   = \C[z^1,\cdots,z^n]
}
\item The $n'+n-1$-th cohomology groups are 
\ali{
H^{n'+n-1}(\paka) &= \LR{\Hol_{-}(\C\punctured)}^{n}, \\
H^{n'+n-1}_{con,(k)}(\paka) &= \LR{\Ocon_{-}(\C\punctured)}^{n}.
}
\end{itemize}

Let 
$$
\widetilde\gamma=\frac{\sign x^1 \cdots \sign x^{n'}}{z^1\cdots z^n} 
, \qquad 
\sign x=\begin{cases}
1 & x>0 \\
-1 & x<0
\end{cases}.
$$
Then the $n'+n-1$-th \v Cech cohomology groups are 
\ali{
H^{n'+n-1}(\paka) &= \overline{\C[\PPi ]\widetilde\gamma}, \\
H_{con,(k)}^{n'+n-1}(\paka) &= \C[\PPi ]\widetilde\gamma 
} 
where the former is the completion of the latter with respect to the nuclear topology. 

Next, we use the collating map of the \v Cech-$\bd$ bicomplex to find a representative of $\widetilde\gamma$ as a differential form. 
Using the ``partition of unity"  
$$
\rho_i=\begin{cases}
\frac{(x^{i})^2}{\rSquare } & 1 \le i \le n' \\
\frac{2z^{i}\bar z^{i}}{\rSquare } & n'+1\le i\le n'+n \\
\end{cases}
$$
the image of $\widetilde\gamma$ under the collating map
\ali{
K: C^k(A^0(\paka),\delta) & \lra (A^k(\paka), \d ) \\
\alpha &\lmt \sum_{i_0<\cdots<i_k} \sum_{l=0}^k \alpha_{i_0\cdots i_k} (-1)^{l}\rho_{i_l} \bd \rho_{i_0} \cdots \widehat{\bd \rho_{i_l}} \cdots \bd \rho_{i_k}
} 
generates the cohomology class. Note that $\lr{\rho_i}$ is not a genuine partition of unity because $\rho_i$ are not compactly supported, so the image of $\widetilde\gamma$
is not smooth (in fact, it diverges at $0$). However, it represents the true cohomology class in the space of currents.
We now find a smooth representative of this cohomology class. Let 
$$
\gamma=\frac{(\rSquare ))^{n'/2}}{x^1\cdots x^{n'} z^1\cdots z^n} \in 
\bigoplus_{i_1\in\lr{1,2},\cdots, i_{n'}\in \lr{2n'-1,2n'}} A^0\LR{U_{i_1}\cap \cdots \cap U_{i_{n'}}\cap U_{2n'+1}\cap \cdots \cap U_{2n'+n}}.
$$
Then $K\gamma$ is equal to $\omega$ up to a multiplicative constant, and a  direct computation shows $K\widetilde\gamma-\omega$ is a $\bd $-coboundary in the space of currents. Moreover, for multi-index $I=(i_1,\cdots,i_n)$, $i_1,\cdots,i_n\ge0$, denote $\partial^I=\partial_{z_1}^{i_1}\cdots \partial_{z_n}^{i_n}$, then the cohomology class of $K(\partial^I\widetilde\gamma)$ is represented by $\partial^I\omega$. Note also that $\partial^I\omega$ is definable because the map
$$
(-)^{n/2}: \R_{>0} \lra \R_{>0}, \qquad n\in \Z,
$$
is definable.  
Thus, we obtain
\ali{
H^{n'+n-1}(\paka) &= \overline{\C[\PPi ]\omega}, \\
H_{con,(k)}^{n'+n-1}(\paka) &= \C[\PPi ]\omega. \\
}
\end{proof}

\subsection{Residues}\label{sec:residues}
Let $\Omega^{n+n'-1}(\paka)$ denote the space of $\bd $-closed $(n+n'-1)$-forms on \\ $\paka$. 
Let $S^{n'+2n-1}$ be any sphere in $\paka$ with center $0$. Since $\omega dz^1\cdots dz^n$ is  the integral kernel for the inverse of $\d $, we have
\begin{lem}
Let  $i_1,\cdots,i_n\ge0$. Then
$$\int_{S^{n'+2n-1}} \, (z^1)^{i_1}\cdots (z^n)^{i_n} \omega \,dz^1\cdots dz^n=
\begin{cases}
1 & i_1=\cdots=i_n=0 \\
0 & \text{else}
\end{cases}.
$$
\end{lem}

Recall we denote $\partial_i=\LL_{\partial_{z^i}}$. 
\begin{lem} \label{vanish-lie-derivative}
Let $\alpha \in \Omega^{n+n'-1}(\paka)$. Then $\int_{S^{n'+2n-1}} \partial_i\alpha \,dz^1\cdots dz^n=0$.
\end{lem}
\begin{proof}
Denote by $d_{total}$ the total de Rham differential on $\R^{n'} \times \C^n$. 
Set $\beta=\alpha \,dz^1\cdots dz^n$. Since $d_{total} \beta=0$, we have $\partial_i\beta=\LL_{\partial_{z^i}}\beta=d_{total}\iota_{\partial_{z^i}}\beta+\iota_{\partial_{z^i}} d_{total}\beta=d_{total}\iota_{\partial_{z^i}}\beta$ and
$$
\int_{S^{n'+2n-1}} \partial_i\beta
=\int_{S^{n'+2n-1}} d_{total}\iota_{\partial_{z^i}}\beta
=\int_{\partial S^{n'+2n-1}} \iota_{\partial_{z^i}}\beta 
=0.
$$
\end{proof}

\begin{defn}
Define the residue map on $\Omega^{n+n'-1}(\paka)$ by 
\ali{
\oint: \Omega^{n+n'-1}(\paka) &\lra \C[[\PPi ]]\omega \\
  \beta       &\lmt \sum_{i_1,\cdots,i_n {\ge 0}}\LR{\int_{S^{n'+2n-1}} \, \frac{(-z^1)^{i_1}}{i_1!}\cdots \frac{(-z^n)^{i_n}}{i_n!} \beta \,dz^1\cdots dz^n} \partial_{1}^{i_1}\cdots\partial_{n}^{i_n}\omega
}
\end{defn}
Since $\bd \beta=0$, the result is independent of the choice of the  sphere $S^{n'+2n-1}$.  

\begin{lem}
When restricted to the subspace $\overline{\C[\PPi ]\omega} \subset \Omega^{n+n'-1}(\paka)$, the residue $\oint$ coincides with the natural inclusion  $\overline{\C[\PPi ]\omega} \hookrightarrow  \C[[\PPi ]]\omega$. 
\end{lem}
\begin{proof}
Using Lemma \ref{vanish-lie-derivative}, we see $\oint \partial_{1}^{i_1}\cdots\partial_{n}^{i_n}\omega = \partial_{1}^{i_1}\cdots\partial_{n}^{i_n}\omega$ by an integration-by-parts argument.
\end{proof}
 
\begin{prop}\label{oint-kernel}
$\ker\oint 
$ is precisely the space of $\bd $-exact forms, and $\Im \oint=\overline{\C[\PPi ]\omega}$.
\end{prop}
\begin{proof}
Since $\oint\bd (-)=
0$, the image of 
$$
\bd _{(n'+n-2)}: A^{n'+n-2}(\paka)\lra \Omega^{n'+n-1}(\paka)
$$
lies in $\ker\oint$. We have the following commutative diagram 
$$
\begin{tikzcd}
0 \arrow[r] & \Im  \bd _{(n'+n-2)} \arrow[r] \arrow[d, hook] & \Omega^{n'+n-1}(\paka) \arrow[d, phantom, sloped, "=\!=\!\!="]   &   &   \\
0 \arrow[r] & \Ker\oint \arrow[r] & \Omega^{n'+n-1}(\paka) \arrow[r, "\oint"] & \C[[\PPi ]]\omega   &  
\end{tikzcd}
$$
By the universal property of cokernels, the diagram can be extended to
$$
\begin{tikzcd}
0 \arrow[r] & \Im  \bd _{(n'+n-2)} \arrow[r] \arrow[d] & \Omega^{n'+n-1}(\paka) \arrow[d, phantom, sloped, "=\!=\!\!="] \arrow[r] & H^{n'+n-1}(\paka) \arrow[r] \arrow[d] & 0 \\
0 \arrow[r] & \Ker\oint \arrow[r] & \Omega^{n'+n-1}(\paka) \arrow[r, "\oint"] & \C[[\PPi ]]\omega   &  
\end{tikzcd}
$$ 

Since $H^{n'+n-1}(\paka)$ has a set of representatives $\overline{\C[\PPi ]\omega}$, the map 
$$
H^{n'+n-1}(\paka) \lra \C[[\PPi ]]\omega
$$
is injective with image $\overline{\C[\PPi ]\omega}$. 
By the five lemma, the map $\Im  \bd _{(n'+n-2)} \to \Ker \oint$ is an isomorphism.
\end{proof}
\begin{cor}\label{residue-convergence}
The series $\oint \beta\in \C[[\PPi ]]\omega$ converges on $\aka\punctured$ and  depends only on the cohomology class of $\beta$. 
\end{cor}

\subsection{Ring Structures}\label{sec:loc-ring-structure}

Let $\xalpha=(\alpha_1,\cdots,\alpha_n)$ and $\xbeta=(\beta_1,\cdots,\beta_n)$ be multi-indices and denote 
$$
\zz^{\xalpha}\partial^{\xbeta}\omega = 
(z^1)^{\alpha_1}\cdots (z^n)^{\alpha_n}   \partial_1^{\beta_1}\cdots \partial_n^{\beta_n}\omega.
$$

\begin{prop}\label{prop:local-cohomology}
The ring structures on 
\ali{
H^\bullet (\paka) &= \Hol(\C^n) \oplus \overline{\C[\PPi ]\omega} \\
\Hdef(\paka) &=
\C[\ZZi ] \oplus \C[\PPi ]\omega \\
} 
are given by the following relations:
$$
\zz^{\xalpha_1} \cdot  \zz^{\xalpha_2}= \zz^{\xalpha_1+\xalpha_2}, \qquad \partial^{\xbeta_1}\omega \cdot \partial^{\xbeta_2}\omega=0,
$$ 
$$
\zz^{\xalpha}\partial^{\xbeta}\omega=
\begin{cases}
(-1)^{\alpha_1+\cdots +\alpha_n} \prod_{i=1}^n \frac{\beta_i!}{(\beta_i-\alpha_i)!} \, \partial^{\xbeta-\xalpha}\omega & \alpha_i\le \beta_i \text{ for all } i \\
0 & \alpha_i>\beta_i \text{ for some } i \\
\end{cases}.
$$ 
\end{prop}  

\begin{proof}
Using $[\partial_i,z^i]=1$, we have by induction
$$
z^i \partial_i^{\beta_i}(-) = -\beta_i \partial_i^{\beta_i-1}(-) + \partial_i^{\beta_i}(z^i\cdot(-)),
$$
$$
(z^i)^{\alpha_i} \partial_i^{\beta_i}(-) = (-1)^{\alpha_i} \frac{\beta_i!}{(\beta_i-\alpha_i)!} \partial_i^{\beta_i-\alpha_i}(-) + \partial_i^{\beta_i}((z^i)^{\alpha_i}\cdot(-)).
$$
By Proposition \ref{oint-kernel}, $\zz^{\xalpha}\partial^{\xbeta}\omega$ is $\bd$-exact iff $\oint \zz^{\xalpha}\partial^{\xbeta}\omega=0$, i.e., for every non-negative $\xalpha'$, 
$$
\int_{S^{n'+2n-1}} \, \zz^{\xalpha+\xalpha'}\partial^{\xbeta}\omega\,dz^1\cdots dz^n=0.
$$
Moreover, we have $\int_{S^{n'+2n-1}} \, \partial_i(-) \,dz^1\cdots dz^n=0$ by Lemma \ref{vanish-lie-derivative}.
By induction we see
\begin{itemize}
\item If $\alpha_i>\beta_i$ for some $i$, then $\zz^{\xalpha+\xalpha'}\partial^{\xbeta}\omega$ is $\bd$-exact for any $\xalpha'$, so $\zz^{\xalpha}\partial^{\xbeta}\omega$ is $\bd$-exact.
\item If $\alpha_i\le \beta_i$ for all $i$,  then 
$\zz^{\xalpha}\partial^{\xbeta}\omega- (-1)^{\alpha_1+\cdots +\alpha_n} \prod_{i=1}^n \frac{\beta_i!}{(\beta_i-\alpha_i)!} \,  \partial^{\xbeta-\xalpha}\omega$ is $\bd$-exact.
\end{itemize}

\end{proof}

\begin{prop}
The residue pairings 
\ali{
H^0(\paka) \times H^{n'+n-1}(\paka) &\lra \C \\
H_{con,(k)}^0(\paka) \times H_{con,(k)}^{n'+n-1}(\paka) &\lra \C \\
(\zz^{\xalpha}, \partial^{\xbeta}\omega) & \lmt \int_{S^{n'+2n-1}} \, \zz^{\xalpha} \partial^{\xbeta}\omega\,dz^1\cdots dz^n
}
are nondegenerate. 
\end{prop}  

\subsection{The Mixed Jouanolou Model and Chain Level Relations}\label{sec:chain-level}  
The Jouanolou model is used in \cites{chiral-algebra, Higher-CA-Jouanolou, raviolo-vertex-algebras} as a CDGA model for the derived global sections of configuration spaces. In this section, we construct a mixed Jouanolou model, which is a mixed version of Jouanolou torsor and spheres, serving as an algebraic CDGA model for $\maka$. 

Define
\ali{
J_{n',n} 
= \frac{\C[z^k, y^l, w^k, \bfd y^l, \bfd w^k]_{1\le l\le n', \, 1\le k \le n}}{\lR{\sum_{l=1}^{n'} (y^l)^2  +  \sum_{k=1}^n z^k w^k -1, \,\, \sum_{l=1}^{n'} 2y^l \bfd y^l   +  \sum_{k=1}^n z^k \bfd w^k}}.
}

\begin{prop}\label{prop:coho-Jouanolou}
The cohomology of $J_{n',n}$ is isomorphic to $\Hdef(\paka)$.  
\end{prop}
\begin{proof}[Sketch of proof]
Let $R=\C[z^1,\cdots,z^n]$, $S=\Spec R$ and $I=(z^1,\cdots,z^n)$. Define
$$ 
B=R[y^1,\cdots,y^{n'},w^1,\cdots,w^n]/ (\sum_{l=1}^{n'} (y^l)^2  +  \sum_{k=1}^n z^k w^k -1),
$$
then $J_{n',n}$ is precisely the relative de Rham complex of the affine $R$-scheme $X=\Spec B$
$$
J_{n',n} = \Gamma(X,\Omega^\bullet_{X/S}).
$$
Notice that the projection $X\to S$ is smooth. Let $U=S \setminus V(I)$. On each principal open set $D(z^j)\subset U$ we can solve 
$$
w^j = \frac{1-\sum_l (y^l)^2-\sum_{k\ne j}z^kw^k}{z^j},
$$
so $X_{z^j} \simeq \C^{n'+n-1}_{R_{z^j}}$. The relative Poincar\'e Lemma gives a quasi-isomorphism
$$
\Omega^\bullet_{X_{z^j}/R_{z^j}} \simeq R_{z^j}
$$
for each $j$. These quasi-isomorphisms are compatible, so we have
$$
R\Gamma(X_U,\Omega^\bullet_{X_U/U}) \simeq R\Gamma(U,\OO_U).
$$

For the $I$-supported part, let 
$$
\widehat R = \C[[z^1,\cdots,z^n]], \quad \widehat B = \varprojlim_r B/I^rB.
$$
Define an element $s=(1-\sum_k z^kw^k)^{1/2}\in \widehat B$ and the following $\widehat R$-algebra map
$$
\widehat R[x^1,\cdots,x^{n'},w^1,\cdots,w^n]/(\sum_l (x^l)^2-1) \lra \widehat B, \qquad x^l \mapsto y^l/s, \quad w^k\mapsto w^k
$$
is an  isomorphism. Since $\Omega^\bullet_{B/R}$ is a finitely generated $B$-module, the local cohomology can be computed after completion, which is
$$
H^k (R\Gamma_I(\Omega^\bullet_{X/S})) \simeq H^n_I(R) \otimes H^{k-n}_{dR}(Q)
$$
where $Q=\Spec \C[x^1,\cdots,x^{n'}]/(\sum_l (x^l)^2-1)$, and its algebraic de Rham cohomology $H^{\bullet}_{dR}(Q)$ is computed by the singular cohomology 
$$
H^{\bullet}_{dR}(Q) \cong
H^\bullet_{sing}(Q(\C),\C) \cong H^\bullet_{sing}(TS^{n'-1},\C) \cong H^\bullet_{sing}(S^{n'-1},\C).
$$
Finally, we apply the distinguished triangle
$$
R\Gamma_I(\Omega^\bullet_{X/S}) \to \Gamma(X,\Omega^\bullet_{X/S}) \to R\Gamma(U,\OO_U) \overset{+1}{\to}
$$ 
to find 
$$
H^{k}(X,\Omega^\bullet_{X/S}) \simeq H^n_I(R) \otimes \tilde H^{k-n}_{dR}(Q), \quad \text{for } k>0.
$$
Here $\tilde H$ is the reduced cohomology. Consequently, we have
\ali{
H^\bullet(J_{n',n},\bfd)=
\begin{cases}
\C[z^1,\cdots,z^n] & \bullet=0 \\
\frac{1}{z^1\cdots z^n} \C[\frac{1}{z^1},\cdots,\frac{1}{z^n}] & \bullet=n'+n-1 \\
0 & \text{else}
\end{cases}
} 
\end{proof}
 
Define 
$ 
\mu_i=\begin{cases}
\frac{x^{i}}{\LR{\rSquare}^{1/2}} & 1 \le i \le n' \\
\frac{2\bar z^{i-n'}}{\rSquare } & n'+1\le i\le n'+n \\
\end{cases}
$, so 
$
\rho_i=\begin{cases}
\frac{(x^{i})^2}{\rSquare } = \mu_i^2 & 1 \le i \le n' \\
\frac{2z^{i-n'}\bar z^{i-n'}}{\rSquare } = z^{i-n'}\mu_i & n'+1\le i\le n'+n \\
\end{cases}
$. 

\begin{prop} \label{prop:two-point-Jouanolou} 
There is an injective quasi-isomorphism of CDGAs
$$
(J_{n',n},\bfd) \islra (\Adef(\paka),\d)
$$ 
given by
\ali{
z^k \mapsto z^k, \qquad
y^l \mapsto \mu_l, \qquad
w^k \mapsto \mu_{k+n'}.
} 
\end{prop} 
\begin{proof}
These assignments define a CDGA map 
$$
\C[z^k, y^l, w^k, \bfd y^l, \bfd w^k]_{1\le l\le n', \, 1\le k \le n} \lra \Adef(\paka)
$$
which annihilates the defining relations of $J_{n',n}$, so we get a CDGA map 
$$
(J_{n',n},\bfd) \lra (\Adef(\paka),\d).
$$
Moreover,  it maps the basis of $H^\bullet(J_{n',n},\bfd)$ in the proof of Proposition \ref{prop:coho-Jouanolou} to the basis of $\Hconk(\paka)$ in Proposition \ref{prop:coh-local}.
\end{proof}

\begin{prop}
There is a left $D_{\C^{n}}$-module structure on $J_{n',n}$, where $z^s$ acts by multiplication and 
$$
[\partial_{z^s}, \bfd]=0, \quad 
\partial_{z^s} z^l =\delta_{sl}, \quad 
\partial_{z^s} y^l = -\frac{1}{2}y^l w^s, \quad
\partial_{z^s} w^k = -w^k w^s.
$$
Moreover, the morphism
$$
(J_{n',n},\bfd) \islra \Adef(\paka)
$$
is a  quasi-isomorphism of $D_{\C^n}$-modules. 
\end{prop}

Notice that 
$$
\mu_i \bd \rho_i = \rho_i\bd \mu_i \phi(i), \quad  \phi(i)=\begin{cases}
2 & 1\le i\le n' \\
1 & n'+1\le i\le n'+n
\end{cases}.
$$
Consequently, 
we have
\ali{
\omega
=\frac{1}{\prod_{i=1}^n z^i \prod_{i=1}^{n'}  \mu_i} \sum_{i=1}^{n+n'} (-1)^{i-1} \rho_i \prod_{j\ne i}\bd \rho_j  
=  \sum_{i=1}^{n+n'} (-1)^{i-1} \mu_i \prod_{j\ne i}\bd \mu_j \phi(j). \\  
}

Let $I$ be a subset of $\lr{1,\cdots, n}$. If $n'=0$, we assume that $I$ is nonempty. Suppose $I=\lr{i_1,\cdots,i_k}$, $i_1<\cdots<i_k$, we consider the $(n'+k-1)$-\v Cech cochain 
$$
\gamma^I
=\frac{(\rSquare ))^{n'/2}}{x^1\cdots x^{n'} z^{i_1}\cdots z^{i_k}} 
\in \bigoplus_{j_1\in\lr{1,2},\cdots, j_{n'}\in \lr{2n'-1,2n'}} A^0\LR{U_{j_1}\cap \cdots \cap U_{j_{n'}}\cap U_{2n'+i_1}\cap \cdots \cap U_{2n'+i_k}}.
$$
Set $\omega^I=K\gamma^I$,  $n'+I:=\lr{n'+i_1,\cdots,n'+i_k}$, and $\widehat I=\lr{1,\cdots,n'} \cup (n'+I)$. Then
\ali{
\omega^I
= \frac{1}{\prod_{i\in I}  z^i \prod_{i=1}^{n'} \mu_i} \sum_{i\in \widehat I} \pm \rho_i \prod_{j\in \widehat I \setminus \lr{i}} \bd \rho_j 
=\sum_{i\in \widehat I} \pm \mu_i \prod_{j\in \widehat I \setminus \lr{i}} \bd \mu_j \phi(j).
}
We see $\omega^{\setn}=\omega$
, and 
$$
\bd \omega^I = \sum_{j\in \setn\setminus I} (-1)^{(j,I)} z^j \omega^{I\cup \lr{j}}
$$
where $(j,I)$ is the position from the rear of the index $j$ when $I\cup \lr{j}$ is ordered in the usual manner. For example,
$$
z^i\omega^{\setn} = (-1)^{i-1} \bd  \omega^{\setn\setminus\lr{i}}
$$
is $\bd $-exact.  

We also introduce an ordered version 
$$
\omega^{i_1\cdots i_n}:= (-1)^{(i_1\cdots i_n)} \omega^{\lr{i_1,\cdots, i_n}}.
$$

\subsection{Formality for \texorpdfstring{$n'=0$, $n=1$}{n'=0, n=1}}\label{01}
\begin{prop}
$\A(\C \punctured )$ and $\Adef(\C \punctured)$ are formal.
\end{prop}

\begin{proof}
We have $H^1(\C\punctured)=0$ because noncompact Riemann surfaces are Stein manifolds. Moreover, $H_{con,(k)}^1(\C\punctured)=0$ by Corollary \ref{cor:poincare-punctured-disk}. The proof follows from the natural inclusions 
$$
\Hol(\C\punctured) \islra \A(\C \punctured )
,\qquad  
\C[z,z^{-1}] \islra \A_{con,(k)}(\C \punctured ).
$$ 
\end{proof}

\subsection{Non-formality for \texorpdfstring{$n'=0$, $n=2$}{n'=0, n=2}}\label{02}
\begin{prop}
$\A(\C^2 \punctured )$ and $\Adef(\C^2 \punctured )$ are non-formal.
\end{prop}


\begin{proof}
We prove the statement for $\A(\C^2 \punctured )$; the other case is analogous. 
Notice that $\omega^{\lr{1,2}}=\mu_1\bp \mu_2 - \mu_2\bp \mu_1$, $z^1\omega^{\lr{1,2}}=\bp\mu_2$ and $z^2\omega^{\lr{1,2}}=-\bp\mu_1$. 
The triple Massey product 
$$
\lR{[z^1],[\omega^{\lr{1,2}}],[z^2]}=[z^1\cdot(-\mu_1) - \mu_2 \cdot z^2]=[-1] \in 
\frac{H^\bullet(\C^2 \punctured )}{z^1 \cdot  H^\bullet(\C^2 \punctured )+H^\bullet(\C^2 \punctured ) \cdot z^2} 
$$
is nontrivial because $-1$ is not contained in the ideal of $H^\bullet(\C^2 \punctured )$ generated by $z^1,z^2$. The proof follows from Proposition \ref{Massey-nonformal}.
\end{proof}

\subsection{Non-formality for \texorpdfstring{$n'=0$, $n\ge3$}{n'=0, n≥3}}\label{0>2}
\begin{prop}\label{Cn-nonformal}
For $n\ge3$, $\A(\C^n \punctured )$ and $\Aconk(\C^n \punctured )$ are non-formal. 
\end{prop} 

\begin{lem} \label{lem:omega-tree}
Let $\mathcal T_n$ be the set of maximal trees on the vertex set $\setn$; that is, each $T \in \mathcal T_n$ is an unoriented graph with $n-1$ edges and no cycles. Let $e$ be an edge with vertices $\lr{i,j}$, $i<j$. Denote $\omega^e=\omega^{\lr{i,j}}=\mu_i\bp\mu_j-\mu_j\bp\mu_i$. Then 
$$
\omega^{\setn}
= \sum_{i=1}^n (-1)^{i-1}\mu_i \bp\mu_1\cdots \widehat{\bp\mu_i}\cdots \bp\mu_n 
=\sum_{T\in \mathcal T_n} \sT \prod_{i=1}^n (z^i)^{d^T_i-1} \bigwedge_{e\in T} \omega^{e} 
$$
where $d^T_i$ is the degree of the vertex $i$ in the tree $T$, the wedge product is taken in lexicographical order of the pairs $(i,j)$, and $\sT\in \lr{\pm1}$ is a sign.
\end{lem} 

\begin{proof}
We first determine the sign $\sigma(T)$. Let $T\in\mathcal T_n$ be a maximal tree on $\setn$. Choose any vertex $r\in\setn$ as the root. For each vertex $v\ne r$, let $p(v)$ denote its parent, that is, the neighbor of $v$ that is one step closer to $r$. Let $e_1<\cdots<e_{n-1}$ be all $n-1$ edges of $T$ in lexicographical order, and denote the vertices of $e_i$ by $\lr{v_i,p(v_i)}$.
Define
$$
\sT=(-1)^{\#\set{v\neq r}{v < p(v)}}  \cdot  \sgn \LR{v_1,\cdots,v_{n-1}}.
$$
It's easy to check $\sT$ is independent of the choice of the root. We have 
$$
\sT \bigwedge_{e\in T} \omega^{e}
= \prod_{i=1}^{n} \mu_i^{d^T_i-1}\omega^{\setn},
$$ 
$$
\sT\prod_{i=1}^n (z^i)^{d^T_i-1} \bigwedge_{e\in T} \omega^{e} 
= \prod_{i=1}^n \rho_i^{d^T_i-1}\omega^{\setn},
$$
and the sum over all $n^{n-2}$ maximal trees is
$$
\sum_{T\in \mathcal T_n} \sT \prod_{i=1}^n (z^i)^{d^T_i-1} \bigwedge_{e\in T} \omega^{e} 
= \LR{\sum_{i=1}^n \rho_i}^{n-2}\omega^{\setn} 
= \omega^{\setn}.
$$
\end{proof}

\begin{eg}\label{zhou-eg}
For $n=3$,   
$\omega^{\lr{1,2,3}}= \mu_1\bp\mu_2\bp\mu_3 - \mu_2\bp\mu_1\bp\mu_3 + \mu_3\bp\mu_1\bp\mu_2$, and 
\ali{
&    z_1(\mu_1\bp\mu_2-\mu_2\bp\mu_1)(\mu_1\bp\mu_3-\mu_3\bp\mu_1) + z_2(\mu_1\bp\mu_2-\mu_2\bp\mu_1)(\mu_2\bp\mu_3-\mu_3\bp\mu_2) + z_3(\mu_1\bp\mu_3-\mu_3\bp\mu_1)(\mu_2\bp\mu_3-\mu_3\bp\mu_2) \\
=& \, z_1\mu_1\omega^{\lr{1,2,3}} + z_2\mu_2\omega^{\lr{1,2,3}} +z_3\mu_3\omega^{\lr{1,2,3}}  
=\omega^{\lr{1,2,3}}.
}
\end{eg}

\begin{lem}\label{Zhou-lemma}
Let $A,B$ be non-negatively graded CDGAs with $H^i(A,d)=H^i(B,d)=0$ for $0<i<n-1$, and $f:B\to A$ be a quasi-isomorphism. Assume there exist degree $0$ elements $z^i\in A$ for $i=1,\cdots,n$, and degree $|I|-1$ elements $\omega^I\in A$ for each nonempty subset $I$ of $\lr{1,\cdots,n}$, which satisfy the relations
$$
d\omega^I = \sum_{j\notin I} (-1)^{(j,I)} z^j \omega^{I\cup \lr{j}},
\quad
\omega^{\setn}= \sum_{T\in \mathcal T_n} \sT \prod_{i=1}^n  (z^i)^{d^T_i-1} \bigwedge_{e\in T} \omega^{e}+ d(-).
$$
Moreover, $z^i$ and $\omega:=\omega^{\setn}$ are $d$-closed, which represent nontrivial cohomology classes in $H^\bullet(A,d)$. 

Then there exist $z^i_B,\omega^I_B\in B$ such that $z^i_B$ and $\omega_B:=\omega^{\setn}_B$ are $d$-closed, and
$$
f(z^i_B)=z^i, \quad f(\omega^{\setn}_B)=\omega^{\setn}+d(-), \quad
d\omega^I_B=\sum_{j\notin I} (-1)^{(j,I)} z^j_B \omega^{I\cup \lr{j}}_B, 
$$
$$
\omega^{\setn}_B
-\sum_{T\in \mathcal T_n} \sT \prod_{i=1}^n  (z^i_B)^{d^T_i-1} \bigwedge_{e\in T} \omega^{e}_B =d(-). 
$$   
\end{lem}

\begin{proof}
Let $[z^i_B], [\omega_B]$ be cohomology classes  in $H^\bullet(B,d)$ which are mapped to $[z^i], [\omega]$ in $H^\bullet(A,d)$. Since $A$ is non-negatively graded, we have $f(z^i_B)=z^i$. Moreover, $f(\omega^{\setn}_B)=\omega^{\setn}+d\beta$ for some $\beta$.

We construct $\omega^I_B$ by induction. Since 
$$
d\sum_{j\notin I} (-1)^{(j,I)} z^j_B \omega^{I\cup \lr{j}}_B
= \sum_{j\notin I} (-1)^{(j,I)} z^j_B \sum_{k\notin I\cup \lr{j}} (-1)^{(k,I\cup \lr{j})} z^k_B\omega^{I\cup \lr{j,k}}_B=0
$$
and $B$ has trivial $k$-th cohomology for $0<k<n-1$, we can find some $\omega^I_B$ such that
$$
d\omega^I_B=\sum_{j\notin I} (-1)^{(j,I)} z^j_B \omega^{I\cup \lr{j}}_B.
$$
Let 
$$
\widetilde \omega^{\setn}_B
=\sum_{T\in \mathcal T_n} \sT \prod_{i=1}^n  (z^i_B)^{d^T_i-1} \bigwedge_{e\in T} \omega^{e}_B.
$$
Note that $d\omega^{\lr{i,j}}_B=\sum_{k\ne i,j}\pm z^k\omega^{\lr{i,j,k}}$. Then
\ali{
d\widetilde \omega^{\setn}_B
=& d\sum_{T\in \mathcal T_n} \sT \prod_{i=1}^n  (z^i_B)^{d^T_i-1} \bigwedge_{e\in T} \omega^{e}_B \\
=& \sum_{T\in \mathcal T_n} \sT \prod_{i=1}^n  (z^i_B)^{d^T_i-1} \sum_{e\in T} \pm d\omega^{e}_B \bigwedge_{e'\in T, \,e'\ne e} \omega^{e'}_B \\
}
For each $i<j<k$, let us compute the term involving $\omega^{\lr{i,j,k}}$. Let $e_1,e_2,e_3$ be all edges whose vertices lies in $\lr{i,j,k}$. A maximal tree contains at most two of them since it does not have any cycle. There are two types of trees: 

\begin{itemize}
\item  
Type 1. The tree $T$ has exactly one edge $e\in \lr{e_1,e_2,e_3}$. If we remove the edge $e$ from $T$, we get a graph $T-e$ with two connected components, one of which contains two vertices in $\lr{i,j,k}$ and the other contains one.  Let $e'\in \lr{e_1,e_2,e_3}$ be the other edge such that $T-e+e'$ is connected. Then the total contribution from $T$ and $T-e+e'$ vanishes. 
See Figure \ref{fig:type1}. Here we used the dashed line to indicate two points in the same connected component of $T-e$.

\begin{figure}[H]
    \centering
\tikzset{every picture/.style={line width=0.75pt}} 

\begin{tikzpicture}[x=0.75pt,y=0.75pt,yscale=-1,xscale=1]

\draw  [fill={rgb, 255:red, 0; green, 0; blue, 0 }  ,fill opacity=1 ] (223.33,145.33) .. controls (223.33,143.12) and (225.12,141.33) .. (227.33,141.33) .. controls (229.54,141.33) and (231.33,143.12) .. (231.33,145.33) .. controls (231.33,147.54) and (229.54,149.33) .. (227.33,149.33) .. controls (225.12,149.33) and (223.33,147.54) .. (223.33,145.33) -- cycle ;
\draw  [fill={rgb, 255:red, 0; green, 0; blue, 0 }  ,fill opacity=1 ] (144.33,168.33) .. controls (144.33,166.12) and (146.12,164.33) .. (148.33,164.33) .. controls (150.54,164.33) and (152.33,166.12) .. (152.33,168.33) .. controls (152.33,170.54) and (150.54,172.33) .. (148.33,172.33) .. controls (146.12,172.33) and (144.33,170.54) .. (144.33,168.33) -- cycle ;
\draw  [fill={rgb, 255:red, 0; green, 0; blue, 0 }  ,fill opacity=1 ] (149.33,102.33) .. controls (149.33,100.12) and (151.12,98.33) .. (153.33,98.33) .. controls (155.54,98.33) and (157.33,100.12) .. (157.33,102.33) .. controls (157.33,104.54) and (155.54,106.33) .. (153.33,106.33) .. controls (151.12,106.33) and (149.33,104.54) .. (149.33,102.33) -- cycle ;
\draw    (153.33,102.33) -- (227.33,145.33) ;
\draw  [dash pattern={on 4.5pt off 4.5pt}]  (153.33,102.33) -- (148.33,168.33) ;
\draw  [fill={rgb, 255:red, 0; green, 0; blue, 0 }  ,fill opacity=1 ] (407,145.33) .. controls (407,143.12) and (408.79,141.33) .. (411,141.33) .. controls (413.21,141.33) and (415,143.12) .. (415,145.33) .. controls (415,147.54) and (413.21,149.33) .. (411,149.33) .. controls (408.79,149.33) and (407,147.54) .. (407,145.33) -- cycle ;
\draw  [fill={rgb, 255:red, 0; green, 0; blue, 0 }  ,fill opacity=1 ] (328,168.33) .. controls (328,166.12) and (329.79,164.33) .. (332,164.33) .. controls (334.21,164.33) and (336,166.12) .. (336,168.33) .. controls (336,170.54) and (334.21,172.33) .. (332,172.33) .. controls (329.79,172.33) and (328,170.54) .. (328,168.33) -- cycle ;
\draw  [fill={rgb, 255:red, 0; green, 0; blue, 0 }  ,fill opacity=1 ] (333,102.33) .. controls (333,100.12) and (334.79,98.33) .. (337,98.33) .. controls (339.21,98.33) and (341,100.12) .. (341,102.33) .. controls (341,104.54) and (339.21,106.33) .. (337,106.33) .. controls (334.79,106.33) and (333,104.54) .. (333,102.33) -- cycle ;
\draw    (332,168.33) -- (411,145.33) ;
\draw  [dash pattern={on 4.5pt off 4.5pt}]  (337,102.33) -- (332,168.33) ;

\draw (272,128) node [anchor=north west][inner sep=0.75pt]   [align=left] {$+$};
\draw (454,128) node [anchor=north west][inner sep=0.75pt]   [align=left] {$= \,\, 0$};
\draw (189.2,100.2) node [anchor=north west][inner sep=0.75pt]   [align=left] {$e$};
\draw (376.4,162.2) node [anchor=north west][inner sep=0.75pt]   [align=left] {$e'$};
\end{tikzpicture} 
    \caption{The total contribution from type 1 maximal trees vanishes.}
    \label{fig:type1}
\end{figure}
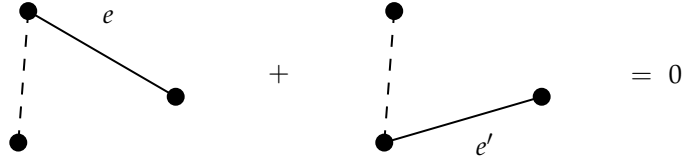

\item 
Type 2. A type 2 tree has two edges in $e_1,e_2,e_3$. If the two edges are removed, the graph  has three connected components, and there are three  maximal trees which contains it as a subgraph. The total contribution from them contains a factor
$$
z^i_B \LR{z^j_B\omega^{kij}_B\omega^{ij}_B - z^k_B\omega^{ki}_B\omega^{ijk}_B} + z^j_B\LR{z^k_B\omega^{ijk}_B\omega^{jk}_B - z^i_B\omega^{ij}_B\omega^{jki}_B} +z^k_B\LR{z^i_B\omega^{jki}_B\omega^{ki}_B - z^j_B\omega^{jk}_B\omega^{kij}_B}
= 0
$$
in $d\LR{z^i_B\omega^{ki}_B\omega^{ij}_B+z^j_B\omega^{ij}_B\omega^{jk}_B+z^k_B\omega^{jk}_B\omega^{ki}_B}$. 
See Figure \ref{fig:type2}. 
 
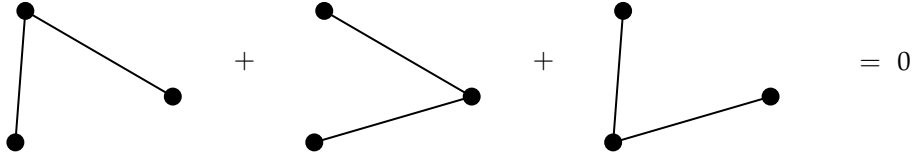
\begin{figure}[H]
    \centering  
\tikzset{every picture/.style={line width=0.75pt}} 

\begin{tikzpicture}[x=0.75pt,y=0.75pt,yscale=-1,xscale=1]

\draw  [fill={rgb, 255:red, 0; green, 0; blue, 0 }  ,fill opacity=1 ] (129,127) .. controls (129,124.79) and (130.79,123) .. (133,123) .. controls (135.21,123) and (137,124.79) .. (137,127) .. controls (137,129.21) and (135.21,131) .. (133,131) .. controls (130.79,131) and (129,129.21) .. (129,127) -- cycle ;
\draw  [fill={rgb, 255:red, 0; green, 0; blue, 0 }  ,fill opacity=1 ] (50,150) .. controls (50,147.79) and (51.79,146) .. (54,146) .. controls (56.21,146) and (58,147.79) .. (58,150) .. controls (58,152.21) and (56.21,154) .. (54,154) .. controls (51.79,154) and (50,152.21) .. (50,150) -- cycle ;
\draw  [fill={rgb, 255:red, 0; green, 0; blue, 0 }  ,fill opacity=1 ] (55,84) .. controls (55,81.79) and (56.79,80) .. (59,80) .. controls (61.21,80) and (63,81.79) .. (63,84) .. controls (63,86.21) and (61.21,88) .. (59,88) .. controls (56.79,88) and (55,86.21) .. (55,84) -- cycle ;
\draw    (59,84) -- (133,127) ;
\draw    (59,84) -- (54,150) ;
\draw  [fill={rgb, 255:red, 0; green, 0; blue, 0 }  ,fill opacity=1 ] (279,127) .. controls (279,124.79) and (280.79,123) .. (283,123) .. controls (285.21,123) and (287,124.79) .. (287,127) .. controls (287,129.21) and (285.21,131) .. (283,131) .. controls (280.79,131) and (279,129.21) .. (279,127) -- cycle ;
\draw  [fill={rgb, 255:red, 0; green, 0; blue, 0 }  ,fill opacity=1 ] (200,150) .. controls (200,147.79) and (201.79,146) .. (204,146) .. controls (206.21,146) and (208,147.79) .. (208,150) .. controls (208,152.21) and (206.21,154) .. (204,154) .. controls (201.79,154) and (200,152.21) .. (200,150) -- cycle ;
\draw  [fill={rgb, 255:red, 0; green, 0; blue, 0 }  ,fill opacity=1 ] (205,84) .. controls (205,81.79) and (206.79,80) .. (209,80) .. controls (211.21,80) and (213,81.79) .. (213,84) .. controls (213,86.21) and (211.21,88) .. (209,88) .. controls (206.79,88) and (205,86.21) .. (205,84) -- cycle ;
\draw    (209,84) -- (283,127) ;
\draw    (283,127) -- (204,150) ;
\draw  [fill={rgb, 255:red, 0; green, 0; blue, 0 }  ,fill opacity=1 ] (429,127) .. controls (429,124.79) and (430.79,123) .. (433,123) .. controls (435.21,123) and (437,124.79) .. (437,127) .. controls (437,129.21) and (435.21,131) .. (433,131) .. controls (430.79,131) and (429,129.21) .. (429,127) -- cycle ;
\draw  [fill={rgb, 255:red, 0; green, 0; blue, 0 }  ,fill opacity=1 ] (350,150) .. controls (350,147.79) and (351.79,146) .. (354,146) .. controls (356.21,146) and (358,147.79) .. (358,150) .. controls (358,152.21) and (356.21,154) .. (354,154) .. controls (351.79,154) and (350,152.21) .. (350,150) -- cycle ;
\draw  [fill={rgb, 255:red, 0; green, 0; blue, 0 }  ,fill opacity=1 ] (355,84) .. controls (355,81.79) and (356.79,80) .. (359,80) .. controls (361.21,80) and (363,81.79) .. (363,84) .. controls (363,86.21) and (361.21,88) .. (359,88) .. controls (356.79,88) and (355,86.21) .. (355,84) -- cycle ;
\draw    (354,150) -- (433,127) ;
\draw    (359,84) -- (354,150) ;

\draw (162,103) node [anchor=north west][inner sep=0.75pt]   [align=left] {$+$};
\draw (312,103) node [anchor=north west][inner sep=0.75pt]   [align=left] {$+$};
\draw (476,103) node [anchor=north west][inner sep=0.75pt]   [align=left] {$= \,\, 0$};

\end{tikzpicture}
    \caption{The total contribution from type 2 maximal trees vanishes.}
    \label{fig:type2}
\end{figure} 
\end{itemize}

Since $H^i(A,d)=0$ for $0<i<n-1$, we can inductively find $\beta^I\in A$ such that
$$
f\omega^I_B  
= \omega^I  +\sum_{j\notin I} (-1)^{(j,I)} z^j \beta^{I\cup\lr{j}} + d\beta^I.
$$ 
Let $\partial e$ denote the set of endpoints of an edge $e$. Then 
\ali{
  f\widetilde \omega^{I}_B  
=& \sum_{T\in \mathcal T_n} \sT \prod_{i=1}^n  (z^i)^{d^T_i-1} \bigwedge_{e\in T} f\omega^{e}_B \\
=& \sum_{T\in \mathcal T_n} \sT \prod_{i=1}^n  (z^i)^{d^T_i-1} \bigwedge_{e\in T} \LR{\omega^e +\sum_{j\notin \partial e} (-1)^{(j,\partial e)} z^j \beta^{\partial e \cup\lr{j}} + d\beta^e} \\
}
Similarly, the total contribution from all type 1 and type 2 graphs involving $\beta^{\partial e \cup\lr{j}}$ vanishes. As a result, 
\ali{
  f\widetilde \omega^{\setn}_B  
=& \sum_{T\in \mathcal T_n}\sT  \prod_{i=1}^n  (z^i)^{d^T_i-1} \bigwedge_{e\in T} \LR{\omega^e  + d\beta^e} \\
=& \sum_{T\in \mathcal T_n}\sT  \prod_{i=1}^n  (z^i)^{d^T_i-1} \bigwedge_{e\in T} \omega^e + \sum_{k=1}^{n-1} \sum_{i_1,\cdots,i_k} \sum_{\substack{T\in \mathcal T_n \\ e_{i_1},\cdots,e_{i_k}\in T}} \pm \sT \prod_{i=1}^n  (z^i)^{d^T_i-1}   \bigwedge_{j=1}^k d\beta^{e_j}  \bigwedge_{\substack{e\in T \\ e\ne e_{i_1},\cdots,e_{i_k}}} \omega^e \\
=& \omega + \sum_{k=1}^{n-1} \sum_{i_1,\cdots,i_k} \sum_{\substack{T\in \mathcal T_n \\ e_{i_1},\cdots,e_{i_k}\in T}} \pm \sT \prod_{i=1}^n  (z^i)^{d^T_i-1} d\LR{\beta^{e_1} \bigwedge_{j=2}^k d\beta^{e_j}} \bigwedge_{\substack{e\in T \\ e\ne e_{i_1},\cdots,e_{i_k}}} \omega^e \\
=& \omega + \sum_{k=1}^{n-1} \sum_{i_1,\cdots,i_k} \sum_{\substack{T\in \mathcal T_n \\ e_{i_1},\cdots,e_{i_k}\in T}} \pm \sT\prod_{i=1}^n  (z^i)^{d^T_i-1}  \beta^{e_1} \bigwedge_{j=2}^k d\beta^{e_j}  \,\, d\LR{\bigwedge_{\substack{e\in T \\ e\ne e_{i_1},\cdots,e_{i_k}}} \omega^e} +d(-)\\
=& \omega + \sum_{k=1}^{n-1} \sum_{i_1,\cdots,i_k} \beta^{e_1} \bigwedge_{i=2}^k d\beta^{e_i} \sum_{\substack{T\in \mathcal T_n \\ e_{i_1},\cdots,e_{i_k}\in T}} \pm \sT\prod_{i=1}^n  (z^i)^{d^T_i-1}  \sum_{\substack{e_0\in T \\ e_0\ne e_{i_1},\cdots,e_{i_k}}} d\omega^{e_0}    \bigwedge_{\substack{e\in T \\ e\ne e_0,e_{i_1},\cdots,e_{i_k}}} \omega^e +d(-)\\
=& \omega + d(-)\\
}
because the total contribution from all type 1 and type 2 graphs of $d\omega^{e_0}$ vanishes.  
Then 
$$
f\LR{[\omega^{\setn}_B - \widetilde \omega^{\setn}_B]}=[0],
$$
which implies $\omega^{\setn}_B - \widetilde \omega^{\setn}_B=d(-)$.
\end{proof} 

\begin{proof}[Proof of Proposition \ref{Cn-nonformal}]
Suppose there is a chain of quasi-isomorphisms of non-negatively graded CDGAs 
$$
(A^\bullet(\C^n\punctured ), \bp)=A_1 \xleftarrow{f_1} B_1 \xrightarrow{g_1} A_2 \xleftarrow{f_2} \cdots  \xrightarrow{g_{k-1}} A_{k}\xleftarrow{f_k} B_{k} \xrightarrow{g_k} A_{k+1}=(H^\bullet(\C^n\punctured),0).
$$
We can lift the system of elements from $A_i$ to $B_i$ using Lemma \ref{Zhou-lemma}, and push  them forward along $g_i$. Finally, we find 
$$
\widetilde \omega^{\setn}_{A_{k+1}}=\sum_{T\in \mathcal T_n} \sT \prod_{i=1}^n  (z^i_{A_{k+1}})^{d^T_i-1} \bigwedge_{e\in T} \omega^{e}_{A_{k+1}} 
$$ 
in $A_{k+1}=(H^\bullet(\C^n\punctured),0)$. However, the degree $1$ part of $A_{k+1}$ is ${0}$, so $\omega^{e}_{A_{k+1}}=0$, which implies $\widetilde \omega^{\setn}_{A_{k+1}}=0$, a contradiction. 
Similarly, $\A(\C^n \punctured )$ is non-formal. 
\end{proof}

\subsection{Non-formality for \texorpdfstring{$n'=1$, $n=1$}{n'=1, n=1}}\label{11}
\begin{prop}\label{11-nonformal}
$\A(\R\times \C \punctured )$ and $\Adef(\R\times \C \punctured )$ are non-formal. 
\end{prop} 

Here we have  $\rho_1=\frac{x^2}{x^2+2z\bar z}$, $\rho_2=\frac{2z\bar z}{x^2+2z\bar z}$ and $\mu_1=\frac{x}{(x^2+2z\bar z)^{1/2}}$, $\mu_2=\frac{2\bar z}{x^2+2z\bar z}$. Then we have
\ali{
\omega^{1}
= \mu_1\bd \mu_2 - 2\mu_2\bd \mu_1  
= 3\mu_1\bd \mu_2+\bd (-), \\
}
\ali{
z\omega^{1}
=\frac{1}{\mu_1}(\rho_1\bd \rho_2-\rho_2\bd \rho_1)  
=-\frac{1}{\mu_1}\bd \rho_1 
=-2 \bd  \mu_1
}
\ali{
\mu_1\omega^1
=\frac{1}{z}(\rho_1\bd \rho_2-\rho_2\bd \rho_1) 
=\frac{1}{z} \bd \rho_2 
= \bd \mu_2
}

\begin{lem}
Let $A,B$ be non-negatively graded CDGAs, and let $f:B\to A$ be a quasi-isomorphism. Assume there exist $z, \mu_1,\mu_2\in A^0$ and $\omega^1\in A^1$ such that 
$$
\omega^1=3\mu_1 d\mu_2 +d(-), \quad
z\omega^{1}= -2d\mu_1, \quad
\mu_1\omega^1=d\mu_2.
$$
Moreover, $z,\omega^1$ are $d$-closed and represent nontrivial cohomology classes in $H^\bullet(A)$. 

Then there exist elements $z_B, \mu_{1,B},\mu_{2,B},\omega^1_B \in B$ such that $z_B,\omega^1_B$ are $d$-closed, and
$$
f(z_B)=z, \quad 
f(\omega^1_B)=\omega^1+d(-), \quad
z_B\omega^{1}_B= -2d\mu_{1,B}, \quad
\mu_{1,B}\omega^1_B=d\mu_{2,B}, \quad
\omega^1_B=3\mu_{1,B} d\mu_{2,B} +d(-).
$$ 
\end{lem}

\begin{proof}
Let $[z_B],[\omega^1_B]$ be cohomology classes  in $H^\bullet(B,d)$ which are mapped to $[z], [\omega^1]$ in $H^\bullet(A,d)$. Since $A$ is non-negatively graded, we have $f(z_B)=z$. Moreover, $f(\omega^{1}_B)=\omega^{1}+d\beta$ for some $\beta$.

Because $d(z_B\omega^{1}_B)=0$ and $f([z_B\omega^{1}_B])=[z(\omega^{1}+d\beta)]=[0]$, we have $[z_B\omega^{1}_B]=0$.  Hence  $z_B\omega^{1}_B= -2d \widetilde\mu_{1,B}$ for some $\widetilde\mu_{1,B}$. We have 
$$
d(f\widetilde\mu_{1,B}-\mu_1+\frac{1}{2}z\beta)
=fd\widetilde\mu_{1,B}-d\mu_1+\frac{1}{2}zd\beta
=-\frac{1}{2}f(z_B\omega^{1}_B)+\frac{1}{2}z\omega^1+\frac{1}{2}zd\beta
=0
$$
so $\gamma:=f\widetilde\mu_{1,B}-\mu_1+\frac{1}{2}z\beta$ is a $d$-closed element in $A$. We can find some $[\gamma_B]\in H^\bullet(B,d)$ such that $f([\gamma_B])=[\gamma]$. We have $f(\gamma_B)=\gamma$ because $B$ is non-negatively graded. Let $\mu_{1,B}=\widetilde\mu_{1,B} - \gamma_B$. Then 
$$
f\mu_{1,B}
=f\widetilde\mu_{1,B} - f\gamma_B
=\LR{\gamma+\mu_1-\frac{1}{2}z\beta}-\gamma
=\mu_1-\frac{1}{2}z\beta,
$$
$$
d\mu_{1,B}
=d\widetilde\mu_{1,B} - d\gamma_B
=-\frac{1}{2}z_B\omega^{1}_B.
$$
We have $d(\mu_{1,B}\omega^{1}_B)=-\frac{1}{2}z_B\omega^{1}_B\omega^{1}_B=0$, and 
\ali{
f(\mu_{1,B}\omega^{1}_B) 
&=\LR{\mu_1-\frac{1}{2}z\beta}\LR{\omega^{1}+d\beta} \\
&=\mu_1\omega^{1} + \mu_1d\beta -\frac{1}{2}z\beta\omega^{1} -\frac{1}{2}z\beta d\beta \\
&=d\mu_2 + \mu_1d\beta +\beta d\mu_1  -\frac{1}{2}z\beta d\beta \\
&=d\LR{\mu_2+\mu_1\beta-\frac{1}{4}z\beta^2}.
} 
So $[\mu_{1,B}\omega^{1}_B]=[0]$ and we can find $\mu_{1,B}\omega^{1}_B=d\mu_{2,B}$ for some $\mu_{2,B}$.  

Let $\widetilde \omega^1_B=3\mu_{1,B}d\mu_{2,B}$, we see 
$$
d\widetilde \omega^1_B= 3d\mu_{1,B}d\mu_{2,B}=3\LR{-\frac{1}{2}z_b\omega^{1}_B}(\mu_{1,B}\omega^1_B)=0.
$$
Denote $\omega^1=3\mu_1 d\mu_2 +d\alpha$, then we have 
\ali{
f\widetilde \omega^1_B
&=3f\mu_{1,B}fd\mu_{2,B} \\
&=3\LR{\mu_1-\frac{1}{2}z\beta} d\LR{\mu_2+\mu_1\beta-\frac{1}{4}z\beta^2} \\
&=\omega^1+d\LR{-\alpha+3\mu_1^2\beta-\frac{3}{2}\mu_1z\beta^2+\frac{1}{4}z^2\beta^3}.
}
So $\omega^1_B=\widetilde \omega^1_B+d(-)$.
\end{proof}

\begin{proof}[Proof of Proposition \ref{11-nonformal}]
If $(\A(\R\times \C\punctured ), \d )$ is linked to  $A_{k+1}=(H^\bullet(\R\times \C\punctured),d=0)$ by  quasi-isomorphisms of non-negatively graded CDGAs, then we find $[\omega^1_{A_{k+1}}]=[3\mu_{1,{A_{k+1}}}d\mu_{2,{A_{k+1}}}]=[0]$, contradiction. Similarly, $\Aconk(\R\times \C\punctured )$ is non-formal.
\end{proof}

\subsection{Non-formality for \texorpdfstring{$n'=1$, $n\ge2$}{n'=1, n≥2}}\label{1>1}

Throughout this section we fix $n\ge2$, and set
$$
R=\C[z^1,\ldots,z^n],\qquad
I=(z^1,\ldots,z^n),\qquad
k=R/I,
$$
and write
$$
\HH=H^\bullet(J_{1,n}) 
=R\oplus H_I^n(R)[-n], \qquad 
\HH^0=R, \quad \HH^n=H_I^n(R)[-n].
$$
where $J_{1,n}$ is the mixed Jouanolou model introduced in Section \ref{sec:chain-level}, and $H_I^n(R)$ is the local cohomology of $R$ with support in $I$. 
\begin{lemma}\label{lemma-R-zigzag}
Suppose that $J_{1,n}$ and $\HH$ are connected by a zigzag of quasi-isomorphisms of CDGAs over $\mathbb C$.  Then there exist
\begin{itemize}
\item a CDGA $Q$,
\item CDGA quasi-isomorphisms $p:Q\to J_{1,n}$, $g:Q\to \HH$,
\item a CDGA map $h:R\to Q$,
\item an automorphism $\sigma: R\to R$,
\end{itemize}
such that 
\begin{itemize}
\item $p$ is $R$-linear,
\item $g$ is $R$-linear after precomposing the standard $R$-action on $\HH$ with $\sigma$,
\item $\sigma(I)=I$,
\item the following diagram commutes:
\begin{center}
\begin{tikzcd}
{J_{1,n}} & Q \arrow[l, "p"'] \arrow[r, "g"]                       & \HH                \\
  & R  \arrow[r, "\sigma"'] \arrow[u, "h"] & \HH^0=R \arrow[u, hook]
\end{tikzcd}
\end{center}
\end{itemize}
\end{lemma}

\begin{proof}
The map $\mathbb C\to R$ is a cofibration in the model category of CDGAs because $R$ is freely generated as a commutative algebra by degree-zero elements. Choose a cofibrant replacement
$$
p:Q\lra J_{1,n}
$$
which is a trivial fibration. There exists a quasi-isomorphism $g:Q\to \HH$ by the zigzag.
	
The lifting problem
\begin{center}
\begin{tikzcd}
\C \arrow[d] \arrow[r]         & Q \arrow[d, "p"] \\
R \arrow[r, "j"', hook] \arrow[ru, "h", dashed] & {J_{1,n}}       
\end{tikzcd}
\end{center}
has a solution $h:R\to Q$, so $p$ is $R$-linear. Moreover, both $H^0(j)$ and $H^0(p)$ are isomorphisms, so $H^0(h)$ is an isomorphism. Composing with the isomorphism $H^0(g)$, we get the isomorphism $\sigma:R\to R$. 
	
Put $M=H_I^n(R)$. Every \v Cech monomial $\frac{1}{(z^1)^{k_1}\cdots (z^n)^{k_n}}$ in $M$ is annihilated by a power of every $z^i$, so we have 
$$
\operatorname{Supp}_R(M)=\{I\}.
$$
Let ${}_{\sigma}M$ be the $R$-module with the $\sigma$-twisted action, i.e., $r\cdot m=\sigma(r)m$. 
Note that $H^n(g)\circ H^n(p)^{-1}: M \to {}_{\sigma}M$ is an isomorphism. Consequently, we have 
$$
\{\sigma^{-1}(I)\}
=\operatorname{Supp}_R({}_{\sigma}M)
=\operatorname{Supp}_R(M)
=\{I\}.
$$
Thus $\sigma(I)=I$.  
\end{proof}

\begin{prop}\label{prop:formal-1>1}
We have
\ali{
H^\bullet(k\otimes_R^{\mathbf L}J_{1,n}) &\cong k\times k, \\
H^\bullet(k\otimes_R^{\mathbf L} \HH ) &\cong k[\epsilon]/(\epsilon^2)
}
as graded algebras concentrated in degree $0$. 
Consequently, $J_{1,n}$ and $\Hconk(\R\times \C^n \punctured)$ are non-formal.
\end{prop}

\begin{proof}
Set
$$
B=R[y,w^1,\ldots,w^n]\Big/\left(y^2+\sum_i z^iw^i-1\right).
$$
Then $J_{1,n}=\Omega^\bullet_{B/R}$. The smoothness of $R\to B$ implies 
\ali{
k\otimes_R^{\mathbf L}J_{1,n} 
&\cong  k\otimes_RJ_{1,n} 
\cong \frac{k[y,w^1,\ldots,w^n,dw^1,\ldots,dw^n]}{(y^2-1)} \\
& \cong k[w^1,\ldots,w^n,dw^1,\ldots,dw^n] \otimes_k k[y]/(y^2-1) \\
&\cong k[w^1,\ldots,w^n,dw^1,\ldots,dw^n] \otimes_k (k\times k)
}  
has cohomology $k\times k$.

Let  
$$
K=R[\xi_1,\ldots,\xi_n],
\qquad |\xi_i|=-1,\qquad d\xi_i=z^i
$$
be the Koszul complex, which is a semi-free commutative $R$-CDGA resolving $k$. Set $R_i=\C[z^i]$, $K_i=R_i[\xi_i]$, $d\xi_i=z^i$ and $M_i=\oplus_{j\ge1} k[(z^i)^{-j}]$, so 
$$
H_I^n(R)\cong M_1\otimes_k\cdots\otimes_kM_n. 
$$
Then 
\ali{
H^\bullet \LR{ k\otimes_R^{\mathbf L} \LR{H_I^n(R)}} 
&\cong H^\bullet \LR{K\otimes_R M_1\otimes_k\cdots\otimes_kM_n} \\
&\cong \bigotimes_{i=1}^{n} H^\bullet (K_i\otimes_{R_i} M_i) 
\cong \bigotimes_{i=1}^{n} k \\
&\cong k.
} 
Thus we have  $H^\bullet(k\otimes_R^{\mathbf L} \HH ) \cong k[\epsilon]/(\epsilon^2)$ with $\epsilon = \Lr{\frac{\xi_1\cdots\xi_n}{z^1\cdots z^n}}$. 
Moreover, if the action is twisted by an automorphism $\sigma$ preserving $I$, then $\sigma$ preserves the residue augmentation $R\to k$.  Replacing the Koszul generators $z^i$ by $\sigma(z^i)$ therefore gives an isomorphic computation.

Suppose $J_{1,n}$ is formal, then applying $K\otimes_R(-)$ to the quasi-isomorphism in Lemma \ref{lemma-R-zigzag} yields an algebra isomorphism $k\times k \cong  k[\epsilon]/(\epsilon^2)$, contradiction.
\end{proof}

\begin{prop}\label{prop:formal-1>1-2}
The CDGA $A_{1,n}:=
A^\bullet
\LR{\mathbb R\times\mathbb C^n\setminus\{0\}}$ is non-formal.
\end{prop}

\begin{proof}[Sketch of proof]
Let $E=H^0(A_{1,n})$. If $A_{1,n}$ is connected to $(H^\bullet(A_{1,n}),0)$ by a zigzag of quasi-isomorphisms of nonnegatively graded CDGAs, then each CDGA in the zigzag contains a copy of $E$, and the $E$-CDGA structure can be transported through the zigzag so that every arrow is $E$-linear, up to an automorphism
$$
\sigma:E\lra E
$$
which fixes the maximal ideal $\mathfrak m_0=(z^1,\cdots,z^n)$. Define $k_0=E/\mathfrak m_0$ and  
$$
K_E=E[\xi_1,\ldots,\xi_n], \qquad |\xi_i|=-1,\qquad d\xi_i=z^i.
$$
We have 
\ali{
H^\bullet(K_E\otimes_EA_{1,n}) &\cong k_0\times k_0, \\
H^\bullet\!\left(
K_E\otimes_EH^\bullet(A_{1,n})
\right)
&\cong k_0[\varepsilon]/(\varepsilon^2).
} 
These algebras are not isomorphic, thus $A_{1,n}$ is non-formal.
\end{proof}

\subsection{Formality for \texorpdfstring{$n'\ge2$, $n\ge1$}{n'≥2, n≥1}}\label{>1>0}

\begin{prop}\label{formality->1>0}
For $n'\ge2$, $n\ge 1$, $(\A(\paka), \d )$ and $(\Adef(\paka), \d )$ are formal. 
\end{prop} 

\begin{proof}
Recall 
$\widehat I=\lr{1,\cdots,n'} \cup (n'+I)$ and
\ali{
\omega^I
=& \sum_{i\in \widehat I} \pm \mu_i \prod_{j\in \widehat I \setminus \lr{i}} \bd \mu_j \phi(j) \\
=& (\mu_{1} \bd \mu_{2} - \mu_{2}\bd \mu_{1}) (-) + (\bd \mu_{1}\bd \mu_{2}) (-).
}
We have 
\ali{
& (\mu_{1} \bd \mu_{2} - \mu_{2}\bd \mu_{1})^2 \\
=& (\mu_{1} \bd \mu_{2} - \mu_{2}\bd \mu_{1})(\bd \mu_{1}\bd \mu_{2}) \\
=& (\bd \mu_{1}\bd \mu_{2})^2=0.
}
Consequently, $\omega^I\omega^J=0$ for any $I,J$. Let $J_{n',n}'$ be the differential graded subalgebra of $J_{n',n}$ generated by $z^i$ and $\partial^I\omega_J$. 

Let $(C^\bullet(\U,A^0(\paka)), \delta)$ be the \v Cech complex defined in Section  \ref{sec:local-coho} and 
\ali{
&(\tilde  C^\bullet(\U,A^0(\paka)), \delta): \\
& \qquad\qquad  0 \lra \ker \delta_0 \lra  C^0(\U,A^0(\paka)) \overset{\delta_0}{\lra} C^1(\U,A^0(\paka))  \overset{\delta_1}{\lra} \cdots
}
be the augmented complex. Let 
$$
\widetilde\gamma_I=\frac{1}{z^{i_1}\cdots z^{i_k}} \in \OO\LR{U_{i_1}\cap \cdots \cap U_{i_k}}, \quad \text{for } I=\lr{i_1,\cdots ,i_k},
$$ 
$$
\widetilde\gamma_{\emptyset}= \LR{1}_{1\le i\le n} \in \bigoplus_{i=1}^n \OO(U_i). 
$$ 
Then the cohomology of the augmented complex is given by 
$$
H^\bullet(\tilde C^\bullet(\U,A^0(\paka)), \delta)=\begin{cases}
\C[\PPi ]\widetilde\gamma_{\setn} & \bullet =n-1 \\
0 & \text{else}
\end{cases}
$$  
and there is an isomorphism of chain complexes 
\ali{
(\tilde C^\bullet(\U,A^0(\paka)), \delta)[-n'] &\lra (J_{n',n}'^{\bullet>0}, \bfd) \\
\widetilde\gamma_I & \lmt \omega^I
}   
Thus, we have
\ali{
H^\bullet(J_{n',n}'^\bullet, \bfd) 
&= H^0(J_{n',n}'^\bullet, \bfd)  \oplus H^\bullet(J_{n',n}'^{\bullet>0}, \bfd) \\
&= \C[\ZZi ] \oplus \C[\PPi ]\omega^{\setn} \\
&= H^\bullet((\paka).
} 
Moreover, the embedding 
$$
(J_{n',n}'^\bullet, \bfd) \islra (J_{n',n}^\bullet, \bfd) 
$$
is a quasi-isomorphism. Moreover, when modulo the differential ideal  of $J_{n',n}'^\bullet$ generated by $\partial^I\omega_J$ for $J\ne \lr{1,\cdots, n}$, we get a quasi-isomorphism
$$
(J_{n',n}'^\bullet, \bfd) \slra 
H^\bullet(J_{n',n}'^\bullet, \bfd) =
\Hdef(\paka).
$$
Since the embedding 
$$
(J_{n',n}^\bullet,\bfd) \islra (\Adef(\paka),\d)
$$ 
is a quasi-isomorphism, we see that $\Adef(\paka)$ is formal.  Moreover, we can complete $J_{n',n}$ using the nuclear topology, and a similar argument shows $\A(\paka)$ is formal. 
\end{proof}

\begin{rem} 
$\Adef(\paka)$ is also formal as a CDGA in the category of $D_{\C^n}$-modules, as all quasi-isomorphisms in the proof of Proposition \ref{formality->1>0} preserve the  $D_{\C^n}$-module structure. 
\end{rem}

\section{Cohomology of Configuration Spaces}\label{sec:cohomology-conf}


The plan of this section is as follows.
\begin{itemize}
\item \ref{sec:cohomology-groups}: We compute the cohomology groups $H^\bullet(\maka)$ and $\Hdef(\maka)$.
\item \ref{sec:ring-structure-conf}: We determine the ring structures on $H^\bullet(\maka)$ and $\Hdef(\maka)$.
\end{itemize}


\subsection{Cohomology Groups}\label{sec:cohomology-groups}

Let $\phi_{ij}$ be the map
\ali{
\phi_{ij}: \quad \maka &\lra \paka \\
  (\ZZ_1,\cdots,\ZZ_m)       &\lmt \ZZ_i-\ZZ_j
}
and denote $\omega_{ij}=\phi_{ij}^*\omega$. 

\begin{thm}\label{thm:conf-space-cohomomogy}
The cohomology group of $\A(\maka)$ is given by
\ali{
&H^\bullet(\maka) \\
=&\hotimes_{j=1}^m \LR{\Hol(\C^n_j)\oplus \oplus_{i=1}^{j-1} \overline{\C[\partial_{z^1_j},\cdots, \partial_{z^n_j}]\omega_{ij}}} \\
=&\oplus_{k=0}^{m-1}\oplus_{\substack{1< j_1<\cdots<j_{k}\le m\\ i_1<j_1,\cdots,i_k<j_k}} \LR{\hotimes_{l\ne j_1,\cdots,j_k}\Hol(\C^n_l)} \hotimes \LR{\hotimes_{l=1}^k\overline{\C[\partial_{z_{j_l}^1},\cdots, \partial_{z_{j_l}^n}]\omega_{i_lj_l}}}
}
where $\Hol(\C^n_i)$ denotes the space of holomorphic functions of $\zz_i\in \C^n$, 
and the cohomology group of \\$\Adef(\maka)$ is
\ali{
&\Hdef(\maka) \\
=&\otimes_{j=1}^m \LR{\C[z_j^1,\cdots,z_j^n]\oplus \oplus_{i=1}^{j-1} {\C[\partial_{z^1_j},\cdots, \partial_{z^n_j}]\omega_{ij}}} \\
=&\oplus_{k=0}^{m-1}\oplus_{\substack{1< j_1<\cdots<j_{k}\le m\\ i_1<j_1,\cdots,i_k<j_k}} \LR{\otimes_{l\ne j_1,\cdots,j_k}\C[z_l^1,\cdots,z_l^n]} \otimes \LR{\otimes_{l=1}^k{\C[\partial_{z_{j_l}^1},\cdots, \partial_{z_{j_l}^n}]\omega_{i_lj_l}}}
}
which is naturally identified as a dense subspace of the former.
\end{thm}

\begin{proof}
In this proof,  $\OO$ denotes one of the following sheaves
\begin{itemize}
\item $\Hol$: the sheaf of smooth $\bd$-closed functions on a 
manifold;
\item $\Ocon$: the sheaf of constructible $C^k$ $\bd$-closed functions, which is viewed as a sheaf on a mixed real-complex algebraic variety with the Zariski topology.
\end{itemize}
$\Hol$ or $\Ocon$. We also denote $Y=\R^{n'} \times \C^n$. Let
\ali{
\pi: \Conf_m(Y) &\lra \Conf_{m-1}(Y) \\
(\ZZ_1,\cdots,\ZZ_m) & \lmt (\ZZ_1,\cdots,\ZZ_{m-1})
}
be the projection onto the first $m-1$ components. Also let 
$$
p: \Conf_{m-1}(Y) \times Y \lra \Conf_{m-1}(Y)
$$
be the projection onto the first component. 

$\Conf_m(Y)$ is an open subset of $\Conf_{m-1}(Y) \times Y$
\ali{
j: \Conf_m(Y) &\hlra \Conf_{m-1}(Y) \times Y \\
(\ZZ_1,\cdots,\ZZ_m) & \lmt ((\ZZ_1,\cdots,\ZZ_{m-1}),\ZZ_m)
}
and its complement $Z=\Conf_{m-1}(Y) \times Y \setminus \Conf_m(Y)$ is the union of $m-1$ diagonals
$$
Z= \bigsqcup_{i=1}^{m-1} \Delta_i, \quad \Delta_i=\set{((\ZZ_1,\cdots,\ZZ_{m-1}),\ZZ_m) \in \Conf_{m-1}(Y) \times Y}{\ZZ_i=\ZZ_m}.
$$ 
There is a distinguished triangle associated with the open immersion $j:\Conf_m(Y) \hookrightarrow \Conf_{m-1}(Y) \times Y$ and its complement $Z$ 
$$
R\Gamma_Z\OO_{\Conf_{m-1}(Y) \times Y} \to \OO_{\Conf_{m-1}(Y) \times Y} \to Rj_* \OO_{\Conf_m(Y)} \overset{+1}{\to}.
$$
Apply $Rp_*$ to the above triangle, we get
$$
Rp_*R\Gamma_Z\OO_{\Conf_{m-1}(Y) \times Y} \to Rp_*\OO_{\Conf_{m-1}(Y) \times Y} \to Rp_*Rj_* \OO_{\Conf_m(Y)} \overset{+1}{\to}
$$
Since $Y$ is affine, the morphism $p$ is affine, so the first term becomes 
$$
Rp_*R\Gamma_Z\OO_{\Conf_{m-1}(Y) \times Y} \cong Rp_*\Gamma_Z\OO_{\Conf_{m-1}(Y) \times Y}
$$  
Moreover, the second term 
$$
Rp_*\OO_{\Conf_{m-1}(Y) \times Y}\cong \OO_{\Conf_{m-1}(Y)} \otimes \Gamma(Y,\OO_Y)
$$
is concentrated in dimension $0$.  Since $\pi=p\circ j$, the third term is 
$$
Rp_* Rj_*\OO_{\Conf_m(Y)}  \cong  R\pi_* \OO_{\Conf_m(Y)}.
$$

Since $\Delta_i$ are pairwise disjoint, we have
$$
R\Gamma_Z(\OO_{\Conf_{m-1}(Y) \times Y}) \cong \bigoplus_{i=1}^{m-1} R\Gamma_{\Delta_i}(\OO_{\Conf_{m-1}(Y) \times Y}).
$$
Each term in the summand is supported on the closed subvariety $\Delta_i$. We have $R^q\Gamma_{\Delta_i}(\OO_{\Conf_{m-1}(Y) \times Y})=0$ unless $q= n'+n$. Moreover, each $R^{n'+n}\Gamma_{\Delta_i}(\OO_{\Conf_{m-1}(Y) \times Y})$ is a locally free sheaf of infinite rank on $\Conf_{m-1}(Y)$. If we denote $\ZZ_{im}=\ZZ_i-\ZZ_m=(z^{im}_1,\cdots,z^{im}_n)$, then a basis of each fiber is, by Proposition \ref{prop:coh-local},  
\begin{itemize}
\item For $\OO=\Hol$: $\Hol(\C^n_m)\oplus \oplus_{i=1}^{m-1} \overline{\C[\partial_{z^1_i},\cdots, \partial_{z^n_i}]\omega_{im}}$
\item For $\OO=\Ocon$: $\C[z_m^1,\cdots,z_m^n]\oplus \oplus_{i=1}^{m-1} {\C[\partial_{z^1_i},\cdots, \partial_{z^n_i}]\omega_{im}}$
\end{itemize}
This completely describes $R^q\pi_*\OO_{\Conf_m(Y)}$.

Apply the Leray spectral sequence to the projection $\pi: \Conf_m(Y) \to \Conf_{m-1}(Y)$ 
$$
E_2^{p,q} = H^p(\Conf_{m-1}(Y), R^q\pi_*\OO_{\Conf_m(Y)}) \Longrightarrow H^{p+q} (\Conf_m(Y), \OO_{\Conf_m(Y)}),
$$
we see $E_2^{p,q}= 0$ unless $q=0,n'+n-1$. Since generators in $E_2^{0,n'+n-1}$ are globally defined cohomology classes on $\Conf_m(Y)$, the spectral sequence degenerates at the $E_2$ page. 
\end{proof}

\begin{rem}
$\Hdef(\maka)$ carries a natural $D_{\C^{mn}}$-module structure.
\end{rem}

\subsection{Ring Structures}\label{sec:ring-structure-conf}

Clearly, $H^\bullet(\maka)$ is topologically generated as a ring by the cohomology classes in $\Hol(\C^{mn})$ and $\overline{\C[ \partial_{z_j^1},\cdots, \partial_{z_j^n}]\omega_{ij}}$ for all $i<j$. Let 
$$
\lR{\Hol(\C^{mn}),\,\, \overline{\C[ \partial_{z_j^1},\cdots, \partial_{z_j^n}]\omega_{ij}}} 
$$
be the topological subring of $\A(\maka)$ generated by these differential forms, which is a free ring modulo the ideal generated by $\LR{\overline{\C[ \partial_{z_j^1},\cdots, \partial_{z_j^n}]\omega_{ij}}}^2$ for each $i<j$. 
There is a natural surjective map of graded topological rings
\begin{equation*}\label{equ:coho-generators}
\lR{\Hol(\C^{mn}),\,\, \overline{\C[ \partial_{z_j^1},\cdots, \partial_{z_j^n}]\omega_{ij}}} 
\lra H^\bullet(\maka).    \tag{*} 
\end{equation*} 
We saw in Proposition \ref{prop:local-cohomology} that in $H^\bullet (\paka)$, the following relation holds:
$$
\zz^{\xalpha}\partial^{\xbeta}\omega=
\begin{cases}
(-1)^{\alpha_1+\cdots +\alpha_n} \prod_{i=1}^n \frac{\beta_i!}{(\beta_i-\alpha_i)!} \, \partial^{\xbeta-\xalpha}\omega & \alpha_i\le \beta_i \text{ for all } i \\
0 & \alpha_i>\beta_i \text{ for some } i \\
\end{cases}.
$$ 
Pulling back along the projection 
\ali{
\phi_{ij}: \quad \maka &\lra \paka \\
  (\ZZ_1,\cdots,\ZZ_m)       &\lmt \ZZ_i-\ZZ_j
}
gives, in $H^\bullet (\maka)$,
$$
\zz_{ij}^{\xalpha}\partial_{\zz_j}^{\xbeta}\omega_{ij}=
\begin{cases}
(-1)^{\alpha_1+\cdots +\alpha_n} \prod_{i=1}^n \frac{\beta_i!}{(\beta_i-\alpha_i)!} \, \partial_{\zz_j}^{\xbeta-\xalpha}\omega_{ij} & \alpha_i\le \beta_i \text{ for all } i \\
0 & \alpha_i>\beta_i \text{ for some } i \\
\end{cases}.
$$ 
Here $\zz_{ij}=\zz_i-\zz_j$ and   $\partial_{\zz_j}^{\xbeta}=\partial_{z_j^1}^{\beta_1}\cdots \partial_{z_j^n}^{\beta_n}$. We have used $\partial_{z_i^k}\omega_{ij}=-\partial_{z_j^k}\omega_{ij}$.

\begin{rem}\label{rem:Arnold}
As in the real case, we have Arnold relations. Heuristically speaking, $\bd\omega_{ij}=\delta_{ij}$ is a delta distribution, and we have 
\ali{
& \bd \int_{\ZZ_l} d\zz_l\,  \omega_{li}\omega_{lj}\omega_{lk} \\
=& \int_{\ZZ_l} d\zz_l \, (\delta_{li}\omega_{lj}\omega_{lk}+\delta_{lj}\omega_{lk}\omega_{li}+\delta_{lk}\omega_{li}\omega_{lj}) \\
=& \omega_{ij}\omega_{jk}+\omega_{jk}\omega_{ki}+\omega_{ki}\omega_{ij},
} 
so $\omega_{ij}\omega_{jk}+\omega_{jk}\omega_{ki}+\omega_{ki}\omega_{ij}$ is $\bd$-exact. 
In Section \ref{sec:Regularized-Integrals}, we construct the regularized integral $\dashint$ to give a rigorous proof. In general, for $f(\partial_{\zz_i})\omega_{li} \in \overline{\C[ \partial_{z_i^1},\cdots, \partial_{z_i^n}]\omega_{li}}$ and $g(\partial_{\zz_j})\omega_{lj}$, $h(\partial_{\zz_k})\omega_{lk}$ similarly defined, we have 
\ali{
& \bd \dashint_{\ZZ_l} d\zz_l\,  f(\partial_{\zz_i})\omega_{li}g(\partial_{\zz_j})\omega_{lj}h(\partial_{\zz_k})\omega_{lk} \\
=& g(\partial_{\zz_j})(f(\partial_{\zz_i})\omega_{ij}h(\partial_{\zz_k})\omega_{jk})+h(\partial_{\zz_k})(g(\partial_{\zz_j})\omega_{jk}f(\partial_{\zz_i})\omega_{ki})+f(\partial_{\zz_i})(h(\partial_{\zz_k})\omega_{ki}g(\partial_{\zz_j})\omega_{ij}).
} 
\end{rem}

Using the basis given in Theorem \ref{thm:conf-space-cohomomogy}, we see the kernel of (\ref{equ:coho-generators}) is precisely the one generated by $\zz_{ij}^{\xalpha}\partial_{\zz_j}^{\xbeta}\omega_{ij}$ with $\alpha_k>\beta_k$ for some $k$, and the Arnold relations. In conclusion, we have

\begin{thm} \label{thm:coho-ring-conf}
$H^\bullet(\maka)$ is topologically generated as a ring by 
$$
\Hol(\C^{mn})\oplus \oplus_{i<j} \overline{\C[\partial_{z^1_j},\cdots, \partial_{z^n_j}]\omega_{ij}},
$$ 
subject to the following relations:
\begin{itemize}
\item $f(\partial_{\zz_j})\omega_{ij}\cdot g(\partial_{\zz_j})\omega_{ij}=0$, for $f(\partial_{\zz_j})\omega_{ij}, g(\partial_{\zz_j})\omega_{ij}\in \overline{\C[\partial_{z^1_j},\cdots, \partial_{z^n_j}]\omega_{ij}}$,
\item $ 
\zz_{ij}^{\xalpha}\partial_{\zz_j}^{\xbeta}\omega_{ij}=
\begin{cases}
(-1)^{\alpha_1+\cdots +\alpha_n} \prod_{i=1}^n \frac{\beta_i!}{(\beta_i-\alpha_i)!} \, \partial_{\zz_j}^{\xbeta-\xalpha}\omega_{ij} & \alpha_i\le \beta_i \text{ for all } i \\
0 & \alpha_i>\beta_i \text{ for some } i \\
\end{cases},
$ 
\item the Arnold relation:
$$g(\partial_{\zz_j})(f(\partial_{\zz_i})\omega_{ij}h(\partial_{\zz_k})\omega_{jk})+h(\partial_{\zz_k})(g(\partial_{\zz_j})\omega_{jk}f(\partial_{\zz_i})\omega_{ki})+f(\partial_{\zz_i})(h(\partial_{\zz_k})\omega_{ki}g(\partial_{\zz_j})\omega_{ij})=0.$$
\end{itemize}
\end{thm}

\begin{cor} \label{cor:Hdef-group-structure}
$\Hdef(\maka)$ is isomorphic to the subring of $H^\bullet(\maka)$ generated by 
$$
\C[z_j^1,\cdots,z_j^n]\oplus \oplus_{i<j} {\C[\partial_{z^1_j},\cdots, \partial_{z^n_j}]\omega_{ij}},
$$
subject to the following relations
\begin{itemize}
\item $f(\partial_{\zz_j})\omega_{ij}\cdot g(\partial_{\zz_j})\omega_{ij}=0$, for $f(\partial_{\zz_j})\omega_{ij}, g(\partial_{\zz_j})\omega_{ij}\in {\C[\partial_{z^1_j},\cdots, \partial_{z^n_j}]\omega_{ij}}$,
\item $
\zz_{ij}^{\xalpha}\partial_{\zz_j}^{\xbeta}\omega_{ij}=
\begin{cases}
(-1)^{\alpha_1+\cdots +\alpha_n} \prod_{i=1}^n \frac{\beta_i!}{(\beta_i-\alpha_i)!} \, \partial_{\zz_j}^{\xbeta-\xalpha}\omega_{ij} & \alpha_i\le \beta_i \text{ for all } i \\
0 & \alpha_i>\beta_i \text{ for some } i \\
\end{cases},
$
\item the Arnold relation: 
$$g(\partial_{\zz_j})(f(\partial_{\zz_i})\omega_{ij}h(\partial_{\zz_k})\omega_{jk})+h(\partial_{\zz_k})(g(\partial_{\zz_j})\omega_{jk}f(\partial_{\zz_i})\omega_{ki})+f(\partial_{\zz_i})(h(\partial_{\zz_k})\omega_{ki}g(\partial_{\zz_j})\omega_{ij})=0.$$
\end{itemize}
\end{cor}

\begin{rem} \label{rem:local-relations}
The second relation in cohomology
$$
\zz_{ij}^{\xalpha}\partial_{\zz_j}^{\xbeta}\omega_{ij}=
\begin{cases}
(-1)^{\alpha_1+\cdots +\alpha_n} \prod_{i=1}^n \frac{\beta_i!}{(\beta_i-\alpha_i)!} \, \partial_{\zz_j}^{\xbeta-\xalpha}\omega_{ij} & \alpha_i\le \beta_i \text{ for all } i \\
0 & \alpha_i>\beta_i \text{ for some } i \\
\end{cases},
$$
can be proved directly using regularized integrals. For example, 
$$
\bd  \dashint_{\ZZ_k} d\zz_k \, z^a_k \omega_{ki}\omega_{kj} =z^a_{ij}\omega_{ij},
$$
$$
\bd \dashint_{\ZZ_k} d\zz_k \,  z^a_k \omega_{ki}\partial_{z^a_j}\omega_{kj}
=\omega_{ij} + z^a_{ij} \partial_{z^a_j}\omega_{ij}.
$$
Consequently, any $\bd$-exact term has a $\bd$-preimage represented by a regularized integral. This suggests packaging algebraic relations into the CDGA of diagrams, which leads to the proof of the formality theorem 
in Section \ref{sec:formality-conf}.
\end{rem}

\section{Regularized Integrals and Residues on Configuration Spaces}\label{sec:Regularized-Integrals}
In this section, we define regularized integrals and residues of differential forms on $\maka$ whose singularities arise from Bochner kernels, which are used to construct the configuration space integral $\GIK(\Gamma)$ in Section \ref{sec:Conf-Space-Integrals}.

The plan of this section is as follows.
\begin{itemize}
\item \ref{sec:cmaka}: We define the compactification $\cmaka$
of $\maka$.
\item \ref{sec:Asymptotic-Bochner-kernel}: 
We study the  singularities of $\omega_{BM}$  when lifted to $\cmaka$.
\item \ref{sec:regints}: 
We define regularized integrals and residues.
\end{itemize}

\subsection{Compactification of Configuration Spaces}\label{sec:cmaka} 
\subsubsection{Mixed Projective Spaces}\label{sec:projective-space}

Let the coordinates on $\R^{n'}\times \C^{n+1}$ be denoted by $(\xx,\zz)=(x^1,\cdots,x^{n'},z^0,\cdots,z^{n})$. 
The group $\C^* \cong \R_{+} \times S^1$ acts on $\R^{n'}\times \C^{n+1}\punctured$ by 
$$
(r, e^{i\theta}) \cdot (\xx,\zz) = (r\xx,re^{i\theta}\zz).
$$ 
Denote an equivalence class by $[\xx:\zz]$ and the set of equivalence classes by $P^{n',n}$. $P^{n',n}$ is an orbifold with an open cover $\lr{U_i,V_j^{\pm}}$ where
$$
U_i=\set{[\xx:\zz]}{z^i\ne 0}
=\lr{[x^1:\cdots:x^{n'}:z^0:\cdots :1:\cdots:z^{n}]}
\cong \aka,
$$
$$
V_j^{\pm}=\set{[\xx:\zz]}{x^j/|x^j| =\pm 1} =\lr{[x^1:\cdots:\pm 1:\cdots:x^{n'}:z^0:\cdots:z^{n}]}
\cong \R^{n'-1} \times (\C^{n+1}/S^1).
$$
Notice that $\C^{n+1}/S^1$ is singular at $0$. Indeed, it is a cone with regular points $(\C^{n+1}\punctured)/S^1 \cong \R_+ \times \C\P^{n}$. 
To obtain a smooth version, we blow up the singular locus $\lr{[\xx:0]}$. After blowup, $V_j^\pm$ is replaced by  
$$
\widetilde V_j^\pm = \set{([\xx:\zz],[\ww])\in P^{n',n} \times \C\P^{n}}{x^j=\pm 1, \,\zz\in [\ww]}.
$$
Now $\widetilde V_j \cong \R^{n'-1} \times \R_{\ge 0} \times \C\P^{n}$ is smooth. Moreover, $\widetilde V_j^\pm$ has an open cover $\lr{\widetilde V_{jk}^\pm}$, where 
\ali{
\widetilde V_{jk}^\pm
=& \set{([\xx:\zz],[\ww]) \in \widetilde V_j^\pm }{w^k\ne 0} \\ 
=& \set{([x^1:\cdots:\pm 1:\cdots:x^{n'}:z^0:\cdots:z^{n}],[w^0:\cdots:1:\cdots:w^{n}])}{\zz\in[\ww]} \\
\cong& \lr{([x^1:\cdots:\pm 1:\cdots:x^{n'}],z^k,[w^0:\cdots:1:\cdots:w^{n}])}  \\
\cong& \R^{n'-1} \times \C/S^1 \times \C^{n} \\
\cong& \R^{n'-1} \times \R_{\ge 0} \times \C^{n}.
} 
Consequently, we define the mixed projective space
$$
\P^{n',n} = \set{([\xx:\zz],[\ww])\in P^{n',n} \times \C\P^{n}}{\zz\in [\ww]}
$$
which is a smooth manifold with boundary,  with an open cover $\lr{U_i, \widetilde V_{jk}^\pm}$. Moreover, $\P^{n',n}$ is orientable.

\subsubsection{Blowup}\label{sec:blow-up}
To study the singularity of the Bochner kernel, we need to blow up the origin $0\in \aka $. 
If we simply replace $\lr{0}$ by a projective space $P^{n',n-1}$, the result
$$
\BL_0(\aka) = \set{((\xx,\zz), [\yy,\ww]) \in (\aka ) \times P^{n',n-1} }{(\xx,\zz)\in [\yy,\ww]}
$$
is only an orbifold. 
To get a smooth one, we replace $\R^{n'}\times \lr{0} \subset \aka $ by $\R^{n'}\times \C\P^{n-1}$ and get
$$
\BL_{\R^{n'}\times \lr{0}} (\aka) 
=\set{((\xx,\zz),[\vv])\in (\aka)\times \C\P^{n-1}}{\zz\in [\vv]}.
$$
Take a further blowup at 
$$
\C\P^{n-1} \cong 
\set{((0,\zz),[\vv])\in (\aka)\times \C\P^{n-1}}{\zz\in[\vv]} \subset \BL_{\R^{n'}\times \lr{0}} (\aka),
$$
we get the smooth blowup
\ali{
\widetilde{\BL}_0(\aka)
&=\BL_{\C\P^{n-1}} \BL_{\R^{n'}\times \lr{0}} (\aka) \\
&=\set{((\xx,\zz),[\yy,\ww],[\vv]) \in (\aka ) \times P^{n',n-1}\times \C\P^{n-1}}{(\xx,\zz)\in [\yy,\ww],\ww\in[\vv]} \\
&=\set{((\xx,\zz),([\yy,\ww],[\vv])) \in (\aka )\times \P^{n',n-1}}{(\xx,\zz)\in [\yy,\ww],\ww\in[\vv]}
}
with a natural projection 
$$
p: \widetilde{\BL}_0(\aka) \lra \aka, \qquad p^{-1}(0) \cong \P^{n',n-1}.
$$

\subsubsection{The Fulton-MacPherson Compactification}
The Fulton-MacPherson compactification \cite{compactification} extends to compactification spaces of a manifold with boundary, or more generally, with corners \cite{blowup-corner}. For a manifold $X$, there is a natural embedding
$$
\Conf_m(X) \subset X^m \times \prod_{|S|\ge 2} \BL_{\Delta} X^S
$$
where $S\subset \lr{1,\cdots,m}$ and $\BL_{\Delta} X^S$ is the blowup of $X^S$ at the small diagonal. The compactification $F_m(X)$ is defined to be the closure of $\Conf_m(X)$ in the product. Denote by $F_M(X)$ the space constructed from any finite set $M$, so $F_M(X)=F_m(X)$ when $M$ is identified with  $\lr{1,\cdots,m}$. There is a canonical map from $F_M(X)$ to $F_S(X)$ for each $S\subset M$.

For each subset $S\subset \lr{1,\cdots,m}$ with $|S|\ge 2$, there is a divisor $D_S$ that describes the locus where the points in $S$ move close together and collide. $D_S \cap D_{S'}$ is nonempty iff $S$ and $S'$ are either disjoint or one is contained in the other.

\subsubsection{Compactification of $\maka$}
In this section we define the compactification $\cmaka$ of $\maka$.

Heuristically, the embedding $\aka\cong U_0=\set{[\xx:\zz]}{z^0\ne 0} \subset \P^{n',n}$ induces a dense embedding 
\ali{
\Conf_m(\aka) \hlra  \Conf_m(\P^{n',n}) \hlra F_m(\P^{n',n}).   
}  
However, $F_m(\P^{n',n})$ is only an orbifold, whose singular loci come from degenerate $S^1$ actions (see Section \ref{sec:projective-space}). To get a smooth space, we need to take a further blowup on the locus that complex coordinates of some points coincide.  

Recall
$$
\P^{n',n} = \set{([\xx:\zz],[\ww])\in P^{n',n} \times \C\P^{n}}{\zz\in [\ww]}.
$$
Let $(([\xx_1:\zz_1],[\ww_1]),\cdots,([\xx_m:\zz_m],[\ww_m]))$ be coordinates on
$(\P^{n',n})^m$, and 
$$
\Delta_{ij}=\lr{([\xx_i:\zz_i],[\ww_i])=([\xx_j:\zz_j],[\ww_j])}, \quad
\Delta'_{ij}=\lr{[\ww_i]=[\ww_j]} \subset (\P^{n',n})^m. 
$$
In the following, the blowup of an union of subspaces are all carried out in a similar manner as the Fulton-MacPherson compactification, and we omit the restrictions of the subspaces to blowup, for example, $\BL_{\bigcup_{i<j} \Delta'_{ij}} (\aka)^m$ means $\BL_{\bigcup_{i<j} \Delta'_{ij}|_{(\aka)^m}} (\aka)^m$. 

Successive blowups along $\bigcup_{i<j}\Delta'_{ij}$ give a space 
$$
\BL_{\bigcup_{i<j} \Delta'_{ij}} (\P^{n',n})^m
$$
with a natural projection 
$$
p': \BL_{\bigcup_{i<j} \Delta'_{ij}} (\P^{n',n})^m \lra (\P^{n',n})^m
$$
Successive blowups along $p'^{-1}\LR{\bigcup_{i<j}\Delta'_{ij}}$ gives a space  
$$
\cmaka := \BL_{p'^{-1}\LR{\bigcup_{i<j}\Delta'_{ij}}} \BL_{\bigcup_{i<j} \Delta'_{ij}} (\P^{n',n})^m.
$$
There are natural projections $\cConf_{m}(\aka) \to \cConf_{m-p}(\aka)$.
Moreover, the following diagram commutes:
\begin{center}
\begin{tikzcd}
\BL_{\bigcup_{i<j} \Delta'_{ij}} \maka \arrow[r, hook] & \BL_{p'^{-1}\LR{\bigcup_{i<j}\Delta'_{ij}}} \BL_{\bigcup_{i<j} \Delta'_{ij}} (\aka)^m \arrow[d] \arrow[r, hook] & \cmaka \arrow[d]  \\
& \BL_{\bigcup_{i<j} \Delta'_{ij}} (\aka)^m \arrow[d] \arrow[r, hook] & \BL_{\bigcup_{i<j} \Delta'_{ij}} (\P^{n',n})^m \arrow[d]   \\
& (\aka)^m \arrow[r, hook] & {(\P^{n',n})^m}
\end{tikzcd}
\end{center}


\subsubsection{Compactification of Fibers}\label{sec:cfiber}
In this section, we give a substitute for the projection 
\ali{
\pr: \Conf_{m}(\aka) &\lra \Conf_{m-p}(\aka) \\
(\ZZ_1,\cdots,\ZZ_m) &\lmt (\ZZ_1,\cdots,\ZZ_{m-p})
}
which is a smooth fiber bundle with compact fibers. 

The compactification gives the following commutative diagram
\begin{center}
\begin{tikzcd}
\cmaka \arrow[r, "\overline{\pr}"]   & \cConf_{m-p}(\aka)     \\
\BL_{\bigcup_{i<j} \Delta'_{ij}} \maka \arrow[d, "q"'] \arrow[r] \arrow[u, hook, "j"] & \BL_{\bigcup_{i<j} \Delta'_{ij}} \Conf_{m-p}(\aka) \arrow[d, "p"] \arrow[u, hook, "i"'] \\
\maka \arrow[r, "\pr"]     & \Conf_{m-p}(\aka)   
\end{tikzcd}
\end{center} 
where $i,j$ are dense embeddings. 
Note that the projection between compactified configuration spaces
$$
\overline{\pr}: \cmaka \lra \cConf_{m-p}(\aka)
$$
is not locally trivial in general; the fibers over the exceptional divisors are not homeomorphic to those on the dense open subset $\Im i$ of the image of $i$.  

Define $C^{n',n}_{m,p}=\overline{\pr}^{-1}(\Im i)$. 
Restrict $\overline{\pr}$ to  $\Im i$, we get a pullback diagram
\begin{center}
\begin{tikzcd}
\cmaka \arrow[r, "\overline{\pr}"]    & \cConf_{m-p}(\aka)   \\
C^{n',n}_{m,p} \arrow[u, hook] \arrow[r, "\widetilde\pr"] & \BL_{\bigcup_{i<j} \Delta'_{ij}} \Conf_{m-p}(\aka) \arrow[u, hook, "i"']
\end{tikzcd}
\end{center}
where 
$$
\widetilde\pr: C^{n',n}_{m,p} \lra \BL_{\bigcup_{i<j} \Delta'_{ij}} \Conf_{m-p}(\aka)
$$
is a smooth fiber bundle with compact fibers.

\subsection{Asymptotic Behavior of the Generalized Bochner-Martinelli Kernel}\label{sec:Asymptotic-Bochner-kernel}
 
The generalized Bochner-Martinelli kernel $\omega_{BM}$ has a singularity at the origin $0\in\aka$, and on a sphere 
$$
S^{n'+2n-1}_{0,r} = \set{(\xx,\zz)\in \aka }{|\xx|^2+|\zz|^2=r^2}
$$
centered at $0$, we have 
$$
\int_{S^{n'+2n-1}_{0,r}} \omega_{BM} = 1.
$$
Indeed, the same statement holds for the hyperplane at infinity. Let  $H=\set{[\xx:\zz]\in P^{n',n}}{z^0=0}$ which is the complement of the inclusion 
\ali{
\aka  &\hlra P^{n',n} \\
(x^1,\cdots,x^{n'},\ZZi ) &\lmt [x^1:\cdots:x^{n'}:1:z^1:\cdots:z^n]
} 
Consider an open neighborhood of $H$ that is homeomorphic to the unit disk bundle of $H$, whose boundary is a sphere $S^{n'+2n-1}_{H,r}=\set{[x^1:\cdots:x^{n'}:r:z^1:\cdots:z^n]\in P^{n',n}}{\sum_{i=1}^{n'}(x^i)^2+\sum_{i=1}^{n'}|z^i|^2=1}$ around $H$, and we have 
$$
\int_{S^{n'+2n-1}_{H,r}} \omega_{BM} = -1
$$
because  
$$
\LR{\int_{S^{n'+2n-1}_{0,r}} + \int_{S^{n'+2n-1}_{H,r}}} \omega_{BM}
=\int_{\partial U} \omega_{BM}
=\int_U d\omega_{BM}
=0
$$
where $U$ is the open region between two spheres. 

Indeed, in both cases we have line bundles $\mathcal O(-1)$ and $\mathcal O(1)$ and  Hopf fibrations
\begin{center}
\begin{tikzcd}
S^1 \arrow[r, hook] & \BL_{\xx=0}S^{n'+2n-1} \arrow[d] \\
              & \P^{n',n-1}             
\end{tikzcd}
\end{center}
Below we will show $\omega_{BM}|_{S^{n'+2n-1}}= d\theta \wedge \dVol_{P^{n',n-1}}$, where $\theta\in S^1$ is the fiber coordinate, and $\dVol_{P^{n',n-1}}$ is a volume form on $P^{n',n-1}$. 

\subsubsection{Asymptotic Behavior at \texorpdfstring{$0$}{0}}\label{sec:Asymptotic-zero}
Define a pair of conjugate vector fields
$$
V=\sum_{i=1}^n 2z^i\partial_{z^i}+\sum_{i=1}^{n'} x^i\partial_{x^i}, \quad 
\bar V=\sum_{i=1}^n 2\bar z^i\partial_{\bar z^i}+\sum_{i=1}^{n'} x^i\partial_{x^i}.
$$
Since $[V,\bar V]=0$, we have the commutation relations
$$
\Lr{\mathcal L_{V},\iota_{V}}=
\Lr{\mathcal L_{V},\iota_{\bar V}}=
\Lr{\mathcal L_{\bar V},\iota_{V}}=
\Lr{\mathcal L_{\bar V},\iota_{\bar V}}=0.
$$ 
Let $\dVol_{\aka}$ be the standard volume form on $\aka$. A direct verification shows that
$$
\mathcal L_{V} \LR{\frac{\dVol_{\aka}}{(\rSquare )^{n + \frac{n'}{2}}} }
=\mathcal L_{\bar V} \LR{\frac{\dVol_{\aka}}{(\rSquare )^{n + \frac{n'}{2}}}}
=0
$$
and the generalized Bochner-Martinelli kernel $\omega_{BM}$ is given by
$$
\omega_{BM}=\iota_{\bar V} \LR{\frac{\dVol_{\aka}}{(\rSquare )^{n + \frac{n'}{2}}}}.
$$

Let 
$$
\alpha=\frac{\sum_{i=1}^n \bar z^i dz^i +\sum_{i=1}^{n'}x^idx^i }{\rSquare }, \quad 
\bar \alpha=\frac{\sum_{i=1}^n  z^i d\bar z^i +\sum_{i=1}^{n'}x^idx^i }{\rSquare }.
$$
Then we have
$$
\iota_{V}\alpha=\iota_{\bar V}\bar \alpha=1, \quad
\iota_{V}\bar \alpha=\iota_{\bar V}\alpha=\frac{|\xx|^2}{\rSquare },
$$
\ali{
\alpha\wedge \iota_{V}\iota_{\bar V} \LR{\frac{\dVol_{\aka}}{(\rSquare )^{n + \frac{n'}{2}}}}
=&  (\iota_{V}\alpha) \iota_{\bar V} \LR{\frac{\dVol_{\aka}}{(\rSquare )^{n + \frac{n'}{2}}}} 
- (\iota_{\bar V}\alpha) \iota_{V} \LR{\frac{\dVol_{\aka}}{(\rSquare )^{n + \frac{n'}{2}}}} \\
=& \omega_{BM} 
- \frac{|\xx|^2}{\rSquare }  \iota_{V} \LR{\frac{\dVol_{\aka}}{(\rSquare )^{n + \frac{n'}{2}}}},  
}
\ali{
\bar \alpha\wedge \iota_{V}\iota_{\bar V} \LR{\frac{\dVol_{\aka}}{(\rSquare )^{n + \frac{n'}{2}}}}
=&   (\iota_{V}\bar \alpha) \iota_{\bar V} \LR{\frac{\dVol_{\aka}}{(\rSquare )^{n + \frac{n'}{2}}}} 
- (\iota_{\bar V}\bar \alpha) \iota_{V} \LR{\frac{\dVol_{\aka}}{(\rSquare )^{n + \frac{n'}{2}}}} \\
=&  \frac{|\xx|^2}{\rSquare } \omega_{BM} 
- \iota_{V}\LR{\frac{\dVol_{\aka}}{(\rSquare )^{n + \frac{n'}{2}}}}. 
} 
Combining the two preceding equations yields
\ali{
\LR{\frac{\rSquare }{|\xx |^2} \alpha - \bar \alpha}   \wedge \iota_{V}\iota_{\bar V} \LR{\frac{\dVol_{\aka}}{(\rSquare )^{n + \frac{n'}{2}}}}
= \LR{\frac{\rSquare }{|\xx |^2} -\frac{|\xx|^2}{\rSquare }}   \omega_{BM}.
}
So
\ali{
\omega_{BM}
=&\frac{(\rSquare ) ((\rSquare ) \alpha - |\xx |^2 \bar \alpha)}{4|\zz|^2(|\zz|^2 + |\xx |^2)} 
 \wedge \iota_{V}\iota_{\bar V} \LR{\frac{\dVol_{\aka}}{(\rSquare )^{n + \frac{n'}{2}}}} \\ 
}
The first factor simplifies to
\ali{
&\frac{(\rSquare ) ((\rSquare ) \alpha - |\xx |^2 \bar \alpha)}{4|\zz|^2(|\zz|^2 + |\xx |^2)} \\
=&\frac{(\rSquare )\sum_{i}\bar z^idz^i - |\xx |^2 \sum_i z^id\bar z^i + 2 |\zz|^2 \sum_i x^idx^i }{4|\zz|^2(|\zz|^2 + |\xx |^2)} \\
=&\frac{(|\zz|^2 + |\xx |^2)(\sum_{i}\bar z^i dz^i -\sum_i z^i d\bar z^i)  + |\zz|^2 (\sum_{i}\bar z^idz^i +\sum_i z^id\bar z^i+ 2 \sum_i x^idx^i)}{4|\zz|^2(|\zz|^2 + |\xx |^2)} \\
=&\frac{\sum_{i}\bar z^idz^i -\sum_i z^id\bar z^i}{4|\zz|^2} + \frac{\sum_{i}\bar z^idz^i +\sum_i z^id\bar z^i+2\sum_i x^idx^i}{4(|\zz|^2 + |\xx |^2)} \\
=& \frac12 (id\theta +d\log r) \\
=& \frac12 d\log z
} 
Here $z\in \C^*$ is the fiber coordinate of the $\C^*$-action on the blowup of $\aka$ at $0$ (see \cite{Griffiths-Harris}, page 373 for the complex case). If we take a further blowup of $\R^{n'}\times \lr{0}$, the $S^1$ action will be nondegenerate, and $d\theta$ lifts to a smooth form on the blowup of the sphere. 

Let 
$$
\tau= \frac12\iota_{V} \omega_{BM} 
= \frac12\iota_{V}\iota_{\bar V} \LR{\frac{\dVol_{\aka}}{(\rSquare )^{n + \frac{n'}{2}}}}.
$$
Then we have
$$
\iota_{V}\tau= \iota_{\bar V}\tau 
=\mathcal L_{V} \tau =\mathcal L_{\bar V} \tau
=d\tau=0.
$$ 
As a result, $\tau$ descends to a top form on the quotient space $P^{n',n-1}$, which is a volume form. It lifts to a smooth form on $\P^{n',n-1}$ after a further blowup. 

Consequently, we see $\omega_{BM}$ takes the form
$$
\omega_{BM} = d\log z \wedge \dVol_{P^{n',n-1}}.
$$
When restricted to a sphere $S^{n'+2n-1}_0$, it becomes 
$$
id\theta \wedge \dVol_{P^{n',n-1}}.
$$
As a result, the integral
$$
\int_{S^{n'+2n-1}}\omega_{BM}
$$
is interpreted as 
$$
\int_{P^{n',n-1}} \int_{S^1}id\theta \wedge \dVol_{P^{n',n-1}}
$$



\subsubsection{Asymptotic Behavior at Infinity}\label{sec:bochner-infty}
In this section, we examine the behavior of $\omega$ near the hyperplane $H$ at infinity in 
$$
\paka \hlra P^{n',n} = \aka \cup H, \qquad H=\set{[\xx:\zz]\in P^{n',n}}{z^0=0}.
$$
The description is more complicated than that at $0$, because the coordinate changes on mixed projective spaces are not biholomorphic. 
 

Since the chart $\aka=U_0=\lr{z^0\ne 0}$ does not intersect with $H$, we need to use another chart, for example, $U_p=\set{[\xx:\zz]}{z^p\ne 0}$. On $U_0\cap U_p$, the coordinate transformation is given by
\ali{
\lr{[x^1:\cdots:x^{n'}:z^0:\cdots :z^{p-1}:1:z^{p+1}:\cdots:z^{n}]} 
=\lr{\Lr{\frac{x^1}{|z^0|}:\cdots:\frac{x^{n'}}{|z^0|}:1:\frac{z^1}{z^0}\cdots :\frac{1}{z^0}:\cdots:\frac{z^n}{z^0}}}.
}
Let $(\yy_k,\ww_k)$ denote points in $\aka$. For $(\xx,\zz)\in U_0\cap U_p$, we have 
$$
(\xx-\yy_k,\zz-\ww_k)=\LR{\frac{x^1}{|z^0|}-y_k^1,\cdots,\frac{x^{n'}}{|z^0|}-y_k^{n'},\frac{z^1}{z^0}-w_k^1,\cdots,\frac{1}{z^0}-w_k^p,\cdots,\frac{z^n}{z^0}-w_k^n}.
$$ 
Recall $\omega$ is of the form 
$$
\omega=  \frac{\frac{2^n \Gamma \left( n + \frac{n'}{2} \right)}{\pi^{n + \frac{n'}{2}}} }{(\rSquare )^{n + \frac{n'}{2}}} \cdot 
\left( \sum_{i = 1}^n (- 1)^{i - 1} 2\bar z^i \left(
    \prod_{j \neq i} d \bar z^j \right) d^{n'} x +
 \sum_{i = 1}^{n'} (- 1)^{n + i - 1} x^i d^n
    \bar z \left( \prod_{j \neq i} d x^j \right) \right).
$$
Then 
\ali{
&\omega(\xx-\yy_k,\zz-\ww_k) \\
=&  \frac{2^n \Gamma \left( n + \frac{n'}{2} \right)}{\pi^{n + \frac{n'}{2}}}     \frac{1}{\LR{2\sum_{i\ne 0,p} \abs{\frac{z^i}{z^0}-w_k^i}^2 + \abs{\frac{1}{z^0}-w_k^p}^2+ \sum_{i=1}^{n'} \abs{\frac{x^i}{|z^0|}-y_k^i}^2 }^{n + \frac{n'}{2}}} \cdot \\
&  \left\{\left[  \sum_{i\ne 0,p} (- 1)^{i - 1} 2\LR{\frac{\bar z^i}{\bar z^0}-\bar w_k^i} \LR{\prod_{\substack{j\ne 0,i\\j<p}} d\LR{\frac{\bar z^j}{\bar z^0}-\bar w_k^j}} d\LR{\frac{1}{\bar z^0}-\bar w_k^p} \LR{\prod_{\substack{j\ne 0,i\\j>p}} d\LR{\frac{\bar z^j}{\bar z^0}-\bar w_k^j}}    \right.\right. \\
&\qquad + \left. (- 1)^{p - 1} \LR{\frac{1}{\bar z^0}-\bar w_k^p} \prod_{j\ne 0,p } d\LR{\frac{\bar z^j}{\bar z^0}-\bar w_k^j} \right] \prod_{i=1}^{n'} d\LR{\frac{x^i}{|z^0|}-y_k^i} \\
& + \left. \sum_{i = 1}^{n'} (- 1)^{n + i - 1} \LR{\frac{x^i}{|z^0|}-y_k^i} \prod_{j\ne 0,i}  d\LR{\frac{x^j}{|z^0|}-y_k^j} \LR{\prod_{\substack{j\ne 0,i\\j<p}} d\LR{\frac{\bar z^j}{\bar z^0}-\bar w_k^j}} d\LR{\frac{1}{\bar z^0}-\bar w_k^p} \LR{\prod_{\substack{j\ne 0,i\\j>p}} d\LR{\frac{\bar z^j}{\bar z^0}-\bar w_k^j}} \right\} \\ 
\overset{(1)}{=}&   \frac{2^n \Gamma \left( n + \frac{n'}{2} \right)}{\pi^{n + \frac{n'}{2}}}     \frac{(z^0)^n}{\LR{2\sum_{i\ne 0,p} \abs{z^i-w_k^iz^0}^2 + \abs{1 - w_k^p z^0}^2+ \sum_{i=1}^{n'} \abs{x^i - y_k^i |z^0|}^2 }^{n + \frac{n'}{2}}} \cdot \\
&  \left\{\left[  \sum_{i\ne 0,p} (- 1)^{i - 1} 2\LR{\bar z^i - \bar w_k^i \bar z^0} \LR{\prod_{\substack{j\ne 0,i\\j<p}} d\LR{\bar z^j - \bar w_k^j \bar z^0}} d\LR{1 - \bar w_k^p \bar z^0} \LR{\prod_{\substack{j\ne 0,i\\j>p}} d\LR{\bar z^j - \bar w_k^j \bar z^0}}    \right.\right. \\
&\qquad + \left. (- 1)^{p - 1} \LR{1 - \bar w_k^p \bar z^0} \prod_{j\ne 0,p } d\LR{\bar z^j - \bar w_k^j \bar z^0} \right] \prod_{i=1}^{n'} d\LR{x^i - y_k^i |z^0|} \\
& + \left. \sum_{i = 1}^{n'} (- 1)^{n + i - 1} \LR{x^i - y_k^i |z^0|} \prod_{j\ne 0,i}  d\LR{x^j - y_k^j |z^0|} \LR{\prod_{\substack{j\ne 0,i\\j<p}} d\LR{\bar z^j - \bar w_k^j \bar z^0}} d\LR{1 - \bar w_k^p \bar z^0} \LR{\prod_{\substack{j\ne 0,i\\j>p}} d\LR{\bar z^j - \bar w_k^j \bar z^0}} \right\} \\
& + O\LR{\frac{d\bar z^0}{\bar z^0},\frac{d|z^0|}{|z^0|}} \\
\overset{(2)}{=}&  \frac{2^n \Gamma \left( n + \frac{n'}{2} \right)}{\pi^{n + \frac{n'}{2}}}     \frac{(z^0)^n}{\LR{2\sum_{i\ne 0,p} \abs{z^i-w_k^iz^0}^2 + \abs{1 - w_k^p z^0}^2+ \sum_{i=1}^{n'} \abs{x^i - y_k^i |z^0|}^2 }^{n + \frac{n'}{2}}}      \cdot (- 1)^{p - 1}  \prod_{j\ne 0,p } d \bar z^j     \prod_{i=1}^{n'} d x^i  \\ 
& + O\LR{\frac{d\bar z^0}{\bar z^0},\frac{d|z^0|}{|z^0|}, \bar z^0, |z^0|}  
}
Here in step (1) we extract all terms involving $\bar z^0 d\frac1{\bar z^0}$ or $|z^0| d\frac1{|z^0|}$, and in step (2) we extract all terms involving $d(\bar w_k^i\bar z^0)$ or  $d(y_k^i|z^0|)$. 
 
The behavior of $\omega$ at charts $V^{\pm}_q$ is similar, which we omit  to avoid repetition.

\subsubsection{The Kernel $\omega$ lifted to $\cmaka$}
When lifted to $\cConf_2(\aka)$, the space of compactified configurations of $\ZZ_i$ and $\ZZ_j$, the singular loci of $\omega_{ij}$ are the union of the diagonal $\Delta$ and infinite planes. After a pull-back along the canonical projection
\ali{
\cmaka \lra \cConf_2(\aka),
}
we see $\omega_{ij}$ is singular on the divisor $D_{S}$ iff $\lr{i,j}\subset S$.

\subsection{Regularized Integrals and Residues}\label{sec:regints}

\subsubsection{Principal Values and Residues}\label{sec:PV-Res}
In this section, we review the theory of principal values and residues from \cite{higher-residue}.

Let $W\subset \C^q$ be an open subset. Let $\xi$ and $\theta$ be compactly supported smooth differential forms on $W$, with degree $2q$ and $2q-1$, and $\varphi$ be a holomorphic function on $W$. It is shown in \cite{higher-residue} that the limits 
$$
\lim_{r\to 0} \int_{|\varphi|>r} \frac{\xi}{\varphi}
\quad \text{and} \quad
\lim_{r\to 0} \int_{|\varphi|=r} \frac{\theta}{\varphi}
$$
exist. Using a partition of unity, these local currents  patch together to globally defined principal values $\PV$ and residues $\Res$ on a complex space $X$. Moreover, the Stokes theorem implies that 
$$
\PV \circ  d = \Res 
$$
on $2q-1$ forms.

\begin{eg}[Local description of residue]\label{eg:local-residue}
Let $W\subset\C^q$ be an open subset, and let $Y_1,\cdots,Y_k$ be normal crossing divisors defined by local equations $z^1,\cdots,z^k$. Let $\theta$ be a compactly supported smooth $2q-1$ form on $W$. 

We compute the residue $\Res \frac{\theta}{(z^1)^{a_1}\cdots (z^k)^{a_k}}$. By \cite{higher-residue}, Proposition 6.5, $\Res$ vanishes on $(q-1,q)$-forms, and on $(q,q-1)$-forms, $\Res=\sum_{i=1}^k \Res_i$,  where  
$$
\Res_i \LR{\frac{f}{(z^1)^{a_1}\cdots (z^k)^{a_k}} dz^l \wedge \prod_{j\ne l} dz^j d\bar z^j }
=\begin{cases}
\frac{2\pi i}{(a_i-1)!} \PV_{Y_i \cap W} \LR{\frac{\LR{\partial_{z^i}^{a_i-1}f}|_{Y_i \cap W}}{(z^1)^{a_1}\cdots\widehat{(z^i)^{a_i}}\cdots (z^k)^{a_k}} \prod_{j\ne i} dz^j d\bar z^j} & l=i\\
0 & l\ne i 
\end{cases}
$$
is the residue along $Y_i$. In other words, $\Res$ is characterized by annihilating d and extracting the coefficient of $\frac{dz^i}{z^i}$  along $Y_i$. 
\end{eg}

In this paper, we will use a slight generalization of $\PV$ and $\Res$ to forms with polynomial dependence on $\bar z^i,|z^i|, \frac{d\bar z^i}{\bar z^i},\frac{d|z^i|}{|z^i|}$. 
The following lemma shows $\PV$ and $\Res$ still works. 

\begin{lem}
For integers $b,c\ge 0$, we have 
\ali{
\lim_{\epsilon\to 0}\int_{|z|=\epsilon} z^a \bar z^b |z|^c \frac{dz}{z} 
&=\lim_{\epsilon\to 0} r^{a+b+c} \int_0^{2\pi} e^{(a-b)i\theta} id\theta
=2\pi i \delta_{a,0}\delta_{b,0}\delta_{c,0}, \\
\lim_{\epsilon\to 0}\int_{|z|=\epsilon} z^a \bar z^b |z|^c \frac{d\bar z}{\bar z} 
&=\lim_{\epsilon\to 0} r^{a+b+c} \int_0^{2\pi} e^{(a-b)i\theta} (-id\theta)
=-2\pi i \delta_{a,0}\delta_{b,0}\delta_{c,0}, \\
\lim_{\epsilon\to 0}\int_{|z|=\epsilon} z^a \bar z^b |z|^c \frac{d|z|}{|z|} 
&=0.
} 
\end{lem}

\begin{prop}\label{prop:constructibility-reg-int}
Parametric regularized integrals preserve constructibility.
\end{prop}

\begin{proof}
Using the Fubini Theorem for regularized integrals \cite{Regularized-integral}, the proof reduces to the one-dimensional case. Let $f(z,\xx)$ be a constructible $C^k$ function on $\C\times \R^p$ with compact support on $\C\times \lr{\xx_0}$ for every $\xx_0\in\R^p$, and $l<k$. Using integration by parts, we have
$$
\dashint_{\C} \dVol_z \, \frac{f(z,\xx)}{z^l} 
=\frac{(-1)^{l-1}}{(l-1)!} \dashint_{\C} \dVol_z \, \frac{\partial_z^{l-1}f(z,\xx)}{z}.
$$
The right hand side converges absolutely, so the result is a constructible function on $\R^p$ \cite{int-closed}.
\end{proof}

\begin{rem}
Parametric residues also preserve constructibility. To show $\Res_i$ in Example \ref{eg:local-residue} preserves constructibility, we first take regularized integrals with respect to $z^1,\cdots,\widehat z^i,\cdots, z^k$, then the  constructibility comes from the one-dimensional Stokes theorem \cite{Regularized-integral}
$$
\dashint_{\Sigma} d(-) = -2\pi i\Res_{\Sigma} (-) + \int_{\partial\Sigma} (-).
$$
\end{rem}

\subsubsection{Regularized Integrals and Residues on Compactified Configuration Spaces}\label{sec:reg-conf}

In this section we define the regularized integrals and residues on compactified configuration spaces which will be used  in Section \ref{sec:Conf-Space-Integrals}. 

The projection
\ali{
\pr: \Conf_{m}(\aka) &\lra \Conf_{m-p}(\aka) \\
(\ZZ_1,\cdots,\ZZ_m) &\lmt (\ZZ_1,\cdots,\ZZ_{m-p})
}
has non-compact fibers, so in Section \ref{sec:cfiber} we constructed a replacement
$$
\widetilde\pr: C^{n',n}_{m,p} \lra \BL_{\bigcup_{i<j} \Delta'_{ij}} \Conf_{m-p}(\aka)
$$
which is a smooth fiber bundle with compact fibers, 
together with a commutative diagram
\begin{center}
\begin{tikzcd}
\cmaka \arrow[rr, "\overline{\pr}"]  & & \cConf_{m-p}(\aka)     \\
 & C^{n',n}_{m,p} \arrow[lu, hook] \arrow[rd, "\widetilde\pr"] & \\
\BL_{\bigcup_{i<j} \Delta'_{ij}} \maka \arrow[d, "q"'] \arrow[rr] \arrow[ru, hook, "j'"] \arrow[uu, hook, "j"] & & \BL_{\bigcup_{i<j} \Delta'_{ij}} \Conf_{m-p}(\aka) \arrow[d, "p"] \arrow[uu, hook, "i"'] \\
\maka \arrow[rr, "\pr"]    & & \Conf_{m-p}(\aka)   
\end{tikzcd}
\end{center}  

Next, we define regularized integrals and residues associated with the projection
\ali{
\pr: \Conf_{m}(\aka) &\lra \Conf_{m-p}(\aka) \\
(\ZZ_1,\cdots,\ZZ_m) &\lmt (\ZZ_1,\cdots,\ZZ_{m-p})
}
For residues, we furthermore need a divisor $D$. 
They map a singular form in
$$
\lR{\Hol(\C^{mn}),\,\, \overline{\C[ \partial_{z_j^1},\cdots, \partial_{z_j^n}]\omega_{ij}}} \subset A^\bullet(\maka)
$$
(see Section \ref{sec:ring-structure-conf} for definition) to 
a differential form in 
$$A^\bullet(\Conf_{m-p}(\aka)).$$

Let $\alpha\in \lR{\Hol(\C^{mn}),\,\, \overline{\C[ \partial_{z_j^1},\cdots, \partial_{z_j^n}]\omega_{ij}}}$. The differential form $q^*\alpha$ on $\BL_{\bigcup_{i<j} \Delta'_{ij}} \maka$ can be viewed as a densely defined form on $C^{n',n}_{m,p}$. 

After compactification, the singular locus of each Bochner kernel becomes a one-dimensional complex direction, i.e., the fiber coordinate $\C^*$ acting on the origin or a coordinate chart of the infinite plane of $\aka$. Moreover, on $\cmaka$, all these divisors on which some $\omega_{ij}$ are singular 
are normal crossing. Consequently, the singularity of $q^*\alpha$ on $C^{n',n}_{m,p}$ along the projection $\widetilde\pr$ is of the form in Section \ref{sec:PV-Res}, and we may use a partition of unity (smooth or constructible $C^k$) and then apply the principal values and residues in Section \ref{sec:PV-Res}, followed by an integral of compactly supported differential forms. 
The choice of partition of unity will not affect the result. 

The local principal values or residues along fibers of
$$
\widetilde\pr: C^{n',n}_{m,p} \lra \BL_{\bigcup_{i<j} \Delta'_{ij}} \Conf_{m-p}(\aka)
$$
glue together to give a differential form on $\BL_{\bigcup_{i<j} \Delta'_{ij}} \Conf_{m-p}(\aka)$. By Lemma \ref{lem:form-descent} below, it descends to a differential form on $\Conf_{m-p}(\aka)$, which is denoted by $\dashint_{fibers}\alpha$ or $\GlobalRes_D \alpha$.

\begin{lem}\label{lem:form-descent}
Let $p:M\to N$ be a projection and $\alpha$ be a differential form on $M$. Then $\alpha$ is of the form $p^*\beta$ for some $\beta$ on $N$ iff $\LL_V\alpha=\iota_V\alpha=0$ for every vector field $V$ on $M$ which is killed by $p_*$.
\end{lem}

\begin{rem}


Regularized integrals and residues can be defined similarly if we replace holomorphic coefficients by smooth differential forms with compact support.

\end{rem}

\section{Formality of Configuration Spaces}\label{sec:formality-conf}
In this section $n'\ge2$, $n\ge1$ are fixed integers, unless specified otherwise. The plan of this section is as follows.
\begin{itemize}
\item \ref{sec:diagram-complex}: We define $\cD(A)$, the CDGA of admissible diagrams.
\item \ref{sec:Conf-Space-Integrals}: We construct the configuration space integral
$$
\operatorname{I}: \cD(A)  \lra \DefinableChainsA  
$$
which is a morphism of CDGAs.
\item \ref{sec:formality-thm}: Since $\cD(A)$ is quasi-isomorphic to both $\DefinableChainsA$ and its cohomology, we conclude that $\DefinableChainsA$ is formal. Similarly, $\A(\maka)$ is formal.
\end{itemize}

\subsection{The CDGA of Admissible Diagrams}\label{sec:diagram-complex}


\subsubsection{Construction of the CDGA of Admissible Diagrams}

In this section, we define the CDGA of admissible diagrams, following \cite{Formality-little-disks}. 
The main difference between our diagrams and those in \cite{Formality-little-disks} is that both vertices and edges carry weights, which encode generators of $\Hdef(\maka)$. Moreover, our admissible diagrams are allowed to have bivalent internal vertices with non-constant weights to cancel relations in Remark \ref{rem:local-relations}.

\begin{defn}\label{def:diagram}
A diagram consists of the data $\Gamma=(A_\Gamma,I_\Gamma,E_\Gamma,s_\Gamma,t_\Gamma, W^V_{\Gamma},W^E_{\Gamma})$, where
\begin{itemize}
\item $A_\Gamma$ is a finite set of external vertices.
\item $I_\Gamma$ is a linearly ordered finite set of internal vertices, disjoint from $A_\Gamma$.

Denote by $V_\Gamma:= A_\Gamma \sqcup  I_\Gamma$ the set of all vertices.
\item $E_\Gamma$ is a linearly ordered finite set of edges.  
\item $s_\Gamma,t_\Gamma:E_\Gamma \to V_\Gamma$ are the source and the target. Both of them are endpoints of edges. 
\item $W^V_\Gamma:V_\Gamma \to \C[\ZZi ]$ is the weight of vertices.
\item $W^E_\Gamma:E_\Gamma \to \C[\partial_{1},\cdots,\partial_{n}]$ is the weight of edges. 
\end{itemize}  
We adopt the following terminology and notation:
\begin{itemize}
\item Denote the set of endpoints of an edge $e$ by $\partial e$.
\item We extend the linear order of $I_\Gamma$ to a partial order on $V_\Gamma$ by letting $a<i$ when $a\in A_\Gamma$ and $i\in I_\Gamma$. 
\item Two distinct vertices are said to be adjacent if they are the endpoints of an edge.
\item We say that the edge $e$ is oriented from $s_\Gamma(e)$ to $t_\Gamma(e)$.
\item We divide the set of edges into the following four families:
\begin{itemize}
    \item a loop is an edge whose endpoints are identical;
    \item a chord is an edge between two distinct external vertices;
    \item a dead end is an edge that is not a loop and such that at least one of its endpoints is internal and has only one adjacent vertex;
    \item a contractible edge is an edge that is neither a chord, nor a loop, nor a dead end.
\end{itemize}
\item The {valence} of a vertex is the number of edges for which the vertex is
an endpoint, with loops adding two to the valence.
\item An edge $e$ is {simple} if there exists no other edge with the same set of endpoints.
\item {Double edges} are distinct edges with the same set of endpoints. 
\item Two vertices $v$ and $w$ are {connected} if there exists a path of edges joining them (ignoring orientations). 
\item Given a finite set $A$, a  {diagram on $A$} is a diagram $\Gamma$ such that $A_\Gamma=A$.
\item A diagram on $A$ is a {unit} if it has no internal vertices or edges.  We denote a unit by $\boldsymbol{1}$. 
\item Two diagrams  $\Gamma$ and  $\Gamma'$ are {isomorphic} if $E_\Gamma=E_{\Gamma'}$ and there exist two order-preserving bijections $\phi_E\colon E_\Gamma\cong E_{\Gamma'}$ and $\phi_I\colon I_\Gamma\cong I_{\Gamma'}$. 
\item We will abuse notation by denoting a diagram and its isomorphism class by the same letter $\Gamma$.
\end{itemize}
\end{defn}

\begin{defn} 
The space of diagrams on a set $A$ is the free $\C$-vector space $\GD(A)$ generated by  isomorphism classes of diagrams with external vertex set $A$, modulo the equivalence relation generated by edge reversals and transpositions in the linear orders of internal vertices and edges. This equivalence relation is $\C$-linear with respect to weights.  
\end{defn}
See \cite{Formality-little-disks}, Definition 6.5 for the sign convention. 

\begin{eg}\label{eg:diagram}
Consider the diagram in Figure \ref{fig:diagram}. All the external vertices are drawn on a horizontal line which is not a part of the diagram. We have omitted any trivial weight $1$. The picture represents a diagram $\Gamma$ with
\begin{itemize}
\item the set of external vertices $A_\Gamma=\lr{1,\cdots,5}$,
\item the set of internal vertices $I_\Gamma=\lr{6,\cdots,12}$,
\item the set $E_\Gamma$ of $11$ edges, with one loop, three dead ends, a chord, a pair of double contractible edges, and four simple contractible edges.
\item nontrivial weights of vertices $W^V_\Gamma(1)=(z^1)^2$, $W^V_\Gamma(12)=z^1z^2$,
\item nontrivial weights of edge $W^E_\Gamma((2,3))=\partial_1$, $W^E_\Gamma((8,9))=\partial_1\partial_2$.
\end{itemize}
\end{eg}

\begin{figure}[H]
    \centering
\tikzset{every picture/.style={line width=0.75pt}} 

\begin{tikzpicture}[x=0.75pt,y=0.75pt,yscale=-1,xscale=1]

\draw [color={rgb, 255:red, 128; green, 128; blue, 128 }  ,draw opacity=0.41 ]   (122.5,185) -- (439.5,185) ;
\draw  [fill={rgb, 255:red, 0; green, 0; blue, 0 }  ,fill opacity=1 ] (155,185) .. controls (155,183.34) and (156.34,182) .. (158,182) .. controls (159.66,182) and (161,183.34) .. (161,185) .. controls (161,186.66) and (159.66,188) .. (158,188) .. controls (156.34,188) and (155,186.66) .. (155,185) -- cycle ;
\draw  [fill={rgb, 255:red, 0; green, 0; blue, 0 }  ,fill opacity=1 ] (153.83,74.5) .. controls (153.83,72.84) and (155.18,71.5) .. (156.83,71.5) .. controls (158.49,71.5) and (159.83,72.84) .. (159.83,74.5) .. controls (159.83,76.16) and (158.49,77.5) .. (156.83,77.5) .. controls (155.18,77.5) and (153.83,76.16) .. (153.83,74.5) -- cycle ;
\draw  [fill={rgb, 255:red, 0; green, 0; blue, 0 }  ,fill opacity=1 ] (195,127) .. controls (195,125.34) and (196.34,124) .. (198,124) .. controls (199.66,124) and (201,125.34) .. (201,127) .. controls (201,128.66) and (199.66,130) .. (198,130) .. controls (196.34,130) and (195,128.66) .. (195,127) -- cycle ;
\draw  [fill={rgb, 255:red, 0; green, 0; blue, 0 }  ,fill opacity=1 ] (335,185) .. controls (335,183.34) and (336.34,182) .. (338,182) .. controls (339.66,182) and (341,183.34) .. (341,185) .. controls (341,186.66) and (339.66,188) .. (338,188) .. controls (336.34,188) and (335,186.66) .. (335,185) -- cycle ;
\draw  [fill={rgb, 255:red, 0; green, 0; blue, 0 }  ,fill opacity=1 ] (275,185) .. controls (275,183.34) and (276.34,182) .. (278,182) .. controls (279.66,182) and (281,183.34) .. (281,185) .. controls (281,186.66) and (279.66,188) .. (278,188) .. controls (276.34,188) and (275,186.66) .. (275,185) -- cycle ;
\draw  [fill={rgb, 255:red, 0; green, 0; blue, 0 }  ,fill opacity=1 ] (215,185) .. controls (215,183.34) and (216.34,182) .. (218,182) .. controls (219.66,182) and (221,183.34) .. (221,185) .. controls (221,186.66) and (219.66,188) .. (218,188) .. controls (216.34,188) and (215,186.66) .. (215,185) -- cycle ;
\draw  [fill={rgb, 255:red, 0; green, 0; blue, 0 }  ,fill opacity=1 ] (395,185) .. controls (395,183.34) and (396.34,182) .. (398,182) .. controls (399.66,182) and (401,183.34) .. (401,185) .. controls (401,186.66) and (399.66,188) .. (398,188) .. controls (396.34,188) and (395,186.66) .. (395,185) -- cycle ;
\draw  [fill={rgb, 255:red, 0; green, 0; blue, 0 }  ,fill opacity=1 ] (261.5,118.5) .. controls (261.5,116.84) and (262.84,115.5) .. (264.5,115.5) .. controls (266.16,115.5) and (267.5,116.84) .. (267.5,118.5) .. controls (267.5,120.16) and (266.16,121.5) .. (264.5,121.5) .. controls (262.84,121.5) and (261.5,120.16) .. (261.5,118.5) -- cycle ;
\draw  [fill={rgb, 255:red, 0; green, 0; blue, 0 }  ,fill opacity=1 ] (275,67) .. controls (275,65.34) and (276.34,64) .. (278,64) .. controls (279.66,64) and (281,65.34) .. (281,67) .. controls (281,68.66) and (279.66,70) .. (278,70) .. controls (276.34,70) and (275,68.66) .. (275,67) -- cycle ;
\draw    (198,127) -- (158,185) ;
\draw    (198,127) -- (218,185) ;
\draw    (198,127) -- (264.5,118.5) ;
\draw    (198,127) .. controls (210.5,102.25) and (244,98.25) .. (264.5,118.5) ;
\draw    (264.5,118.5) -- (278,185) ;
\draw    (221,185) .. controls (230.5,169.25) and (262,167.25) .. (278,185) ;
\draw  [fill={rgb, 255:red, 0; green, 0; blue, 0 }  ,fill opacity=1 ] (378.67,57.5) .. controls (378.67,55.84) and (380.01,54.5) .. (381.67,54.5) .. controls (383.32,54.5) and (384.67,55.84) .. (384.67,57.5) .. controls (384.67,59.16) and (383.32,60.5) .. (381.67,60.5) .. controls (380.01,60.5) and (378.67,59.16) .. (378.67,57.5) -- cycle ;
\draw  [fill={rgb, 255:red, 0; green, 0; blue, 0 }  ,fill opacity=1 ] (344.17,65.5) .. controls (344.17,63.84) and (345.51,62.5) .. (347.17,62.5) .. controls (348.82,62.5) and (350.17,63.84) .. (350.17,65.5) .. controls (350.17,67.16) and (348.82,68.5) .. (347.17,68.5) .. controls (345.51,68.5) and (344.17,67.16) .. (344.17,65.5) -- cycle ;
\draw    (347.17,65.5) .. controls (350.67,58.25) and (363.17,51.75) .. (381.67,57.5) ;
\draw    (347.17,65.5) .. controls (365.67,68.75) and (371.67,66.25) .. (381.67,57.5) ;
\draw  [fill={rgb, 255:red, 0; green, 0; blue, 0 }  ,fill opacity=1 ] (400,138) .. controls (400,136.34) and (401.34,135) .. (403,135) .. controls (404.66,135) and (406,136.34) .. (406,138) .. controls (406,139.66) and (404.66,141) .. (403,141) .. controls (401.34,141) and (400,139.66) .. (400,138) -- cycle ;
\draw    (278,67) -- (264.5,118.5) ;
\draw    (403,138) .. controls (416.5,136.25) and (434.5,120.25) .. (426,111.75) .. controls (417.5,103.25) and (405.21,126.37) .. (403,138) -- cycle ;
\draw    (403,138) -- (398,185) ;

\draw (152,189) node [anchor=north west][inner sep=0.75pt]   [align=left] {$1$};
\draw (212,189) node [anchor=north west][inner sep=0.75pt]   [align=left] {$2$};
\draw (272,189) node [anchor=north west][inner sep=0.75pt]   [align=left] {$3$};
\draw (332,189) node [anchor=north west][inner sep=0.75pt]   [align=left] {$4$};
\draw (392,189) node [anchor=north west][inner sep=0.75pt]   [align=left] {$5$};
\draw (142.67,71) node [anchor=north west][inner sep=0.75pt]   [align=left] {$6$};
\draw (186.5,115) node [anchor=north west][inner sep=0.75pt]   [align=left] {$7$};
\draw (271.33,115.33) node [anchor=north west][inner sep=0.75pt]   [align=left] {$8$};
\draw (266.33,51.33) node [anchor=north west][inner sep=0.75pt]   [align=left] {$9$};
\draw (328.33,67.33) node [anchor=north west][inner sep=0.75pt]   [align=left] {$10$};
\draw (387.33,53.67) node [anchor=north west][inner sep=0.75pt]   [align=left] {$11$};
\draw (378.67,129.67) node [anchor=north west][inner sep=0.75pt]   [align=left] {$12$};
\draw (130,166) node [anchor=north west][inner sep=0.75pt]   [align=left] {$(z_1)^2$};
\draw (408,136.5) node [anchor=north west][inner sep=0.75pt]   [align=left] {$z_1z_2$};
\draw (235,156) node [anchor=north west][inner sep=0.75pt]   [align=left] {$\partial_1$};
\draw (273,85) node [anchor=north west][inner sep=0.75pt]   [align=left] {$\partial_1\partial_2$};
\end{tikzpicture}
    \caption{An example of a diagram. 
    }
    \label{fig:diagram}
\end{figure}
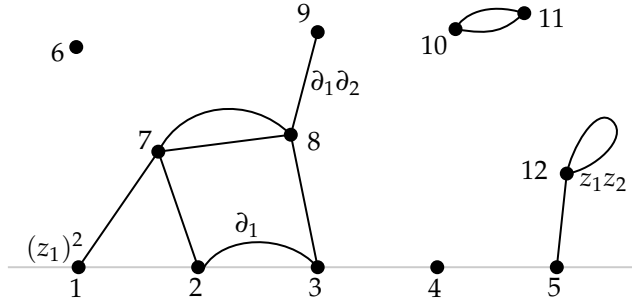

\begin{defn}
The degree of a diagram $\Gamma$ is defined by
$$
\deg \Gamma = |E_\Gamma| \cdot (n'+n-1) - |I_\Gamma| \cdot (n'+n). 
$$
\end{defn}

\begin{defn}[Product of diagrams]
Let $\Gamma_1$ and $\Gamma_2$ be two isomorphism classes of diagrams on the same set $A$. Define the product diagram $\Gamma_1 \cdot \Gamma_2$ by specifying the following data:
\begin{itemize}
\item $A_{\Gamma_1 \cdot \Gamma_2}=A$
\item $I_{\Gamma_1 \cdot \Gamma_2}=I_{\Gamma_1} \sqcup I_{\Gamma_2}$
\item $E_{\Gamma_1 \cdot \Gamma_2}=E_{\Gamma_1} \sqcup E_{\Gamma_2}$ 
\item $s_{\Gamma_1 \cdot \Gamma_2}|_{E_{\Gamma_i}}=s_{\Gamma_i}$
\item $t_{\Gamma_1 \cdot \Gamma_2}|_{E_{\Gamma_i}}=t_{\Gamma_i}$
\item $W^V_{\Gamma_1 \cdot \Gamma_2}|_{A}=W^V_{\Gamma_1} \cdot W^V_{\Gamma_2}$
\item $W^V_{\Gamma_1 \cdot \Gamma_2}|_{I_{\Gamma_i}}=W^V_{\Gamma_i}$
\item $W^E_{\Gamma_1 \cdot \Gamma_2}|_{E_{\Gamma_i}}=W^E_{\Gamma_i}$
\end{itemize} 
Here we assume that the sets $I_{\Gamma_1}$ and $I_{\Gamma_2}$, and $E_{\Gamma_1}$ and
$E_{\Gamma_2}$ respectively, are disjoint. The orders on linearly ordered sets are extended to the disjoint union of by letting $x<y$ for $x$ in the first summand and $y$ in the second summand.  
\end{defn}

\begin{eg}\label{eg:diagram-product}
An example of a product of two isomorphism classes if diagrams is presented in Figure \ref{fig:diagram-product}.

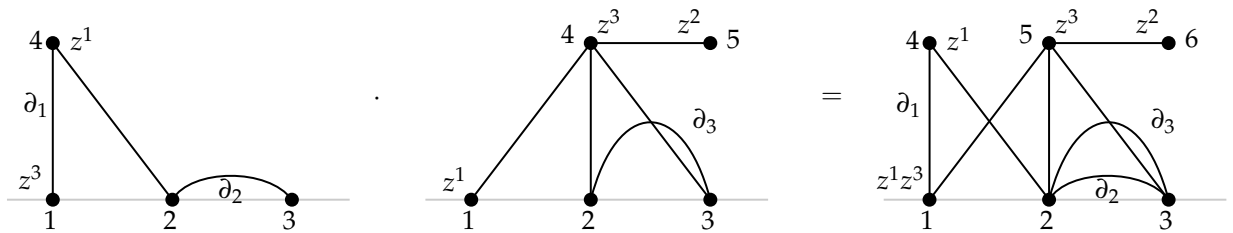
\begin{figure}[H]
    \centering
\tikzset{every picture/.style={line width=0.75pt}} 

\begin{tikzpicture}[x=0.75pt,y=0.75pt,yscale=-1,xscale=1]

\draw [color={rgb, 255:red, 128; green, 128; blue, 128 }  ,draw opacity=0.41 ]   (45,353) -- (217,353) ;
\draw  [fill={rgb, 255:red, 0; green, 0; blue, 0 }  ,fill opacity=1 ] (65,353) .. controls (65,351.34) and (66.34,350) .. (68,350) .. controls (69.66,350) and (71,351.34) .. (71,353) .. controls (71,354.66) and (69.66,356) .. (68,356) .. controls (66.34,356) and (65,354.66) .. (65,353) -- cycle ;
\draw  [fill={rgb, 255:red, 0; green, 0; blue, 0 }  ,fill opacity=1 ] (185,353) .. controls (185,351.34) and (186.34,350) .. (188,350) .. controls (189.66,350) and (191,351.34) .. (191,353) .. controls (191,354.66) and (189.66,356) .. (188,356) .. controls (186.34,356) and (185,354.66) .. (185,353) -- cycle ;
\draw  [fill={rgb, 255:red, 0; green, 0; blue, 0 }  ,fill opacity=1 ] (125,353) .. controls (125,351.34) and (126.34,350) .. (128,350) .. controls (129.66,350) and (131,351.34) .. (131,353) .. controls (131,354.66) and (129.66,356) .. (128,356) .. controls (126.34,356) and (125,354.66) .. (125,353) -- cycle ;
\draw    (68,274.4) -- (68,353) ;
\draw    (68,274.4) -- (128,353) ;
\draw    (338,274.4) -- (398,274.4) ;
\draw    (128,353) .. controls (137.5,337.25) and (176.2,336.8) .. (188,353) ;
\draw [color={rgb, 255:red, 128; green, 128; blue, 128 }  ,draw opacity=0.41 ]   (255,353) -- (427,353) ;
\draw  [fill={rgb, 255:red, 0; green, 0; blue, 0 }  ,fill opacity=1 ] (275,353) .. controls (275,351.34) and (276.34,350) .. (278,350) .. controls (279.66,350) and (281,351.34) .. (281,353) .. controls (281,354.66) and (279.66,356) .. (278,356) .. controls (276.34,356) and (275,354.66) .. (275,353) -- cycle ;
\draw  [fill={rgb, 255:red, 0; green, 0; blue, 0 }  ,fill opacity=1 ] (395,353) .. controls (395,351.34) and (396.34,350) .. (398,350) .. controls (399.66,350) and (401,351.34) .. (401,353) .. controls (401,354.66) and (399.66,356) .. (398,356) .. controls (396.34,356) and (395,354.66) .. (395,353) -- cycle ;
\draw  [fill={rgb, 255:red, 0; green, 0; blue, 0 }  ,fill opacity=1 ] (335,353) .. controls (335,351.34) and (336.34,350) .. (338,350) .. controls (339.66,350) and (341,351.34) .. (341,353) .. controls (341,354.66) and (339.66,356) .. (338,356) .. controls (336.34,356) and (335,354.66) .. (335,353) -- cycle ;
\draw [color={rgb, 255:red, 128; green, 128; blue, 128 }  ,draw opacity=0.41 ]   (485,353) -- (657,353) ;
\draw  [fill={rgb, 255:red, 0; green, 0; blue, 0 }  ,fill opacity=1 ] (505,353) .. controls (505,351.34) and (506.34,350) .. (508,350) .. controls (509.66,350) and (511,351.34) .. (511,353) .. controls (511,354.66) and (509.66,356) .. (508,356) .. controls (506.34,356) and (505,354.66) .. (505,353) -- cycle ;
\draw  [fill={rgb, 255:red, 0; green, 0; blue, 0 }  ,fill opacity=1 ] (625,353) .. controls (625,351.34) and (626.34,350) .. (628,350) .. controls (629.66,350) and (631,351.34) .. (631,353) .. controls (631,354.66) and (629.66,356) .. (628,356) .. controls (626.34,356) and (625,354.66) .. (625,353) -- cycle ;
\draw  [fill={rgb, 255:red, 0; green, 0; blue, 0 }  ,fill opacity=1 ] (565,353) .. controls (565,351.34) and (566.34,350) .. (568,350) .. controls (569.66,350) and (571,351.34) .. (571,353) .. controls (571,354.66) and (569.66,356) .. (568,356) .. controls (566.34,356) and (565,354.66) .. (565,353) -- cycle ;
\draw  [fill={rgb, 255:red, 0; green, 0; blue, 0 }  ,fill opacity=1 ] (65,274.4) .. controls (65,272.74) and (66.34,271.4) .. (68,271.4) .. controls (69.66,271.4) and (71,272.74) .. (71,274.4) .. controls (71,276.06) and (69.66,277.4) .. (68,277.4) .. controls (66.34,277.4) and (65,276.06) .. (65,274.4) -- cycle ;
\draw  [fill={rgb, 255:red, 0; green, 0; blue, 0 }  ,fill opacity=1 ] (395,274.4) .. controls (395,272.74) and (396.34,271.4) .. (398,271.4) .. controls (399.66,271.4) and (401,272.74) .. (401,274.4) .. controls (401,276.06) and (399.66,277.4) .. (398,277.4) .. controls (396.34,277.4) and (395,276.06) .. (395,274.4) -- cycle ;
\draw    (338,274.4) -- (338,353) ;
\draw    (338,274.4) -- (398,353) ;
\draw    (278,353) -- (338,274.4) ;
\draw  [fill={rgb, 255:red, 0; green, 0; blue, 0 }  ,fill opacity=1 ] (335,274.4) .. controls (335,272.74) and (336.34,271.4) .. (338,271.4) .. controls (339.66,271.4) and (341,272.74) .. (341,274.4) .. controls (341,276.06) and (339.66,277.4) .. (338,277.4) .. controls (336.34,277.4) and (335,276.06) .. (335,274.4) -- cycle ;
\draw    (338,353) .. controls (351,302.4) and (383.8,300) .. (398,353) ;
\draw    (508.03,274.4) -- (508.03,353) ;
\draw    (508.03,274.4) -- (568.03,353) ;
\draw    (568.03,353) .. controls (577.53,337.25) and (616.23,336.8) .. (628.03,353) ;
\draw  [fill={rgb, 255:red, 0; green, 0; blue, 0 }  ,fill opacity=1 ] (505.03,274.4) .. controls (505.03,272.74) and (506.38,271.4) .. (508.03,271.4) .. controls (509.69,271.4) and (511.03,272.74) .. (511.03,274.4) .. controls (511.03,276.06) and (509.69,277.4) .. (508.03,277.4) .. controls (506.38,277.4) and (505.03,276.06) .. (505.03,274.4) -- cycle ;
\draw    (568,274.4) -- (628,274.4) ;
\draw  [fill={rgb, 255:red, 0; green, 0; blue, 0 }  ,fill opacity=1 ] (625,274.4) .. controls (625,272.74) and (626.34,271.4) .. (628,271.4) .. controls (629.66,271.4) and (631,272.74) .. (631,274.4) .. controls (631,276.06) and (629.66,277.4) .. (628,277.4) .. controls (626.34,277.4) and (625,276.06) .. (625,274.4) -- cycle ;
\draw    (568,274.4) -- (568,353) ;
\draw    (568,274.4) -- (628,353) ;
\draw    (508,353) -- (568,274.4) ;
\draw  [fill={rgb, 255:red, 0; green, 0; blue, 0 }  ,fill opacity=1 ] (565,274.4) .. controls (565,272.74) and (566.34,271.4) .. (568,271.4) .. controls (569.66,271.4) and (571,272.74) .. (571,274.4) .. controls (571,276.06) and (569.66,277.4) .. (568,277.4) .. controls (566.34,277.4) and (565,276.06) .. (565,274.4) -- cycle ;
\draw    (568,353) .. controls (581,302.4) and (613.8,300) .. (628,353) ;

\draw (380.2,255.9) node [anchor=north west][inner sep=0.75pt]   [align=left] {$z^2$};
\draw (149.67,341.17) node [anchor=north west][inner sep=0.75pt]   [align=left] {$\partial_2$};
\draw (62,357) node [anchor=north west][inner sep=0.75pt]   [align=left] {$1$};
\draw (122,357) node [anchor=north west][inner sep=0.75pt]   [align=left] {$2$};
\draw (182,357) node [anchor=north west][inner sep=0.75pt]   [align=left] {$3$};
\draw (272,357) node [anchor=north west][inner sep=0.75pt]   [align=left] {$1$};
\draw (332,357) node [anchor=north west][inner sep=0.75pt]   [align=left] {$2$};
\draw (392,357) node [anchor=north west][inner sep=0.75pt]   [align=left] {$3$};
\draw (502,357) node [anchor=north west][inner sep=0.75pt]   [align=left] {$1$};
\draw (562,357) node [anchor=north west][inner sep=0.75pt]   [align=left] {$2$};
\draw (622,357) node [anchor=north west][inner sep=0.75pt]   [align=left] {$3$};
\draw (54.8,266.6) node [anchor=north west][inner sep=0.75pt]   [align=left] {$4$};
\draw (404.8,266.2) node [anchor=north west][inner sep=0.75pt]   [align=left] {$5$};
\draw (227.2,298) node [anchor=north west][inner sep=0.75pt]   [align=left] {$\cdot$};
\draw (321.6,263.8) node [anchor=north west][inner sep=0.75pt]   [align=left] {$4$};
\draw (75.4,264.7) node [anchor=north west][inner sep=0.75pt]   [align=left] {$z^1$};
\draw (339.8,255.9) node [anchor=north west][inner sep=0.75pt]   [align=left] {$z^3$};
\draw (49.4,334.3) node [anchor=north west][inner sep=0.75pt]   [align=left] {$z^3$};
\draw (262.2,335.1) node [anchor=north west][inner sep=0.75pt]   [align=left] {$z^1$};
\draw (452,298) node [anchor=north west][inner sep=0.75pt]   [align=left] {$=$};
\draw (387.67,306) node [anchor=north west][inner sep=0.75pt]   [align=left] {$\partial_3$};
\draw (52,297.83) node [anchor=north west][inner sep=0.75pt]   [align=left] {$\partial_1$};
\draw (589.7,341.17) node [anchor=north west][inner sep=0.75pt]   [align=left] {$\partial_2$};
\draw (494.83,266.6) node [anchor=north west][inner sep=0.75pt]   [align=left] {$4$};
\draw (515.43,264.7) node [anchor=north west][inner sep=0.75pt]   [align=left] {$z^1$};
\draw (480,334.3) node [anchor=north west][inner sep=0.75pt]   [align=left] {$z^1z^3$};
\draw (490,297.83) node [anchor=north west][inner sep=0.75pt]   [align=left] {$\partial_1$};
\draw (610.2,255.9) node [anchor=north west][inner sep=0.75pt]   [align=left] {$z^2$};
\draw (634.8,265.53) node [anchor=north west][inner sep=0.75pt]   [align=left] {$6$};
\draw (551.6,263.8) node [anchor=north west][inner sep=0.75pt]   [align=left] {$5$};
\draw (569.8,255.9) node [anchor=north west][inner sep=0.75pt]   [align=left] {$z^3$};
\draw (617.67,306) node [anchor=north west][inner sep=0.75pt]   [align=left] {$\partial_3$};
\end{tikzpicture}
    \caption{Example of a product of two diagrams.}
    \label{fig:diagram-product}
\end{figure}
\end{eg}

\begin{prop}
The product defined above extends to a degree $0$ linear map 
$$
\GD(A) \otimes \GD(A) \lra \GD(A)
$$
with unit $\boldsymbol{1}$, 
which endows $\GD(A)$ with the structure of a commutative $\Z$-graded algebra.  
\end{prop}

The differential on $\GD(A)$ is defined via contracting edges and transferring the weights of contracted edges to the remaining vertices/edges, corresponding to integration by parts.

\begin{defn}
Let $\Gamma$ be a diagram and let $e$ be a contractible edge of $\Gamma$. The diagram obtained from $\Gamma$ by contraction of the edge $e$ is the diagram $\Gamma/e$ defined as follows: 
\begin{itemize}
\item $A_{\Gamma/e}=A_\Gamma$
\item $I_{\Gamma/e}=I_{\Gamma} \setminus \lr{\max(s_\Gamma(e),t_{\Gamma}(e))}$
\item $E_{\Gamma/e}=E_\Gamma\setminus \lr{e}$
\item $s_{\Gamma/e}=q\circ s_\Gamma$ and $t_{\Gamma/e}=q\circ t_\Gamma$ where $q$ is defined by 
\ali{
q: V_\Gamma & \lra V_{\Gamma/e} \\
v & \lmt \begin{cases}
\min (s_\Gamma(e),t_{\Gamma}(e)) & v=\max(s_\Gamma(e),t_{\Gamma}(e)) \\
v & \mathrm{otherwise}
\end{cases}
}
where the linear orders on $I_{\Gamma/e}$ and $E_{\Gamma/e}$ are restrictions. 
\item $W^V_{\Gamma/e}=\begin{cases}
W^V_{\Gamma}(s_\Gamma(e))\cdot W^V_{\Gamma}(t_\Gamma(e)) & w=\min (s_\Gamma(e),t_{\Gamma}(e)) \\
W^V_{\Gamma}(w) & \text{else} \\
\end{cases}$
\item $W^E_{\Gamma/e}=W^E_{\Gamma}|_{E_{\Gamma/e}}$
\end{itemize} 
\end{defn}


\begin{defn}
Let $\Gamma$ be a diagram and let $v$ be a vertex of $\Gamma$. Define the operator $\partial_{v,i}$ by its action on $\Gamma$, which is the sum of diagrams  with the same vertices and edges and different weights 
$$
\partial_{v,i} \Gamma = \partial^v_{i} \Gamma + \sum_{e \text{ with source } v}  \partial^e_{i} \Gamma - \sum_{e \text{ with target } v}  \partial^e_{i} \Gamma
$$
where $\partial^v_{i} \Gamma$ is the diagram obtained from $\Gamma$ by apply $\partial_{z^i}$ to $W^V(v)$ and   $\partial^e_{i} \Gamma$ is the diagram obtained from $\Gamma$ by apply $\partial_{z^i}$ to $W^E(e)$.  
\end{defn}

Let $\Gamma$ be a diagram and let $e$ be a contractible edge of $\Gamma$.  Denote 
by $v_e$ be the vertex $\min(s_\Gamma(e),t_\Gamma(e))$ in $\Gamma/e$.

\begin{defn} 
Define the differential $d$ on $\GD(A)$ by 
$$
d(\Gamma) = \sum_{e \text{ contractable}} \pm W^E(e)(-\partial_{v_e,1},\cdots,-\partial_{v_e,n}) \cdot \Gamma/e
$$
where $\pm$ is the sign which comes from posets, see \cite{Formality-little-disks}, Example 6.10. 
\end{defn}

\begin{eg}\label{eg:diagram-differential}
Three examples of the diagram differential are presented in Figure \ref{fig:diagram-differential1} and Figure \ref{fig:diagram-differential2}. 
See Remark \ref{rem:Arnold} and Remark \ref{rem:local-relations} for algebraic relations of cohomology of configuration spaces encoded in the diagrams.

\begin{figure}[H]
    \centering
\tikzset{every picture/.style={line width=0.75pt}} 

\begin{tikzpicture}[x=0.75pt,y=0.75pt,yscale=-1,xscale=1]

\draw    (223,129.2) .. controls (229,95) and (253.83,95) .. (259.33,129.2) ;
\draw [color={rgb, 255:red, 128; green, 128; blue, 128 }  ,draw opacity=0.41 ]   (50,129.2) -- (167,129.2) ;
\draw  [fill={rgb, 255:red, 0; green, 0; blue, 0 }  ,fill opacity=1 ] (70,129.2) .. controls (70,127.54) and (71.34,126.2) .. (73,126.2) .. controls (74.66,126.2) and (76,127.54) .. (76,129.2) .. controls (76,130.86) and (74.66,132.2) .. (73,132.2) .. controls (71.34,132.2) and (70,130.86) .. (70,129.2) -- cycle ;
\draw  [fill={rgb, 255:red, 0; green, 0; blue, 0 }  ,fill opacity=1 ] (142.67,129.2) .. controls (142.67,127.54) and (144.01,126.2) .. (145.67,126.2) .. controls (147.32,126.2) and (148.67,127.54) .. (148.67,129.2) .. controls (148.67,130.86) and (147.32,132.2) .. (145.67,132.2) .. controls (144.01,132.2) and (142.67,130.86) .. (142.67,129.2) -- cycle ;
\draw  [fill={rgb, 255:red, 0; green, 0; blue, 0 }  ,fill opacity=1 ] (106.33,129.2) .. controls (106.33,127.54) and (107.68,126.2) .. (109.33,126.2) .. controls (110.99,126.2) and (112.33,127.54) .. (112.33,129.2) .. controls (112.33,130.86) and (110.99,132.2) .. (109.33,132.2) .. controls (107.68,132.2) and (106.33,130.86) .. (106.33,129.2) -- cycle ;
\draw    (109.33,50.6) -- (109.33,129.2) ;
\draw    (110,50.6) -- (145.67,129.2) ;
\draw    (73,129.2) -- (110,50.6) ;
\draw  [fill={rgb, 255:red, 0; green, 0; blue, 0 }  ,fill opacity=1 ] (106.33,50.6) .. controls (106.33,48.94) and (107.68,47.6) .. (109.33,47.6) .. controls (110.99,47.6) and (112.33,48.94) .. (112.33,50.6) .. controls (112.33,52.26) and (110.99,53.6) .. (109.33,53.6) .. controls (107.68,53.6) and (106.33,52.26) .. (106.33,50.6) -- cycle ;
\draw [color={rgb, 255:red, 128; green, 128; blue, 128 }  ,draw opacity=0.41 ]   (200,129.2) -- (317,129.2) ;
\draw  [fill={rgb, 255:red, 0; green, 0; blue, 0 }  ,fill opacity=1 ] (220,129.2) .. controls (220,127.54) and (221.34,126.2) .. (223,126.2) .. controls (224.66,126.2) and (226,127.54) .. (226,129.2) .. controls (226,130.86) and (224.66,132.2) .. (223,132.2) .. controls (221.34,132.2) and (220,130.86) .. (220,129.2) -- cycle ;
\draw  [fill={rgb, 255:red, 0; green, 0; blue, 0 }  ,fill opacity=1 ] (292.67,129.2) .. controls (292.67,127.54) and (294.01,126.2) .. (295.67,126.2) .. controls (297.32,126.2) and (298.67,127.54) .. (298.67,129.2) .. controls (298.67,130.86) and (297.32,132.2) .. (295.67,132.2) .. controls (294.01,132.2) and (292.67,130.86) .. (292.67,129.2) -- cycle ;
\draw  [fill={rgb, 255:red, 0; green, 0; blue, 0 }  ,fill opacity=1 ] (256.33,129.2) .. controls (256.33,127.54) and (257.68,126.2) .. (259.33,126.2) .. controls (260.99,126.2) and (262.33,127.54) .. (262.33,129.2) .. controls (262.33,130.86) and (260.99,132.2) .. (259.33,132.2) .. controls (257.68,132.2) and (256.33,130.86) .. (256.33,129.2) -- cycle ;
\draw [color={rgb, 255:red, 128; green, 128; blue, 128 }  ,draw opacity=0.41 ]   (350,129.2) -- (467,129.2) ;
\draw  [fill={rgb, 255:red, 0; green, 0; blue, 0 }  ,fill opacity=1 ] (370,129.2) .. controls (370,127.54) and (371.34,126.2) .. (373,126.2) .. controls (374.66,126.2) and (376,127.54) .. (376,129.2) .. controls (376,130.86) and (374.66,132.2) .. (373,132.2) .. controls (371.34,132.2) and (370,130.86) .. (370,129.2) -- cycle ;
\draw  [fill={rgb, 255:red, 0; green, 0; blue, 0 }  ,fill opacity=1 ] (442.67,129.2) .. controls (442.67,127.54) and (444.01,126.2) .. (445.67,126.2) .. controls (447.32,126.2) and (448.67,127.54) .. (448.67,129.2) .. controls (448.67,130.86) and (447.32,132.2) .. (445.67,132.2) .. controls (444.01,132.2) and (442.67,130.86) .. (442.67,129.2) -- cycle ;
\draw  [fill={rgb, 255:red, 0; green, 0; blue, 0 }  ,fill opacity=1 ] (406.33,129.2) .. controls (406.33,127.54) and (407.68,126.2) .. (409.33,126.2) .. controls (410.99,126.2) and (412.33,127.54) .. (412.33,129.2) .. controls (412.33,130.86) and (410.99,132.2) .. (409.33,132.2) .. controls (407.68,132.2) and (406.33,130.86) .. (406.33,129.2) -- cycle ;
\draw [color={rgb, 255:red, 128; green, 128; blue, 128 }  ,draw opacity=0.41 ]   (500,129.2) -- (617,129.2) ;
\draw  [fill={rgb, 255:red, 0; green, 0; blue, 0 }  ,fill opacity=1 ] (520,129.2) .. controls (520,127.54) and (521.34,126.2) .. (523,126.2) .. controls (524.66,126.2) and (526,127.54) .. (526,129.2) .. controls (526,130.86) and (524.66,132.2) .. (523,132.2) .. controls (521.34,132.2) and (520,130.86) .. (520,129.2) -- cycle ;
\draw  [fill={rgb, 255:red, 0; green, 0; blue, 0 }  ,fill opacity=1 ] (592.67,129.2) .. controls (592.67,127.54) and (594.01,126.2) .. (595.67,126.2) .. controls (597.32,126.2) and (598.67,127.54) .. (598.67,129.2) .. controls (598.67,130.86) and (597.32,132.2) .. (595.67,132.2) .. controls (594.01,132.2) and (592.67,130.86) .. (592.67,129.2) -- cycle ;
\draw  [fill={rgb, 255:red, 0; green, 0; blue, 0 }  ,fill opacity=1 ] (556.33,129.2) .. controls (556.33,127.54) and (557.68,126.2) .. (559.33,126.2) .. controls (560.99,126.2) and (562.33,127.54) .. (562.33,129.2) .. controls (562.33,130.86) and (560.99,132.2) .. (559.33,132.2) .. controls (557.68,132.2) and (556.33,130.86) .. (556.33,129.2) -- cycle ;
\draw    (373,129.2) .. controls (379,95) and (403.83,95) .. (409.33,129.2) ;
\draw    (223,129.2) .. controls (223.67,37) and (296.33,37.67) .. (295.67,129.2) ;
\draw    (409.33,129.2) .. controls (415.33,95) and (440.17,95) .. (445.67,129.2) ;
\draw    (523,129.2) .. controls (529,95) and (553.83,95) .. (559.33,129.2) ;
\draw    (523,129.2) .. controls (523.67,37) and (596.33,37.67) .. (595.67,129.2) ;

\draw (329,87) node [anchor=north west][inner sep=0.75pt]   [align=left] {$+$};
\draw (67,133.2) node [anchor=north west][inner sep=0.75pt]   [align=left] {$1$};
\draw (105,133.2) node [anchor=north west][inner sep=0.75pt]   [align=left] {$2$};
\draw (141.33,133.2) node [anchor=north west][inner sep=0.75pt]   [align=left] {$3$};
\draw (104.6,32) node [anchor=north west][inner sep=0.75pt]   [align=left] {$4$};
\draw (217,133.2) node [anchor=north west][inner sep=0.75pt]   [align=left] {$1$};
\draw (255,133.2) node [anchor=north west][inner sep=0.75pt]   [align=left] {$2$};
\draw (291.33,133.2) node [anchor=north west][inner sep=0.75pt]   [align=left] {$3$};
\draw (367,133.2) node [anchor=north west][inner sep=0.75pt]   [align=left] {$1$};
\draw (405,133.2) node [anchor=north west][inner sep=0.75pt]   [align=left] {$2$};
\draw (441.33,133.2) node [anchor=north west][inner sep=0.75pt]   [align=left] {$3$};
\draw (517,133.2) node [anchor=north west][inner sep=0.75pt]   [align=left] {$1$};
\draw (555,133.2) node [anchor=north west][inner sep=0.75pt]   [align=left] {$2$};
\draw (591.33,133.2) node [anchor=north west][inner sep=0.75pt]   [align=left] {$3$};
\draw (479,87) node [anchor=north west][inner sep=0.75pt]   [align=left] {$+$};
\draw (179,87) node [anchor=north west][inner sep=0.75pt]   [align=left] {$=$};
\draw (29,87) node [anchor=north west][inner sep=0.75pt]   [align=left] {$d$};
\end{tikzpicture}
    \caption{An example of the diagram differential.}
    \label{fig:diagram-differential1}
\end{figure}
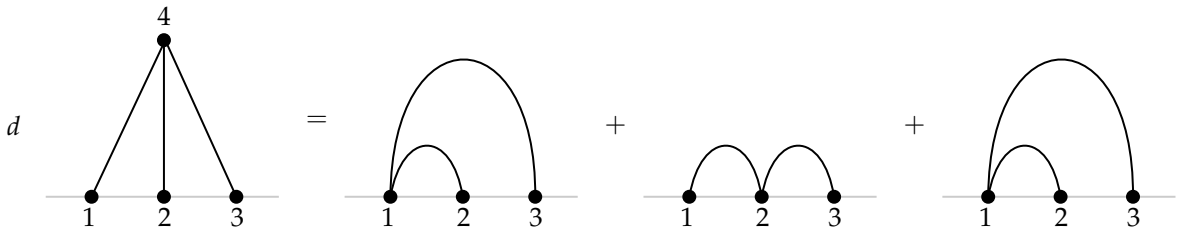

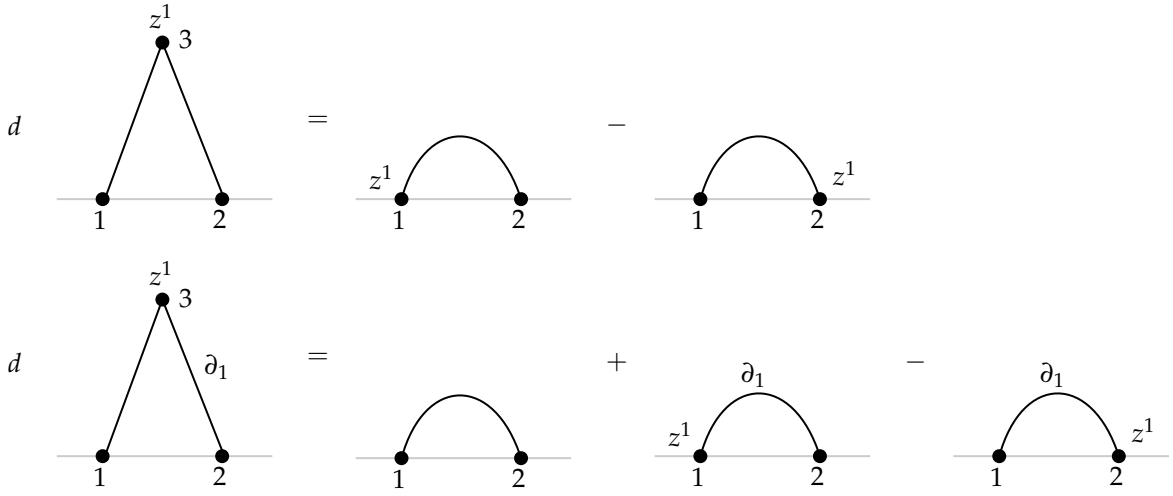
\begin{figure}[H]
    \centering
\tikzset{every picture/.style={line width=0.75pt}} 

\begin{tikzpicture}[x=0.75pt,y=0.75pt,yscale=-1,xscale=1]

\draw [color={rgb, 255:red, 128; green, 128; blue, 128 }  ,draw opacity=0.41 ]   (50,121.5) -- (158.2,121.5) ;
\draw  [fill={rgb, 255:red, 0; green, 0; blue, 0 }  ,fill opacity=1 ] (70,121.5) .. controls (70,119.84) and (71.34,118.5) .. (73,118.5) .. controls (74.66,118.5) and (76,119.84) .. (76,121.5) .. controls (76,123.16) and (74.66,124.5) .. (73,124.5) .. controls (71.34,124.5) and (70,123.16) .. (70,121.5) -- cycle ;
\draw  [fill={rgb, 255:red, 0; green, 0; blue, 0 }  ,fill opacity=1 ] (130,121.5) .. controls (130,119.84) and (131.34,118.5) .. (133,118.5) .. controls (134.66,118.5) and (136,119.84) .. (136,121.5) .. controls (136,123.16) and (134.66,124.5) .. (133,124.5) .. controls (131.34,124.5) and (130,123.16) .. (130,121.5) -- cycle ;
\draw    (103,42.9) -- (74.1,121.5) ;
\draw    (103,42.9) -- (134.1,121.5) ;
\draw    (223,121.5) .. controls (234.5,79.75) and (269.5,79.25) .. (283,121.5) ;
\draw  [fill={rgb, 255:red, 0; green, 0; blue, 0 }  ,fill opacity=1 ] (100,42.9) .. controls (100,41.24) and (101.34,39.9) .. (103,39.9) .. controls (104.66,39.9) and (106,41.24) .. (106,42.9) .. controls (106,44.56) and (104.66,45.9) .. (103,45.9) .. controls (101.34,45.9) and (100,44.56) .. (100,42.9) -- cycle ;
\draw [color={rgb, 255:red, 128; green, 128; blue, 128 }  ,draw opacity=0.41 ]   (200,121.5) -- (308.2,121.5) ;
\draw  [fill={rgb, 255:red, 0; green, 0; blue, 0 }  ,fill opacity=1 ] (220,121.5) .. controls (220,119.84) and (221.34,118.5) .. (223,118.5) .. controls (224.66,118.5) and (226,119.84) .. (226,121.5) .. controls (226,123.16) and (224.66,124.5) .. (223,124.5) .. controls (221.34,124.5) and (220,123.16) .. (220,121.5) -- cycle ;
\draw  [fill={rgb, 255:red, 0; green, 0; blue, 0 }  ,fill opacity=1 ] (280,121.5) .. controls (280,119.84) and (281.34,118.5) .. (283,118.5) .. controls (284.66,118.5) and (286,119.84) .. (286,121.5) .. controls (286,123.16) and (284.66,124.5) .. (283,124.5) .. controls (281.34,124.5) and (280,123.16) .. (280,121.5) -- cycle ;
\draw    (373,121.5) .. controls (384.5,79.75) and (419.5,79.25) .. (433,121.5) ;
\draw [color={rgb, 255:red, 128; green, 128; blue, 128 }  ,draw opacity=0.41 ]   (350,121.5) -- (458.2,121.5) ;
\draw  [fill={rgb, 255:red, 0; green, 0; blue, 0 }  ,fill opacity=1 ] (370,121.5) .. controls (370,119.84) and (371.34,118.5) .. (373,118.5) .. controls (374.66,118.5) and (376,119.84) .. (376,121.5) .. controls (376,123.16) and (374.66,124.5) .. (373,124.5) .. controls (371.34,124.5) and (370,123.16) .. (370,121.5) -- cycle ;
\draw  [fill={rgb, 255:red, 0; green, 0; blue, 0 }  ,fill opacity=1 ] (430,121.5) .. controls (430,119.84) and (431.34,118.5) .. (433,118.5) .. controls (434.66,118.5) and (436,119.84) .. (436,121.5) .. controls (436,123.16) and (434.66,124.5) .. (433,124.5) .. controls (431.34,124.5) and (430,123.16) .. (430,121.5) -- cycle ;
\draw [color={rgb, 255:red, 128; green, 128; blue, 128 }  ,draw opacity=0.41 ]   (50,250.5) -- (158.2,250.5) ;
\draw  [fill={rgb, 255:red, 0; green, 0; blue, 0 }  ,fill opacity=1 ] (70,250.5) .. controls (70,248.84) and (71.34,247.5) .. (73,247.5) .. controls (74.66,247.5) and (76,248.84) .. (76,250.5) .. controls (76,252.16) and (74.66,253.5) .. (73,253.5) .. controls (71.34,253.5) and (70,252.16) .. (70,250.5) -- cycle ;
\draw  [fill={rgb, 255:red, 0; green, 0; blue, 0 }  ,fill opacity=1 ] (130,250.5) .. controls (130,248.84) and (131.34,247.5) .. (133,247.5) .. controls (134.66,247.5) and (136,248.84) .. (136,250.5) .. controls (136,252.16) and (134.66,253.5) .. (133,253.5) .. controls (131.34,253.5) and (130,252.16) .. (130,250.5) -- cycle ;
\draw    (103,171.9) -- (74.1,250.5) ;
\draw    (103,171.9) -- (134.1,250.5) ;
\draw    (373,250.5) .. controls (384.5,208.75) and (419.5,208.25) .. (433,250.5) ;
\draw  [fill={rgb, 255:red, 0; green, 0; blue, 0 }  ,fill opacity=1 ] (100,171.9) .. controls (100,170.24) and (101.34,168.9) .. (103,168.9) .. controls (104.66,168.9) and (106,170.24) .. (106,171.9) .. controls (106,173.56) and (104.66,174.9) .. (103,174.9) .. controls (101.34,174.9) and (100,173.56) .. (100,171.9) -- cycle ;
\draw [color={rgb, 255:red, 128; green, 128; blue, 128 }  ,draw opacity=0.41 ]   (350,250.5) -- (458.2,250.5) ;
\draw  [fill={rgb, 255:red, 0; green, 0; blue, 0 }  ,fill opacity=1 ] (370,250.5) .. controls (370,248.84) and (371.34,247.5) .. (373,247.5) .. controls (374.66,247.5) and (376,248.84) .. (376,250.5) .. controls (376,252.16) and (374.66,253.5) .. (373,253.5) .. controls (371.34,253.5) and (370,252.16) .. (370,250.5) -- cycle ;
\draw  [fill={rgb, 255:red, 0; green, 0; blue, 0 }  ,fill opacity=1 ] (430,250.5) .. controls (430,248.84) and (431.34,247.5) .. (433,247.5) .. controls (434.66,247.5) and (436,248.84) .. (436,250.5) .. controls (436,252.16) and (434.66,253.5) .. (433,253.5) .. controls (431.34,253.5) and (430,252.16) .. (430,250.5) -- cycle ;
\draw    (523,250.5) .. controls (534.5,208.75) and (569.5,208.25) .. (583,250.5) ;
\draw [color={rgb, 255:red, 128; green, 128; blue, 128 }  ,draw opacity=0.41 ]   (500,250.5) -- (608.2,250.5) ;
\draw  [fill={rgb, 255:red, 0; green, 0; blue, 0 }  ,fill opacity=1 ] (520,250.5) .. controls (520,248.84) and (521.34,247.5) .. (523,247.5) .. controls (524.66,247.5) and (526,248.84) .. (526,250.5) .. controls (526,252.16) and (524.66,253.5) .. (523,253.5) .. controls (521.34,253.5) and (520,252.16) .. (520,250.5) -- cycle ;
\draw  [fill={rgb, 255:red, 0; green, 0; blue, 0 }  ,fill opacity=1 ] (580,250.5) .. controls (580,248.84) and (581.34,247.5) .. (583,247.5) .. controls (584.66,247.5) and (586,248.84) .. (586,250.5) .. controls (586,252.16) and (584.66,253.5) .. (583,253.5) .. controls (581.34,253.5) and (580,252.16) .. (580,250.5) -- cycle ;
\draw    (223,251.5) .. controls (234.5,209.75) and (269.5,209.25) .. (283,251.5) ;
\draw [color={rgb, 255:red, 128; green, 128; blue, 128 }  ,draw opacity=0.41 ]   (200,251.5) -- (308.2,251.5) ;
\draw  [fill={rgb, 255:red, 0; green, 0; blue, 0 }  ,fill opacity=1 ] (220,251.5) .. controls (220,249.84) and (221.34,248.5) .. (223,248.5) .. controls (224.66,248.5) and (226,249.84) .. (226,251.5) .. controls (226,253.16) and (224.66,254.5) .. (223,254.5) .. controls (221.34,254.5) and (220,253.16) .. (220,251.5) -- cycle ;
\draw  [fill={rgb, 255:red, 0; green, 0; blue, 0 }  ,fill opacity=1 ] (280,251.5) .. controls (280,249.84) and (281.34,248.5) .. (283,248.5) .. controls (284.66,248.5) and (286,249.84) .. (286,251.5) .. controls (286,253.16) and (284.66,254.5) .. (283,254.5) .. controls (281.34,254.5) and (280,253.16) .. (280,251.5) -- cycle ;

\draw (95,22) node [anchor=north west][inner sep=0.75pt]   [align=left] {$z^1$};
\draw (122.67,197.83) node [anchor=north west][inner sep=0.75pt]   [align=left] {$\partial_1$};
\draw (67,125.5) node [anchor=north west][inner sep=0.75pt]   [align=left] {$1$};
\draw (127,125.5) node [anchor=north west][inner sep=0.75pt]   [align=left] {$2$};
\draw (205.5,103) node [anchor=north west][inner sep=0.75pt]   [align=left] {$z^1$};
\draw (217,125.5) node [anchor=north west][inner sep=0.75pt]   [align=left] {$1$};
\draw (277,125.5) node [anchor=north west][inner sep=0.75pt]   [align=left] {$2$};
\draw (438,101) node [anchor=north west][inner sep=0.75pt]   [align=left] {$z^1$};
\draw (367,125.5) node [anchor=north west][inner sep=0.75pt]   [align=left] {$1$};
\draw (427,125.5) node [anchor=north west][inner sep=0.75pt]   [align=left] {$2$};
\draw (95,151) node [anchor=north west][inner sep=0.75pt]   [align=left] {$z^1$};
\draw (67,254.5) node [anchor=north west][inner sep=0.75pt]   [align=left] {$1$};
\draw (127,254.5) node [anchor=north west][inner sep=0.75pt]   [align=left] {$2$};
\draw (355.5,232) node [anchor=north west][inner sep=0.75pt]   [align=left] {$z^1$};
\draw (367,254.5) node [anchor=north west][inner sep=0.75pt]   [align=left] {$1$};
\draw (427,254.5) node [anchor=north west][inner sep=0.75pt]   [align=left] {$2$};
\draw (588,230) node [anchor=north west][inner sep=0.75pt]   [align=left] {$z^1$};
\draw (517,254.5) node [anchor=north west][inner sep=0.75pt]   [align=left] {$1$};
\draw (577,254.5) node [anchor=north west][inner sep=0.75pt]   [align=left] {$2$};
\draw (217,255.5) node [anchor=north west][inner sep=0.75pt]   [align=left] {$1$};
\draw (277,255.5) node [anchor=north west][inner sep=0.75pt]   [align=left] {$2$};
\draw (392,202) node [anchor=north west][inner sep=0.75pt]   [align=left] {$\partial_1$};
\draw (542,202) node [anchor=north west][inner sep=0.75pt]   [align=left] {$\partial_1$};
\draw (24,78) node [anchor=north west][inner sep=0.75pt]   [align=left] {$d$};
\draw (174,78) node [anchor=north west][inner sep=0.75pt]   [align=left] {$=$};
\draw (324,78) node [anchor=north west][inner sep=0.75pt]   [align=left] {$-$};
\draw (24,197) node [anchor=north west][inner sep=0.75pt]   [align=left] {$d$};
\draw (174,197) node [anchor=north west][inner sep=0.75pt]   [align=left] {$=$};
\draw (324,197) node [anchor=north west][inner sep=0.75pt]   [align=left] {$+$};
\draw (474,197) node [anchor=north west][inner sep=0.75pt]   [align=left] {$-$};
\draw (110,35) node [anchor=north west][inner sep=0.75pt]   [align=left] {$3$};
\draw (110,165) node [anchor=north west][inner sep=0.75pt]   [align=left] {$3$};

\end{tikzpicture}

    \caption{Another two examples of the diagram differential.}
    \label{fig:diagram-differential2}
\end{figure}
\end{eg}

\begin{lem}
The map $d$ defines a $\C$-linear map on  $\GD(A)$ of degree $+1$, satisfying $d^2=0$ and 
$$
d(\Gamma\cdot\Gamma')=d(\Gamma)\cdot\Gamma'+(-1)^{\deg(\Gamma)}\Gamma\cdot d(\Gamma').
$$
\end{lem} 

\begin{thm}
$(\GD(A),d)$ is a CDGA. 
\end{thm}

\begin{defn}[Admissible diagrams]
A diagram is called {admissible} if it has no loops, no double edges,
no internal vertices of valence $\leq 1$,
no internal vertices of valence $2$ with weight in $\C$, and each of its internal vertices is connected to some external vertex.
Otherwise a diagram is {non-admissible}. 
We denote by $\NAI({\ExtVert})$ the graded submodule of $\GD({\ExtVert})$ generated by the non-admissible diagrams.
\end{defn}
Thus, an admissible diagram consists only of simple chords and simple contractible edges.

\begin{lem}
The submodule of non-admissible diagrams $\NAI({\ExtVert})$ is a differential ideal of $\GD(A)$.
\end{lem}

\begin{defn}
The CDGA of admissible diagrams is defined as the quotient
$$
D(A) := \GD(A) / \NAI({\ExtVert}).
$$
\end{defn}


\subsubsection{Formality of the CDGA of Admissible Diagrams}

In this section, we show that the CDGA $\cD(A)$ of admissible diagrams on $A$ is quasi-isomorphic to the CDGA $\Hdef(\Conf_A(\aka))$ with zero differential. The proof is parallel to those in \cite{Formality-little-disks}, Section 8. 

Denote by
\begin{itemize}
\item $\VertexDiagram$ the diagram on $A$ with no internal vertices or edges, with
$$
W^V(v)=\begin{cases}
W & v=a \\
1 & \text{else}
\end{cases}
$$ 
\item $\chorddiagram$ the diagram on $A$ 
with no internal vertices, vertex weights identically $1$, and a single edge, which is a chord from $a$ to $b$ with weight $W$.
\end{itemize}

\begin{thm}\label{thm:diagram-cohomology}
There exists a quasi-isomorphism of $\Z$-graded CDGAs
$$
\bar \I: \cD(A) \slra \ConfSpaceCohomology
$$
characterized by the relations
$$
\begin{cases}
\bar \I (\VertexDiagram) = p_a^*W & \text{ for $a\in A$} \\
\bar \I (\chorddiagram) = (p_{ab}^*W)\omega_{ab} & \text{ for $a,b$ distinct in $A$} \\
\bar \I (\Gamma)= 0 & \text{for $\Gamma$ with internal vertices}
\end{cases}.
$$
\end{thm}

\begin{proof}
    
Let $\cD^{(0)}(A) \subset \cD(A)$ be the submodule generated by admissible  diagrams with no internal vertices. Then we have a surjective algebra map 
$$
\bar \I_0 : \cD^{(0)}(A) \lra \ConfSpaceCohomology.
$$
Moreover, we have $\bar \I_0(d\Gamma)=[\operatorname{I}(d\Gamma)]=[d\operatorname{I}(\Gamma)]=0$, where $\operatorname{I}$ is the configuration space integral in Section \ref{sec:Conf-Space-Integrals}. 
Thus $\bar \I_0$ extends to a CDGA morphism 
$$
\bar \I: \cD(A) \lra \ConfSpaceCohomology
$$
that annihilates diagrams with internal vertices and is surjective on cohomology. 

Next, we show $\bar \I$ induces an isomorphism on cohomology.

A diagram $\Gamma$ on ${\ExtVert}$ induces a partition of ${\ExtVert}$
into its connected components; we denote this partition by $\nu_\Gamma$. Let $\cGD(A)$ be the subcomplex of connected admissible diagrams on $A$. 
We have 
$$
\cD(A) \cong \oplus_{\nu} \otimes_{C \in \nu} \cGD(C).
$$
Moreover, $\ConfSpaceCohomology$ has a  similar decomposition, and $\bar \I$ preserves these decompositions. So  it suffices  to show that $\bar \I$ is injective on $\cGD(A)$. 


Fix an element $a\in A$. Consider the following submodules of $\cGD(A)$
\begin{itemize}
\item $\calU_0$ is the submodule  generated by connected admissible diagrams
with $a$ of valence $1$ and such that the only edge with endpoint $a$ is contractible;
\item $\calU_1$ is the submodule generated by connected admissible diagrams
with $a$ of valence $\geq2$;
\item $\cGD'(A)$ is the submodule generated by all connected admissible diagrams that are neither in $\calU_0$ nor in $\calU_1$.
\end{itemize}
Clearly $\cGD'(A)$ is a subcomplex of $\cGD(A)$. A spectral sequence argument in \cite{Formality-little-disks}, Lemma 8.5 shows the quotient complex associated to the inclusion 
$$
\cGD'(A) \lra \cGD(A)
$$
is acyclic, so the inclusion is a quasi-isomorphism. 
 
By induction on the number of vertices, we see $\cGD(A)$ is isomorphic to the corresponding partition of $\ConfSpaceCohomology$ in Corollary \ref{cor:Hdef-group-structure}. 

\end{proof}

\subsection{The Configuration Space Integrals}\label{sec:Conf-Space-Integrals}

Let $\Gamma$ be a diagram on $A$. Define 
\ali{
\theta_v: \Conf_{V_\Gamma}(\aka) &\lra \aka \\
(\ZZ_w)_{w\in V_\Gamma} &\lmt  \ZZ_v
}
\ali{
\theta_e: \Conf_{V_\Gamma}(\aka) &\lra \paka \\
(\ZZ_v)_{v\in V_\Gamma} &\lmt \ZZ_{s_\Gamma(e)}-\ZZ_{t_\Gamma(e)}
}
The configuration space integral $\GIK(\Gamma)$ is defined by the formula 
$$
\GIK(\Gamma) 
= \dashint \prod_{v\in I_\Gamma} d\zz_v \,  \prod_{v\in V_\Gamma} \theta_v^* (W^V_\Gamma(v)) \prod_{e\in E_\Gamma} \theta_e^* (W^E_\Gamma(e)\omega),
$$
where $\dashint$ is taken along fibers of a compactified version of  the projection  
$$
\Conf_{V_\Gamma}(\aka) \lra \Conf_{A}(\aka)
$$ 
which is defined in Section \ref{sec:reg-conf}.

\begin{prop}
$\GIK$ vanishes on 
non-admissible diagrams.
\end{prop}

\begin{proof}
$\GIK$ vanishes on diagrams with
\begin{itemize}
\item loops, because the pullback of $\omega_e$ for a loop $e$ is zero.
\item double edges, because $W^E_{\Gamma}(e) \omega \cdot W^E_{\Gamma}(e')\omega=0$ for double edges $e,e'$.
\item internal vertices of valence $\leq 1$, because the integrand does not contain a top form with respect to the vertex. 
\item internal vertices of valence $2$ with weight in $\C$, by Kontsevich’s trick from \cite{Kontsevich-diagrams}, Lemma 2.1.
\item an internal vertex not connected to any external vertices, because $\omega_{ij}$ relies only on coordinates of the diagonal $\ZZ_{ij}$ and these forms cannot form a top form. 
\end{itemize} 
\end{proof}

\begin{prop}\label{prop:diagram-to-form-map}
$\GIK$ induces a morphism of CDGAs 
$$
\GIK: \GD(A)  \lra \DefinableChainsA. 
$$
When modulo the ideal $\NAI({\ExtVert})$, it induces a morphism of CDGAs 
$$
\operatorname{I}: \cD(A)  \lra \DefinableChainsA. 
$$
\end{prop}

\begin{proof}[Sketch of proof]
Note that $\operatorname{I} (-)$ is constructible by Proposition \ref{prop:constructibility-reg-int}. 
By the fiberwise Stokes theorem, we obtain
$$
\bd \operatorname{I} (-) 
= \sum_{|S|=2}   \GlobalRes_{D_S}(-) + \sum_{|S|>2} \GlobalRes_{D_S}(-) +  \GlobalRes_{\infty}(-) 
$$
is the sum of residues. We will show that the last two terms vanish in Section \ref{sec:vanishing-DS} and \ref{sec:vanishing-infty}. 

For each $S=\lr{i,j}$, if $i$ and $j$ are not adjacent in the diagram, then the residue $\GlobalRes_{D_S}(-)$ vanishes. Otherwise, there is a contractible edge with endpoints $\partial e=S$, and the residue at $D_S$ is 
\ali{ 
\GlobalRes_{D_S}(-)
=& \GlobalRes_{D_S} \theta_e^* (W^E_\Gamma(e)\omega) (-) \\
=& \GlobalRes_{D_S} \theta_e^*\omega \cdot W^E(e)(-\partial_{v_e,1},\cdots,-\partial_{v_e,n})(-) \\
=& \dashint  W^E(e)(-\partial_{v_e,1},\cdots,-\partial_{v_e,n})(-)
} 
which corresponds to the contraction of $e$ in the differential. Here we have used the local calculation in Section \ref{sec:Asymptotic-zero} and Lemma \ref{vanish-lie-derivative}.

\end{proof}

\begin{rem}  
We have seen in Section \ref{sec:2-point-CDGA} that as a differential form on $\paka$, $\omega$ is generated by $\omega_{I}$ when $n'\le 1$. In this case, the differential on diagrams contains contractions of Laman subdiagrams, and the diagram complex is not formal.  

In particular,  when $n'=0$, in Lemma \ref{lem:omega-tree}, $\omega^e$ are represented by regularized integrals associated to diagrams of the form 
where internal vertices are bivalent and weighted by $z^{i_1},\cdots,z^{i_{n-2}}$. See Figure \ref{fig:omegae}. 

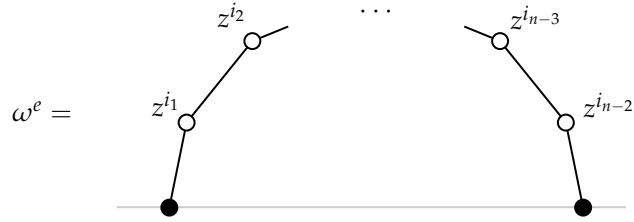
\begin{figure}[H]
    \centering
\tikzset{every picture/.style={line width=0.75pt}} 

\begin{tikzpicture}[x=0.75pt,y=0.75pt,yscale=-1,xscale=1]

\draw [color={rgb, 255:red, 128; green, 128; blue, 128 }  ,draw opacity=0.41 ]   (104,214.46) -- (359.5,214.46) ;

\draw  [fill={rgb, 255:red, 0; green, 0; blue, 0 }  ,fill opacity=1 ] (126.33,214.67) .. controls (126.33,212.46) and (128.12,210.67) .. (130.33,210.67) .. controls (132.54,210.67) and (134.33,212.46) .. (134.33,214.67) .. controls (134.33,216.88) and (132.54,218.67) .. (130.33,218.67) .. controls (128.12,218.67) and (126.33,216.88) .. (126.33,214.67) -- cycle ;
\draw    (139,172) -- (172,131.11) ;
\draw    (139,172) -- (130.33,214.67) ;
\draw  [fill={rgb, 255:red, 255; green, 255; blue, 255 }  ,fill opacity=1 ] (135,172) .. controls (135,169.79) and (136.79,168) .. (139,168) .. controls (141.21,168) and (143,169.79) .. (143,172) .. controls (143,174.21) and (141.21,176) .. (139,176) .. controls (136.79,176) and (135,174.21) .. (135,172) -- cycle ;
\draw  [fill={rgb, 255:red, 0; green, 0; blue, 0 }  ,fill opacity=1 ] (334,214.67) .. controls (334,212.46) and (335.79,210.67) .. (338,210.67) .. controls (340.21,210.67) and (342,212.46) .. (342,214.67) .. controls (342,216.88) and (340.21,218.67) .. (338,218.67) .. controls (335.79,218.67) and (334,216.88) .. (334,214.67) -- cycle ;
\draw    (329.33,172) -- (338,214.67) ;
\draw    (329.33,172) -- (296.33,131.11) ;
\draw  [fill={rgb, 255:red, 255; green, 255; blue, 255 }  ,fill opacity=1 ] (325.33,172) .. controls (325.33,169.79) and (327.12,168) .. (329.33,168) .. controls (331.54,168) and (333.33,169.79) .. (333.33,172) .. controls (333.33,174.21) and (331.54,176) .. (329.33,176) .. controls (327.12,176) and (325.33,174.21) .. (325.33,172) -- cycle ;
\draw    (190,123.78) -- (172,131.11) ;
\draw    (296.33,131.11) -- (278.33,123.78) ;
\draw  [fill={rgb, 255:red, 255; green, 255; blue, 255 }  ,fill opacity=1 ] (168,131.11) .. controls (168,128.9) and (169.79,127.11) .. (172,127.11) .. controls (174.21,127.11) and (176,128.9) .. (176,131.11) .. controls (176,133.32) and (174.21,135.11) .. (172,135.11) .. controls (169.79,135.11) and (168,133.32) .. (168,131.11) -- cycle ;
\draw  [fill={rgb, 255:red, 255; green, 255; blue, 255 }  ,fill opacity=1 ] (292.33,131.11) .. controls (292.33,128.9) and (294.12,127.11) .. (296.33,127.11) .. controls (298.54,127.11) and (300.33,128.9) .. (300.33,131.11) .. controls (300.33,133.32) and (298.54,135.11) .. (296.33,135.11) .. controls (294.12,135.11) and (292.33,133.32) .. (292.33,131.11) -- cycle ;

\draw (120,155) node [anchor=north west][inner sep=0.75pt]   [align=left] {$z^{i_1}$};
\draw (153,110.33) node [anchor=north west][inner sep=0.75pt]   [align=left] {$z^{i_2}$};
\draw (300,111.67) node [anchor=north west][inner sep=0.75pt]   [align=left] {$z^{i_{n-3}}$};
\draw (336.67,155.67) node [anchor=north west][inner sep=0.75pt]   [align=left] {$z^{i_{n-2}}$};
\draw (224,112) node [anchor=north west][inner sep=0.75pt]   [align=left] {$\cdots$};
\draw (50,160) node [anchor=north west][inner sep=0.75pt]   [align=left] {$\omega^e=$};
\end{tikzpicture} 
    \caption{Representation of the $1$-form $\omega^e$.}
    \label{fig:omegae}
\end{figure}

Moreover, $\omega$ is the sum of $n^{n-2}$ maximal trees, each of which corresponds to a diagram of the form in Figure \ref{fig:omega}. These diagrams, with $n^2-3n+2$ vertices and $n^2-2n+1$ edges, are Laman diagrams used in \cite{Feynman-Graph-Integral}.

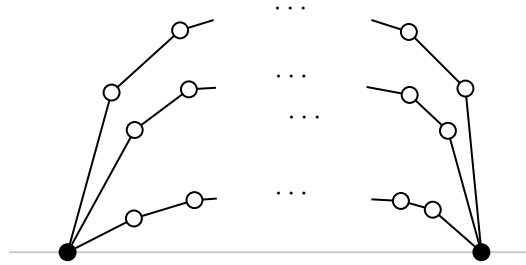
\begin{figure}[H]
    \centering 
\tikzset{every picture/.style={line width=0.75pt}} 

\begin{tikzpicture}[x=0.75pt,y=0.75pt,yscale=-1,xscale=1]

\draw [color={rgb, 255:red, 128; green, 128; blue, 128 }  ,draw opacity=0.41 ]   (185,196.67) -- (447.5,196.67) ;

\draw  [fill={rgb, 255:red, 0; green, 0; blue, 0 }  ,fill opacity=1 ] (210.33,196.67) .. controls (210.33,194.46) and (212.12,192.67) .. (214.33,192.67) .. controls (216.54,192.67) and (218.33,194.46) .. (218.33,196.67) .. controls (218.33,198.88) and (216.54,200.67) .. (214.33,200.67) .. controls (212.12,200.67) and (210.33,198.88) .. (210.33,196.67) -- cycle ;
\draw    (289.2,169.73) -- (278,170.53) -- (247.6,179.73) -- (214.33,196.67) ;
\draw  [fill={rgb, 255:red, 0; green, 0; blue, 0 }  ,fill opacity=1 ] (418,196.67) .. controls (418,194.46) and (419.79,192.67) .. (422,192.67) .. controls (424.21,192.67) and (426,194.46) .. (426,196.67) .. controls (426,198.88) and (424.21,200.67) .. (422,200.67) .. controls (419.79,200.67) and (418,198.88) .. (418,196.67) -- cycle ;
\draw    (288.8,114.13) -- (275.2,114.93) -- (248,135.33) -- (214.33,196.67) ;
\draw    (287.6,80.13) -- (270.8,85.33) -- (236.4,116.53) -- (214.33,196.67) ;
\draw    (366.8,80.13) -- (385.6,86.13) -- (414,114.53) -- (422,196.67) ;
\draw    (364.4,113.73) -- (386,117.73) -- (405.2,135.73) -- (422,196.67) ;
\draw    (366.8,170.13) -- (381.6,170.93) -- (397.6,175.33) -- (422,196.67) ;
\draw  [fill={rgb, 255:red, 255; green, 255; blue, 255 }  ,fill opacity=1 ] (232.4,116.53) .. controls (232.4,114.32) and (234.19,112.53) .. (236.4,112.53) .. controls (238.61,112.53) and (240.4,114.32) .. (240.4,116.53) .. controls (240.4,118.74) and (238.61,120.53) .. (236.4,120.53) .. controls (234.19,120.53) and (232.4,118.74) .. (232.4,116.53) -- cycle ;
\draw  [fill={rgb, 255:red, 255; green, 255; blue, 255 }  ,fill opacity=1 ] (266.8,85.33) .. controls (266.8,83.12) and (268.59,81.33) .. (270.8,81.33) .. controls (273.01,81.33) and (274.8,83.12) .. (274.8,85.33) .. controls (274.8,87.54) and (273.01,89.33) .. (270.8,89.33) .. controls (268.59,89.33) and (266.8,87.54) .. (266.8,85.33) -- cycle ;
\draw  [fill={rgb, 255:red, 255; green, 255; blue, 255 }  ,fill opacity=1 ] (244,135.33) .. controls (244,133.12) and (245.79,131.33) .. (248,131.33) .. controls (250.21,131.33) and (252,133.12) .. (252,135.33) .. controls (252,137.54) and (250.21,139.33) .. (248,139.33) .. controls (245.79,139.33) and (244,137.54) .. (244,135.33) -- cycle ;
\draw  [fill={rgb, 255:red, 255; green, 255; blue, 255 }  ,fill opacity=1 ] (271.2,114.93) .. controls (271.2,112.72) and (272.99,110.93) .. (275.2,110.93) .. controls (277.41,110.93) and (279.2,112.72) .. (279.2,114.93) .. controls (279.2,117.14) and (277.41,118.93) .. (275.2,118.93) .. controls (272.99,118.93) and (271.2,117.14) .. (271.2,114.93) -- cycle ;
\draw  [fill={rgb, 255:red, 255; green, 255; blue, 255 }  ,fill opacity=1 ] (243.6,179.73) .. controls (243.6,177.52) and (245.39,175.73) .. (247.6,175.73) .. controls (249.81,175.73) and (251.6,177.52) .. (251.6,179.73) .. controls (251.6,181.94) and (249.81,183.73) .. (247.6,183.73) .. controls (245.39,183.73) and (243.6,181.94) .. (243.6,179.73) -- cycle ;
\draw  [fill={rgb, 255:red, 255; green, 255; blue, 255 }  ,fill opacity=1 ] (274,170.53) .. controls (274,168.32) and (275.79,166.53) .. (278,166.53) .. controls (280.21,166.53) and (282,168.32) .. (282,170.53) .. controls (282,172.74) and (280.21,174.53) .. (278,174.53) .. controls (275.79,174.53) and (274,172.74) .. (274,170.53) -- cycle ;
\draw  [fill={rgb, 255:red, 255; green, 255; blue, 255 }  ,fill opacity=1 ] (381.6,86.13) .. controls (381.6,83.92) and (383.39,82.13) .. (385.6,82.13) .. controls (387.81,82.13) and (389.6,83.92) .. (389.6,86.13) .. controls (389.6,88.34) and (387.81,90.13) .. (385.6,90.13) .. controls (383.39,90.13) and (381.6,88.34) .. (381.6,86.13) -- cycle ;
\draw  [fill={rgb, 255:red, 255; green, 255; blue, 255 }  ,fill opacity=1 ] (410,114.53) .. controls (410,112.32) and (411.79,110.53) .. (414,110.53) .. controls (416.21,110.53) and (418,112.32) .. (418,114.53) .. controls (418,116.74) and (416.21,118.53) .. (414,118.53) .. controls (411.79,118.53) and (410,116.74) .. (410,114.53) -- cycle ;
\draw  [fill={rgb, 255:red, 255; green, 255; blue, 255 }  ,fill opacity=1 ] (382,117.73) .. controls (382,115.52) and (383.79,113.73) .. (386,113.73) .. controls (388.21,113.73) and (390,115.52) .. (390,117.73) .. controls (390,119.94) and (388.21,121.73) .. (386,121.73) .. controls (383.79,121.73) and (382,119.94) .. (382,117.73) -- cycle ;
\draw  [fill={rgb, 255:red, 255; green, 255; blue, 255 }  ,fill opacity=1 ] (401.2,135.73) .. controls (401.2,133.52) and (402.99,131.73) .. (405.2,131.73) .. controls (407.41,131.73) and (409.2,133.52) .. (409.2,135.73) .. controls (409.2,137.94) and (407.41,139.73) .. (405.2,139.73) .. controls (402.99,139.73) and (401.2,137.94) .. (401.2,135.73) -- cycle ;
\draw  [fill={rgb, 255:red, 255; green, 255; blue, 255 }  ,fill opacity=1 ] (377.6,170.93) .. controls (377.6,168.72) and (379.39,166.93) .. (381.6,166.93) .. controls (383.81,166.93) and (385.6,168.72) .. (385.6,170.93) .. controls (385.6,173.14) and (383.81,174.93) .. (381.6,174.93) .. controls (379.39,174.93) and (377.6,173.14) .. (377.6,170.93) -- cycle ;
\draw  [fill={rgb, 255:red, 255; green, 255; blue, 255 }  ,fill opacity=1 ] (393.6,175.33) .. controls (393.6,173.12) and (395.39,171.33) .. (397.6,171.33) .. controls (399.81,171.33) and (401.6,173.12) .. (401.6,175.33) .. controls (401.6,177.54) and (399.81,179.33) .. (397.6,179.33) .. controls (395.39,179.33) and (393.6,177.54) .. (393.6,175.33) -- cycle ;

\draw (323.2,124.53) node [anchor=north west][inner sep=0.75pt]   [align=left] {$\cdots$};
\draw (316.2,69.73) node [anchor=north west][inner sep=0.75pt]   [align=left] {$\cdots$};
\draw (316.6,104.13) node [anchor=north west][inner sep=0.75pt]   [align=left] {$\cdots$};
\draw (316.6,162.93) node [anchor=north west][inner sep=0.75pt]   [align=left] {$\cdots$};

\end{tikzpicture}
    \caption{The diagram associated to a maximal tree. Here we have omitted all weights.}
    \label{fig:omega}
\end{figure}
\end{rem}

\subsubsection{Vanishing of Residues at $D_S$, $|S|>2$} \label{sec:vanishing-DS}
In this section, we will temporarily allow $n'\in \Z_{\ge0}$ and study the residue associated to a diagram $\Gamma$ at $D_S$, where $|S|>2$. 
Let $\Gamma_S$ be the maximal subdiagram of $\Gamma$ with $V_{\Gamma_S}=S$. 

\begin{lem}\label{lem:Laman1}
For  $n'\in \Z_{\ge0}$ and $|S|>2$,  suppose $\Gamma_S$ has a subdiagram $\Gamma'$ with $|V_{\Gamma'}| \ge2$ and satisfying the inequality 
$$
(n'+n-1) |E_{\Gamma'}| - (n'+n) |V_{\Gamma'}| \le -n'-n-1.
$$ 
Then the residue $\GlobalRes_{D_S}(-)$ in Proposition \ref{prop:diagram-to-form-map} vanishes.
\end{lem}

\begin{proof}
Let $S=\lr{v_{i_1},\cdots,v_{i_k}}$. The differential form $\prod_{e\in E_{\Gamma_S}} \omega_e$ associated with $\Gamma_S$ has degree $(n'+n-1) |E_{\Gamma_S}|$. Moreover,  its form part depends only on $(n'+n)(|V_{\Gamma_S}|-1)$ components of $\ZZ_{i_1}-\ZZ_{i_k},\cdots \ZZ_{i_{k-1}}-\ZZ_{i_k}$ and is annihilated by contraction with the Euler vector field
$$
2 \sum_{j,l} \bar z_{i_j}^l \partial_{\bar z_{i_j}^l} + \sum_{j,l} \bar x_{i_j}^l \partial_{x_{i_j}^l}.
$$
 So the integrand vanishes if 
$$
(n'+n-1) |E_{\Gamma'}| > (n'+n)(|V_{\Gamma'}|-1) -1.
$$ 
\end{proof}

\begin{lem}\label{lem:Laman2}
For  $n'\in \Z_{\ge0}$ and $|S|>2$, the residue $\GlobalRes_{D_S}(-)$ in Proposition \ref{prop:diagram-to-form-map} vanishes unless
$$
(n'+n-1) |E_{\Gamma_S}| - (n'+n) |V_{\Gamma_S}| = -n'-n-1.
$$ 
\end{lem}

\begin{proof}
The integrand  contains a part $\prod_{v\in V_{\Gamma_S}} d\zz_v \prod_{e\in E_{\Gamma_S}} \omega_e$ of degree $n |V_{\Gamma_S}| + (n'+n-1) |E_{\Gamma_S}|$, while the residue requires degree $(n'+2n)(|V_{\Gamma_S}|-1)-1$.
\end{proof}

As a result, the residue vanishes if $\Gamma_S$ is not an ($n'+n$)-Laman diagram in the sense of \cite{Bud+23}:
\begin{defn}
A diagram $\Gamma$ is called ($n'+n$)-Laman if 
\begin{itemize}
\item For any subdiagram $\Gamma' \subset \Gamma$ with $|V_{\Gamma'}| \ge2$, the following inequality holds:
$$
(n'+n-1) |E_{\Gamma'}| - (n'+n) |V_{\Gamma'}| \le -n'-n-1.
$$
\item The following equality holds:
$$
(n'+n-1) |E_{\Gamma}| - (n'+n) |V_{\Gamma}| = -n'-n-1.
$$
\end{itemize}
\end{defn}





\begin{prop}
For $n'\ge2$ and $|S|>2$, the residue $\GlobalRes_{D_S}(-)$ in Proposition \ref{prop:diagram-to-form-map}  vanishes. 
\end{prop}

\begin{proof}
See \cite{topological-holomorphic}, Proposition 3.26. Here we have used the fact that regularized integral coincides with  Schwinger regularization.
\end{proof}


\begin{rem} 
For $n' \le 1$, the differential of the diagram complex contains contraction of ($n'+n$)-Laman subdiagrams.
\end{rem} 

\subsubsection{Vanishing of Residues at Infinity} \label{sec:vanishing-infty}
In this section, we prove the residue at infinity vanishes. Since an internal vertex in an admissible diagram has valence $l\ge 2$, we only need to show
$$
\GlobalRes_{\ZZ=\infty} d^n \zz \, f(\zz) \prod_{k=1}^{l}  \omega(\xx-\yy_k,\zz-\ww_k), \quad \ZZ=(\xx,\zz), \quad l\ge2
$$
vanishes. 
Here $\ZZ=\infty$ denotes a divisor at infinity.

By the computation in Section \ref{sec:bochner-infty}, the integrand takes the form 
\ali{
&d^n \zz \, f(\zz) \prod_{k=1}^{l}  \omega(\xx-\yy_k,\zz-\ww_k) \\
=& d\LR{\frac{z^1}{z^0}}\cdots d\LR{\frac{1}{z^0}} \cdots d\LR{\frac{z^n}{z^0}} f\LR{\frac{z^1}{z^0},\cdots,\frac{1}{z^0},\cdots,\frac{z^n}{z^0}} 
\prod_{k=1}^{l} \omega(\xx-\yy_k,\zz-\ww_k) \\
=& dz^0 \prod_{i\ne 0,p} dz^i \, \frac{(-1)^{p}}{(z^0)^{n+1}} f\LR{\frac{z^1}{z^0},\cdots,\frac{1}{z^0},\cdots,\frac{z^n}{z^0}} 
\prod_{k=1}^{l} \omega(\xx-\yy_k,\zz-\ww_k)\\
=& dz^0 (-) \prod_{k=1}^l \LR{\prod_{j\ne 0,p } d \bar z^j \prod_{i=1}^{n'} d x^i + O\LR{\frac{d\bar z^0}{\bar z^0},\frac{d|z^0|}{|z^0|}, \bar z^0, |z^0|}}.
}    
We therefore conclude that the residue vanishes, because 
$$
\int_{z^0\in S^1} dz^0 \frac{d\bar z^0}{\bar z^0} (-) 
=\int_{z^0\in S^1} dz^0 \frac{d|z^0|}{|z^0|} (-) 
=0
$$
and 
$$
\prod_{k=1}^l \LR{\prod_{j\ne 0,p } d \bar z^j \prod_{i=1}^{n'} d x^i}=0.
$$


\subsection{The Formality Theorem}\label{sec:formality-thm}
\begin{thm}\label{thm:formality}
The CDGA $\Adef(\maka)$ is formal.
\end{thm}

\begin{proof}
Let $A=\lr{1,\cdots,m}$. 
By Theorem \ref{thm:diagram-cohomology}, there exists a quasi-isomorphism 
$$
\bar{\operatorname{I}}: \cD(A) \slra \ConfSpaceCohomology. 
$$
Moreover, the morphism of CDGAs
$$
\operatorname{I}: \cD(A)  \lra \DefinableChainsA. 
$$
maps the generators $\VertexDiagram,\chorddiagram$ of cohomology to the corresponding generators on the right hand side, so $\operatorname{I}$ is a quasi-isomorphism. 
\end{proof}

\begin{rem}
It is easy to see that the two quasi-isomorphisms in the above theorem preserve the $D_{\C^{mn}}$-module structures, so $\Aconk(\maka)$ is formal as a $D_{\C^{mn}}$-module. 
\end{rem}
 
\begin{rem}\label{rem:formality}
The  CDGA $\A(\maka)$ is also formal for the following reason.

Using a completion $\widetilde\cD(A)$ of $\cD(A)$ with respect to the nuclear topology of smooth differential forms on the product space of all weights of vertices and edges, we have a quasi-isomorphism of CDGAs
$$
\widetilde{\operatorname{I}}: \widetilde\cD(A)  \slra \A(\Conf_A(\aka)). 
$$
This, together with a completed version of the quasi-isomorphism in  Theorem \ref{thm:diagram-cohomology}  
$$
\widetilde{\bar \I}: \widetilde\cD(A) \slra H^\bullet(\Conf_A(\aka)),
$$
implies $\A(\maka)$ is formal. 
\end{rem}

\section{Application: Higher-dimensional Vertex Algebras}  
\label{sec:VA}

When $n'\ge2,n\ge1$, the configuration space of $\aka$ is formal. Consequently, the algebra of local observables associated with a topological-holomorphic field theory on $\aka$ takes 
a simple form, providing a higher-dimensional analogue of a vertex algebra.

We set the following notation following \cite{Cohomological-VA}: 
\begin{itemize}
\item A vector in $\Z^n$ is denoted by a bold symbol $\kk=(k_1,\cdots,k_n)$.
\item We denote $\bs{0}=(0,\cdots,0)$ and $\bs{1}=(1,\cdots,1)$, ect.
\item We write $\kk \ge \jj$ to mean $k_i \ge j_i$ for all $i=1,\cdots,n$.
\item Denote 
$$
\kk! := \prod_{i=1}^n k_i!, \qquad 
\binom{\kk}{\jj} = \prod_{i=1}^n \binom{k_i}{j_i}, \qquad 
(-1)^\kk:=\prod_{i=1}^n (-1)^{k_i},
$$
$$
\zz^\kk := \prod_{i=1}^n (z^i)^{k_i}, \qquad 
\partial_{\zz}^{(\kk)} := \frac{1}{\kk!} \partial_{\zz}^{\kk} : = \frac{1}{\kk!} \prod_{i=1}^n \partial_{z^i}^{k_i},  \qquad
\Omegaz^{\kk}=
\begin{cases}
(-1)^\kk \partial_{\zz}^{(\kk)} \omega(\ZZ) & \kk\ge\bs{0} \\
0 & \text{else}
\end{cases}.
$$ 
\end{itemize}

Denote $\Kdisk=\Hconk(\paka)$. It 
has a basis 
$$
\lr{\zz^{\kk}, \quad \Omegaz^{\kk}}_{\kk\ge \bs0}
$$
with $\deg \zz^{\kk} = 0$, $\deg \Omegaz^{\kk} = n'+n-1$. The product structure is given by the relation
$$
\begin{cases}
\zz^{\kk_1} \cdot \zz^{\kk_2} = \zz^{\kk_1+\kk_2} \\
\zz^{\kk_1} \cdot \Omegaz^{\kk_2} = \Omegaz^{\kk_2-\kk_1} \\
\Omegaz^{\kk_1} \cdot \Omegaz^{\kk_2} = 0 \\
\end{cases}.
$$

\begin{rem}
The coefficient notations above are more explicit in \v Cech cohomology, in which the \v Cech representative of $\Omegaz^{\kk}$ is $\frac{1}{\zz^{\kk+\bs{1}}}$.
\end{rem}


For $n'\ge2,n\ge1$, we define the notion of $(n',n)$-vertex algebras associated with $\aka$. 


\begin{defn}
An $(n',n)$-vertex algebra consists of the data $(V, Y, \TT, \VAvacuum)$ with
\begin{itemize}
\item (space of states) a $\Z$-graded $\C$-vector space $V=\oplus_{r\in \Z} V^r$.
\item (vacuum vector) a degree $0$ element $\VAvacuum\in V^0$.
\item (translation operator) a vector $\TT=(T_1,\cdots,T_n)$, where $T_i\in \End V$ are degree $0$ endomorphisms of $V$. 
\item (vertex operators) a degree $0$ $\C$-linear map 
\ali{
Y: V &\lra \Kdisk \otimes \End V \\
   a &\lmt Y(a,\zz) = \sum_{\kk \ge \bs{0}} \zz^{\kk} a_{(-\bs{1}-\kk)} + \Omegaz^\kk a_{(\kk)} 
}
which assigns a field $Y(a,\zz)$ to $a$. 

The condition that $Y(a,\zz)$ is a field means that for every $v\in V$, there exists a vector $\jj \ge \bs0$ such that $a_{(\kk)}v=0$ whenever $\kk>\jj$. 
$Y$ has degree $0$  means 
$$
\deg a_{(-\bs{1}-\kk)}=\deg a, \qquad \deg a_{(\kk)} = \deg a -(n'+n-1).
$$
\end{itemize}
The data must satisfy the following axioms:
\begin{itemize}
\item (vacuum axiom) $Y(\VAvacuum,\zz) = \Id_V$. Furthermore, for any $a\in V$ we have 
$$
Y(a,\zz) \VAvacuum \in \C[[z^1,\cdots,z^n]] \otimes V,
$$
and 
$$
Y(a,\zz) \VAvacuum |_{\zz=0}= a.
$$
\item (translation axiom) For any $a\in V$, 
$$
[T_i,Y(a,\zz)] = \partial_{z^i} Y(a,\zz),
$$
and for every $i$, we have
$$T_i \VAvacuum =0.
$$
\item (locality axiom) All fields $Y(a,\zz)$ are mutually local in the sense of \cites{Cohomological-VA}. 
\end{itemize}
\end{defn}

There are natural notions of ideals, morphisms, subalgebras, tensor products, direct sums, and modules of $(n',n)$-vertex algebras, and a reconstruction theorem holds.

\appendix
\section{Locally H\"older Continuity of Continuous Constructible Functions}

\begin{lemma} \label{lem:constructible-pre}
Let $E \subset \R^m \times [0,\epsilon_0]$ be compact and globally subanalytic. Put $E_0 = E \cap (\R^m \times \{0\})$. Suppose $F: E \to \R$ is continuous, constructible, and vanishes on $E_0$. Then there exist constants $C>0$, $\alpha>0$, and $\epsilon \in (0,\epsilon_0]$ such that
$$
|F(z,t)| \le C t^{\alpha}
$$
for every $(z,t)\in E$ with $0<t<\epsilon$.
\end{lemma}

\begin{proof}
Since $F$ is constructible, it is a finite sum of terms of the form
$$
f_i  \prod_{j=1}^{l_i} \log g_{ij},
$$
where the $f_i$ and $g_{ij}>0$ are globally subanalytic.

We apply the parameterized rectilinearization theorem of Cluckers–Miller \cite{cluckers-miller-2013-lebesgue} to to $f_i$, $g_{ij}$, and the $m+1$ coordinate functions on $E$. This gives a finite partition of the region $D = E \cap \{t>0\}$ into globally subanalytic pieces $A$, and for every such $A$, after suppressing coordinates that are identically determined, there is an analytic globally subanalytic
isomorphism
$$
\Phi:B\longrightarrow A,
\qquad
B=P\times(0,1)^s,
$$
where
$$
\overline P\Subset(0,1]^l.
$$
Denote coordinates of $B$ by $w=(w_1,\cdots,w_l)$, $r=(r_1,\cdots,r_s)$. 
Every nonzero globally subanalytic function $g$ included in the
rectilinearization has a prepared expression
$$
g\circ\Phi(w,r)
=c_g\,w^{a_g}r^{b_g}u_g(w,r),
$$
where $c_g$ is a nonzero constant, $a_g$ and $b_g$ are vectors with rational entries, and $u_g$ is a positive analytic unit obtained by composing an analytic function, defined near a compact set, with finitely many bounded rational monomials. Consequently, there are constants $c'_g,C'_g>0$ such
that
$$
0<c'_g\le u_g(w,r)\le C'_g
$$
on $B$, and $u_g$ extends continuously to
$\overline P\times[0,1]^s$.

The coordinate functions of $\Phi$ are bounded. Their prepared
monomials can therefore have no negative exponent in any 
$r_j$. Hence $\Phi$ extends continuously to
$$
\overline B=\overline P\times[0,1]^s,
$$
and the image of this extension lies in $\overline A\subset E$.

For the coordinate $t$, the prepared expression takes the form
$$
t\circ\Phi(w,r)
=v(w,r)\prod_{j=1}^s r_j^{b_j},
$$
where $b_j\in\mathbb Q_{\ge0}$ and
$$
0<C_1\le v(w,r)\le C_2
$$
for suitable constants $C_1,C_2>0$. If every $b_j$ vanishes, then
$t$ is bounded away from zero on $A$. Such a piece does not meet
$\{0<t<\epsilon\}$ after $\epsilon$ is sufficiently
decreased. It remains to consider pieces for which
$$
I:=\set{j}{b_j>0}\ne\varnothing.
$$

The preparation gives
$$
\log(g_{ij} \circ\Phi)
=\log c_{ij}
+a_{ij}\cdot\log w
+b_{ij}\cdot\log r
+\log u_{ij}(w,r)
$$
where $a_{ij}$, $b_{ij}$ are vectors with rational entries, and $\cdot$ denotes inner product of vectors. 
Since $\overline P\Subset(0,1]^l$, every $\log w_j$ is analytic
near $\overline P$. The term $\log u_{ij}$ is analytic in
the bounded monomials supplied by the rectilinearization. This allows us to write
$$
F\circ\Phi(w,r)
=\sum_{(q,p)\in S}
r^q(\log r)^p A_{q,p}(w,\rho),
\qquad
\rho_j=r_j^{1/N},
$$
where $S$ is a finite index set and every $A_{q,p}$ is analytic
on a neighborhood of 
$
\overline P\times[0,1]^s.
$ 
Thus the coefficients are uniformly controlled up to all chart frontiers.

Because $F$ is bounded and vanishes when any of the $r_k$ tends to $0$, a careful Taylor expansion in each variable shows that all terms with negative exponents or logarithmic singularities must cancel. Consequently, on each piece $A$, we obtain an estimate
$$
|F\circ \Phi(w,r)| \le C' \prod_{k\in I} r_k^{\delta_{k}}
$$
where $I$ is the set of variables that tend to $0$ when $t\to 0$, and $C',\delta_{k}>0$ only depend on $A$.

Finally, using the prepared expression for $t$ itself
$$
t\circ\Phi(w,r)
=v(w,r)\prod_{j=1}^s r_j^{b_j},
$$
we derive
$$
|F(z,t)| \le C" t^{\alpha}
$$
for some $\alpha$ on each piece. Taking the minimum of the finitely many exponents and cutoffs gives the uniform result.
\end{proof}

\begin{prop}
\label{prop:Holder-continuous}
Let $U\subset\R^n$ be an open globally subanalytic set, and let $f:U\to\R$ be a continuous constructible function. Then for every compact subset $K\Subset U$, there exist constants $C>0$, $\alpha>0$, and $r_0>0$ such that
$$
|f(x)-f(y)|\leq C\|x-y\|^\alpha
$$
for all $x,y\in K$ satisfying $\|x-y\|<r_0$.
\end{prop}

\begin{proof} The proof is divided in three steps. 

Step 1. 
For each $x\in K$, choose a closed axis‑parallel box $Q_x \subset U$ with $x$ in its interior. By compactness, there exist finitely many boxes $Q_1,\dots,Q_N$ such that $K \subset \bigcup_j \operatorname{int} Q_j$.

Step 2. 
Fix one box $Q_j$ and a coordinate direction $e_i$. Consider the set
$$
E_{j,i} = \set{(z,t) \in Q_j \times [0,1] }{ z + t e_i \in Q_j},
$$
and the function
$$
F_{j,i}(z,t) = f(z + t e_i) - f(z).
$$
The set is compact and globally subanalytic; the function is continuous, constructible, and vanishes at $t=0$. Lemma \ref{lem:constructible-pre} yields constants $C_{j,i}>0$, $\alpha_{j,i}>0$, and $\epsilon_{j,i}>0$ such that
$$
|f(z + t e_i) - f(z)| \le C_{j,i}\, t^{\alpha_{j,i}}
$$
whenever $z, z+t e_i \in Q_j$ and $0<t<\epsilon_{j,i}$.

Step 3. 
Let $\lambda>0$ be a Lebesgue number for the open cover $\{K\cap \operatorname{int}Q_j\}$ of the compact metric space $K$. Thus any two points of $K$ at distance $<\lambda$ lie together in the interior of some $Q_j$.  
Define
$$
\alpha = \min_{j,i} \alpha_{j,i} > 0, \qquad C_0 = \max_{j,i} C_{j,i},
$$
and choose
$$
0<r_0 < \min\{1, \lambda, \min_{j,i}\epsilon_{j,i}\}.
$$

For $x,y \in K$ with $\|x-y\|<r_0$, pick a box $Q_j$ containing both. Join $x$ to $y$ by the coordinate polygon
$$
x^{(0)}=x, \quad x^{(k)} = (y_1,\dots,y_k, x_{k+1},\dots,x_n), \quad 1\le k\le n.
$$
Since $Q_j$ is convex, every segment lies in $Q_j$. Along the $k$-th segment, the previous estimate yields
$$
|f(x^{(k)}) - f(x^{(k-1)})| \le C_0 |x_k-y_k|^{\alpha_{j,k}} \le C_0 \|x-y\|^\alpha,
$$
where the last inequality uses $|x_k-y_k|\le \|x-y\|<1$ and $\alpha_{j,k}\ge \alpha$. Telescoping yields
$$
|f(x)-f(y)| \le \sum_{k=1}^n |f(x^{(k)}) - f(x^{(k-1)})| \le n C_0 \|x-y\|^\alpha.
$$
Thus the proposition holds with $C = n C_0$.
\end{proof} 

\begin{rem}
Proposition \ref{prop:Holder-continuous} does not follow directly from the H\"older  continuity of continuous globally subanalytic functions \cite{LC2}, because for a continuous constructible function
$$
f=\sum_{i=1}^k f_i \prod_{j=1}^{l_i} \log(g_{ij}),
$$
the globally subanalytic functions $f_i,g_{ij}$ in the representation need not be continuous.
\end{rem}


\begin{bibdiv}
\begin{biblist}

\bibitem{chiral-algebra}
A. Beilinson and V. Drinfeld, {\em Chiral algebras}, American Mathematical Society Colloquium Publications, vol. 51, American Mathematical Society, Providence, RI, 2004.


\bibitem{Esquisse}
A. Grothendieck, 1984. {\em "Esquisse d'un Programme"}, 1984 manuscript, finally published in Schneps and Lochak (1997, I), pp.5-48; English transl., ibid., pp. 243-283.

\bibitem{o-minimal-de-rham}
A. Huber, T. Kaiser and A. Oswal, {\em On the de Rham theorem in the globally subanalytic setting}, arXiv:2508.03499[math.LO]. 


\bibitem{blowup-corner}
A. K\"onig, {\em The blowup of complex manifolds with corners}, Master Thesis, published by the Library of the University of Regensburg, DOI 10.5283/epub.47792.




\bibitem{Secondary-Products}
C. Beem, D. Ben-Zvi,  M. Bullimore, et al, 
{\em Secondary Products in Supersymmetric Field Theory}, Ann. Henri Poincaré 21, 1235–1310 (2020).

\bibitem{Cohomological-VA}
C. Griffin, {\em Cohomological vertex algebras}, arXiv:2501.18720[math.QA]. 

\bibitem{poly-bound2} C. Miller, {\em Expansions of the real field with power functions}, Ann. Pure Appl. Logic 68 (1994), 79-94.

\bibitem{Tamarkin-formality}
D. E. Tamarkin, {\em Another proof of M. kontsevich formality theorem}, arxiv:math/9803025, 1998.


\bibitem{Complex-geometry} D. Huybrechts, {\em Complex geometry}, Universitext, Springer, Berlin, 2005.





\bibitem{algebraic-operads}
J.-L. Loday and B. Vallette, {\it Algebraic operads}, Grundlehren der mathematischen Wissenschaften, 346, Springer, Heidelberg, 2012


\bibitem{Bud+23}
K. Budzik, D. Gaiotto, J. Kulp, J. Wu, and M. Yu,  {\em Feynman diagrams in four-dimensional holomorphic theories and the Operatope}, JHEP 07 (2023), p. 127. arXiv: 2207.14321 [hep-th].

\bibitem{kevin-owen}
K.~Costello and O.~Gwilliam, {\em Factorization algebras in quantum field theory. {V}ol. 1,2},   volume~31 of {\em New Mathematical Monographs}, Cambridge University Press, Cambridge, 2017.



\bibitem{tame-o-minimal}
L. van~den~Dries, {\em Tame topology and o-minimal structures}, London Mathematical Society Lecture Note Series, 248, Cambridge Univ. Press, Cambridge, 1998.

\bibitem{LC2}
L. van~den~Dries and C.~L. Miller, {\em Geometric categories and o-minimal structures}, Duke
Mathematical Journal, 1996.




\bibitem{higher-residue}
M. Herrera, D. Lieberman, {\em Residues and principal values on complex spaces}, Math. Ann. {194} (1971), 259--294.


\bibitem{Kontsevich-DQ}
M. Kontsevich, {\em Deformation quantization of Poisson manifolds, I}, arxiv:math/9709180.

\bibitem{Kontsevich-diagrams}
M. Kontsevich, {\em Feynman diagrams and low-dimensional topology}, First European Congress of Mathematics, Vol. II (Paris, 1992), volume 120 of Progr. Math., pages 97–121. Birkh\"auser, Basel, 1994.

\bibitem{Kontsevich-Operads}
M. Kontsevich, {\em Operads and motives in deformation quantization}, Lett. Math. Phys., 48(1):35–72, 1999. Mosh\'e Flato (1937-1998).


\bibitem{Feynman-Graph-Integral}
M. Wang, {\em Feynman Graph Integrals on $\mathbb{C}^d$}, Communications in Mathematical Physics,
406(5):116, 2025. 


\bibitem{topological-holomorphic}
M. Wang and B. Williams, {\em On the renormalization and quantization of topological-holomorphic field theories}, arXiv:2407.08667[math-ph]. 

 
\bibitem{raviolo-vertex-algebras}
N. Garner, B. R. Williams, {\em Raviolo vertex algebras}, arXiv:2308.04414[math-QA]. 


\bibitem{Griffiths-Harris} P.~A. Griffiths and J.~D. Harris, {\em Principles of algebraic geometry}, Pure and Applied Mathematics, Wiley-Intersci., New York, 1978.

\bibitem{Formality-little-disks}
P. Lambrechts and I. Volic, {\em Formality of the little N-disks operad}, arXiv: 0808.0457[math.AT].




\bibitem{O-minimal-deRham}
R. Bianconi and R. Figueiredo, {\em O-minimal de Rham cohomology}, arxiv:1904.05485[math.LO].



\bibitem{cluckers-miller-2013-lebesgue}
R. Cluckers, D. J. Miller, {\em Lebesgue classes and preparation of real constructible
functions}, Journal of Functional Analysis, 2013.

\bibitem{int-closed}
R. Cluckers, D. J. Miller, {\em Stability under integration of sums of products of real globally subanalytic functions and their logarithms}, Duke Math. J., 156(2):311–348, 2011.

\bibitem{semi-algebraic}
R. Hardt, P. Lambrechts, V. Turchin and I. Voli\'c, {\em Real homotopy theory of semi-algebraic sets}, Algebr. Geom. Topol., 11(5):2477–2545, 2011.
 
\bibitem{Regularized-integral}
S. Li and J. Zhou, {\em Regularized integrals on Riemann surfaces and modular forms}, Comm. Math. Phys. {388} (2021), no.~3, 1403--1474.


\bibitem{derivative-closed}
T. Kaiser, {\em The Derivative of a Constructible Function is Constructible}, arXiv:2508.02517[math.LO].

\bibitem{derivative-closed-2}
T. Kaiser, A. Opris, {\em Differentiability properties of log-analytic functions}, Rocky Mountain J. Math., 52(4):1423–1443, 2022.

\bibitem{Arnold-conf-cohomology}
V.I. Arnold, {\em The cohomology ring of the colored braid group}, Mat. Zametki, 1969, Volume 5, Issue 2, Pages 227–231.

\bibitem{compactification}
W. Fulton and R.~D. MacPherson, {\em A compactification of configuration spaces}, Ann. of Math. (2) 139 (1994), no.~1, 183--225.

\bibitem{Rational-homotopy-theory}
Y. F\'elix, S. Halperin and J.-C. Thomas, {\em Rational homotopy theory}, Graduate Texts in Mathematics, 205, Springer, New York, 2001.

\bibitem{YS-chow1}
Y. Peterzil, S. Starchenko, {\em Complex analytic geometry and analytic-geometric categories}, J. Reine Angew. Math. 626 (2009), 39–74.

\bibitem{YS-chow2}
Y. Peterzil, S. Starchenko, {\em Complex analytic geometry in a nonstandard setting}, Model theory with applications to algebra and analysis. Vol. 1, 117–165, London Math. Soc. Lecture Note Ser., 349, Cambridge Univ. Press, Cambridge, 2008.


\bibitem{Higher-CA-Jouanolou}
Z. Gui, M. Wang and B. R. Williams, {\em Higher-dimensional Chiral Algebras in the Jouanolou Model and free-field realization}, arXiv: 2510.26608 [math.QA].  

\end{biblist}
\end{bibdiv}
\end{document}